\documentclass[11pt,reqno]{amsart}
\usepackage[a4paper,margin=0.85in,footskip=0.35in]{geometry}
\usepackage{amssymb, amsmath, amsthm}
\usepackage{mathrsfs}
\usepackage{mathtools}
\usepackage{tikz-cd}
\usepackage{enumitem}
\usepackage[all,cmtip]{xy}
\usepackage{sseq}
\usepackage{adjustbox}
\usepackage[T2A,T1]{fontenc}
\usepackage{url}
\newcommand*\circled[1]{\tikz[baseline=(char.base)]{
            \node[shape=circle,draw,inner sep=1pt] (char) {#1};}}
\PassOptionsToPackage{pdfusetitle,pagebackref,colorlinks}{hyperref}
\usepackage{bookmark}
\hypersetup{
  linkcolor={red!70!black},
  citecolor={blue!70!black},
  urlcolor={blue!80!black}
}
\usepackage[capitalize]{cleveref}
\makeatletter
  \def\th@plain{
  \thm@headfont{\bfseries} % Heading font is bold
  \thm@notefont{\itshape} % Note is same as heading
  \itshape% Regular text is also italic
}
\makeatother

\makeatletter
  \def\th@definition{
  \thm@headfont{\bfseries} % Heading font is italic
  \thm@notefont{\bfseries} % Note is same as heading
}
\makeatother

\makeatletter
  \def\th@remark{
  \thm@headfont{\bfseries} % Heading font is italic
  \thm@notefont{\bfseries} % Note is same as heading
	}
\makeatother

\newtheorem{theorem}{Theorem}[section]
\newtheorem{lemma}[theorem]{Lemma}
\newtheorem{proposition}[theorem]{Proposition}
\newtheorem{corollary}[theorem]{Corollary}

\newtheorem{alphtheorem}{Theorem}

\theoremstyle{definition}
\newtheorem{definition}[theorem]{Definition}

\newtheorem{example}[theorem]{Example}

\theoremstyle{remark}
\newtheorem{remark}[theorem]{Remark}

\newtheoremstyle{cited}{.5\baselineskip\@plus.2\baselineskip\@minus.2\baselineskip}{.5\baselineskip\@plus.2\baselineskip\@minus.2\baselineskip}{\itshape}{}{\bfseries}{\bfseries .}{5pt plus 1pt minus 1pt}{\thmname{#1}\thmnumber{~#2}\thmnote{ \normalfont#3}}
\theoremstyle{cited}
\newtheorem{citedthm}[theorem]{Theorem}

\newtheorem{citedprop}[theorem]{Proposition}

\newtheoremstyle{citeddef}{.5\baselineskip\@plus.2\baselineskip\@minus.2\baselineskip}{.5\baselineskip\@plus.2\baselineskip\@minus.2\baselineskip}{}{}{\bfseries}{\bfseries .}{5pt plus 1pt minus 1pt}{\thmname{#1}\thmnumber{~#2}\thmnote{ \normalfont#3}}
\theoremstyle{citeddef}

\DeclareSymbolFont{cyrillic}{T2A}{cmr}{m}{n}
\DeclareMathSymbol{\Sha}{\mathalpha}{cyrillic}{216}

\newcommand{\sA}{\mathcal{A}}

\newcommand{\sC}{\mathcal{C}}

\newcommand{\sH}{\mathcal{H}}

\newcommand{\sJ}{\mathcal{J}}

\newcommand{\sM}{\mathcal{M}}

\newcommand{\sO}{\mathcal{O}}
\newcommand{\sP}{\mathcal{P}}
\newcommand{\sQ}{\mathcal{Q}}

\newcommand\sY{{\mathcal Y}}

\newcommand{\bbC}{\mathbb{C}}
\newcommand{\bbD}{\mathbb{D}}

\newcommand{\bbP}{\mathbb{P}}
\newcommand{\bbQ}{\mathbb{Q}}

\newcommand{\bbZ}{\mathbb{Z}}

\newcommand{\into}{\hookrightarrow}

\DeclareMathOperator{\Bl}{Bl}

\DeclareMathOperator{\Coker}{Coker}
 \renewcommand{\div}{\text{div}}

\DeclareMathOperator{\DR}{DR}

\DeclareMathOperator{\image}{Im}
\DeclareMathOperator{\codim}{codim}              % codim
\DeclareMathOperator{\Hom}{Hom}

\DeclareMathOperator{\CH}{CH}

\DeclareMathOperator{\CC}{CC}

\DeclareMathOperator{\im}{Im}

\DeclareMathOperator{\Ker}{Ker}

\DeclareMathOperator{\gr}{gr}

\DeclareMathOperator{\dr}{DR}

\DeclareMathOperator{\Pic}{Pic}

\DeclareMathOperator{\pr}{pr}

\usepackage[all]{xy}
\usepackage{url}
\usepackage{verbatim}
\usepackage[english]{babel}
\usepackage{sseq}

\def\C{\mathbb C}

\def\G{\mathbb G}

\def\P{\mathbb P}

\def\Q{\mathbb Q}

\def\Z{\mathbb Z}
\def\AA{\mathcal A}

\def\CC{\mathcal C}
\def\DD{\mathcal D}

\def\FF{\mathcal F}
\def\HH{\mathcal H}

\def\JJ{\mathcal J}

\def\LL{\mathcal L}
\def\MM{\mathcal M}
\def\NN{\mathcal N}
\def\OO{\mathcal O}
\def\PP{\mathcal P}
\def\QQ{\mathcal Q}

\def\XX{\mathcal X}
\def\YY{\mathcal Y}

\def\qed{{\hfill $\Box$}}

\def\Hom{{\operatorname{Hom}}}
\def\Ker{{\operatorname{Ker}}}
\def\Coker{{\operatorname{Coker}}}
\def\im{{\operatorname{Im}}}
\def\coker{{\operatorname{Coker}}}

\def\div{{\operatorname{div}}}
\def\div{{\operatorname{div}}}

\def\Cl{{\operatorname{Cl}}}

\def\Pic{{\operatorname{Pic}}}
\def\NS{{\operatorname{NS}}}

\def\Bl{{\operatorname{Bl}}}

\def\cl{{\operatorname{cl}}}

\def\Sing{{\operatorname{Sing}}}

\def\codim{{\operatorname{codim}}}

\def\Aut{{\operatorname{Aut}}}

\def\rk{{\operatorname{rk}}}

\def\wt{\widetilde}

\def\ol{\overline}

\def\cHom{\HH om}

\theoremstyle{plain}

\numberwithin{equation}{section}

\newcommand\rmD{\mathrm{D}}

\newcommand\rmH{{\mathrm{H}}}

\newcommand\rmL{\mathrm{L}}

\newcommand\rmR{\mathrm{R}}
\newcommand\rmS{\mathrm{S}}

\newcommand\rmV{\mathrm{V}}

\newcommand\calA{\mathcal{A}}

\newcommand\calC{\mathcal{C}}
\newcommand\calD{\mathcal{D}}

\newcommand\calH{\mathcal{H}}

\newcommand\calJ{\mathcal{J}}

\newcommand\calL{\mathcal{L}}
\newcommand\calM{\mathcal{M}}

\newcommand\calO{\mathcal{O}}
\newcommand\calP{\mathcal{P}}
\newcommand\calQ{\mathcal{Q}}

\newcommand\calX{\mathcal{X}}
\newcommand\calY{\mathcal{Y}}

\DeclareMathOperator{\Br}{Br}

\DeclareMathOperator{\BrAn}{Br_\mathrm{an}}
\DeclareMathOperator{\an}{an}

\DeclareMathOperator{\IC}{IC}

\author{Yajnaseni Dutta}
\address{Mathematisch Instituut, Universiteit Leiden, Gorlaeus Gebouw, Einsteinweg 55, 2333 CA, Leiden, NL}
\email{y.dutta@math.leidenuniv.nl}
\author{Dominique Mattei}
\address{Institut für Algebraische Geometrie, Leibniz Universität Hannover,
Welfengarten 1, 30167 Hannover, DE}
\email{mattei@math.uni-hannover.de}
\author{Evgeny Shinder}
\address{School of Mathematical and Physical Sciences, University of Sheffield, Hounsfield Rd, S3 7RH, Sheffield, UK}
\email{e.shinder@sheffield.ac.uk}

\begin{document}  
\title{Twists of intermediate Jacobian fibrations}

\setcounter{tocdepth}{2}
\begin{abstract}
    We study the sections, Tate--Shafarevich twists, and the period 
    for an OG10 hyperk\"ahler Lagrangian
    associated to a cubic fourfold.
    To do so, we introduce the analytic relative Jacobian sheaf for a Lagrangian fibration of a hyperk\"ahler variety.
    The
    Tate--Shafarevich group parameterizing twists
    is isomorphic to the first cohomology group of this sheaf and we
    compute it in terms of certain analytic Brauer groups associated to the cubic fourfold.
    We prove that the primitive Hodge lattice of the cubic fourfold is, up to a sign, isometric to a distinguished sublattice of the second cohomology group of the associated OG10 hyperk\"ahler manifold.
    Among the main tools we use
    are intersection complexes with integral coefficients,
    Decomposition Theorem,
    Hodge modules
    and Deligne cohomology.
\end{abstract}

\maketitle

\tableofcontents

\section{Introduction}
To a Lagrangian fibration $f\colon M\to B$ of a complex projective hyperk\"ahler manifold $M$ over a smooth base $B$, we can 
associate
the constructible pushforward sheaf
$\Lambda_M\coloneqq \rmR^1f_*\bbZ_M$ and the relative Jacobian sheaf $\sJ_{M/B} \coloneqq (\rmR^1 f_* \OO_M) /\Lambda_M$, considered in the analytic topology. 
Under appropriate assumptions, the two sheaves $\Lambda_M$ and $\JJ_{M/B}$ control the sections, the twists, and the Hodge lattice
of $f\colon M\to B$. 
In this paper, we prove several results in favor of this statement for an OG10 fibration where the fibers are generically given by intermediate Jacobians of cubic threefolds.

To motivate the problem, let us briefly recall how this works for K3 surfaces.
Let 
$f\colon S\to \bbP^1$ be an
elliptic K3 surface 
with a section, which we, for simplicity, assume to have irreducible fibers. 
We denote by $\theta,\eta \in \NS(S)$ the classes
of the section and of the fiber, respectively.
We can consider
 the constructible pushforward sheaf $\Lambda_S = \rmR^1 f_*\Z_S$ and
 the 
 sheaf  $\sJ_{S/\P^1}$
 of local analytic sections of $f$.
These two sheaves are related by the short sequence
\[
0 \to \Lambda_S \to \Omega^1_{\P^1} \to \JJ_{S/\P^1} \to 0.
\]
The sections of $f$ are identified with the sections $\rmH^0(\bbP^1, \sJ_{S/\P^1})$ and by \cite[Theorem 1.3]{Shioda}
\begin{equation}\label{eq:Shioda-intro}
\rmH^0(\bbP^1, \sJ_{S/\P^1}) \simeq \NS(S) / \langle \theta, \eta \rangle.
\end{equation}
The twists of $f$ are parametrized by the 
Tate--Shafarevich group 
$
\Sha(S/\bbP^1)\coloneqq \rmH^1(\bbP^1, \sJ_{S/\P^1}).
$
It classifies all K3 surfaces, not necessarily projective,
that can be constructed as twists of $S$ by regluing the fibers.
An easy spectral sequence 
argument gives an isomorphism
\begin{equation}\label{eq:K3-Sha-intro}
\Sha(S/\bbP^1)\simeq \rmH^2(S, \OO_S^*) =: \Br_{\an}(S) \simeq  \C / \Z^{22-\rk(\NS(S))}.
\end{equation}
Tate--Shafarevich twists
that are projective
are parametrized by the torsion subgroup of $\Br_{\an}(S)$
which is isomorphic to 
$\Br(S) = \rmH^2_{\mathrm{et}}(S, \G_m)$, see e.g. \cite[\S 11]{HuybrechtsK3}.
Finally, for the Hodge lattice, one can show that there is a Hodge isometry (for a natural Hodge  lattice structure on the group on the left in the equation below)
\begin{equation}\label{eq:S-Lambda-intro}
\rmH^1(\P^1, \Lambda_S) \simeq \langle \theta,\eta \rangle^\perp \subset \rmH^2(S, \Z).
\end{equation}

\medskip

All these phenomena seem to be replicated by hyperk\"ahler manifolds with a Lagrangian fibration, which is a higher dimensional analog of elliptic K3 surfaces. 
A hyperk\"ahler manifold $M$ is a simply connected compact K\"ahler manifold that admits  a  nondegenerate complex $2$-form $\sigma$ which generates $\rmH^{0}(M, \Omega^2_M)$. Whenever
a hyperk\"ahler manifold admits a fibration onto a lower dimensional smooth variety $B$ that is not a point, then $B\simeq \bbP^n$ \cite{HwangBase} and the fibration is forced to be Lagrangian \cite{Mats1, Mats2}. 
This means that $\sigma$ restricts trivially to  fibers of $f$. Furthermore, it follows that general fibers of $f$ are  Abelian varieties \cite[Proposition 2.1]{Camisotrivial}. 

In addition to the relative Jacobian
sheaf $\JJ_{M/B}$, another generalization
of the sheaf $\JJ_{S/\P^1}$ for a K3 surface 
is the so-called
relative Albanese sheaf
$\AA_{M/B}$. 
This sheaf naturally acts on $M$ over $B$.
The relative  Albanese sheaf was constructed
by Markushevich \cite{Markushevich} 
and in larger generality by Abasheva--Rogov \cite{AbashevaRogov} 
as a subsheaf of the sheaf of vertical automorphisms of $f$.
When all fibers of $f$ are integral and it admits a section, 
the relative Albanese sheaf
is represented by
the smooth locus of $f$
\cite{AriFed}.
The purely algebraic construction of the Albanese sheaf given by \cite{AriFed}
in the case of integral fibers
was generalized in \cite{Yoonjoo}
for any Lagrangian fibration without multiple fibers.
Our first result is the following comparison between the Jacobian and the Albanese sheaves.

\begin{alphtheorem}\emph{(=Theorem \ref{thm:JMBisAMppHK})}\label{intro:JMBisAMppHK}
     Let $M$ be a projective hyperk\"ahler manifold 
     of dimension $2n$ and $f\colon M\to B=\bbP^n$ be a Lagrangian fibration  with no strictly multiple fibers in codimension one. Assume that there exists a class in $\rmH^2(M, \Q)$ that restricts to a principal polarization on all smooth fibers of $f$.
    Then
    $$\calA_{M/B}\simeq \calJ_{M/B} \coloneqq \dfrac{\rmR^1f_*\sO_M}{\rmR^1f_*\bbZ_M}.$$
\end{alphtheorem}
Roughly speaking, this isomorphism extends the self-duality induced by the principal polarization of the generic fiber. 
The assumption about multiple fibers is quite mild (see Definition \ref{def:strictly-mult}). It is currently unknown if Lagrangian fibrations can ever have multiple fibers in codimension one.
The theorem is proved in \S \ref{sec:compareJJAM} building on previous work 
of Abasheva--Rogov \cite{AbashevaRogov}. 
As one of the main steps of the proof, we establish that the sheaf $\Lambda_M = \rmR^1f_*\bbZ_M$ is isomorphic to its pushforward from any Zariski open subset $U \subset B$, see \cref{thm:R1f-tcf-cor}.
We call this the tcf property (torsion and cotorsion free, \cref{def:tcf}),
and it is used in a crucial way throughout the paper.

\medskip 

The main focus of this paper, where the isomorphism in Theorem \ref{intro:JMBisAMppHK} is used, is to study the Albanese sheaf
$\sA_{M/B}$ of an OG10 Lagrangian
fibration $f\colon M \to B \coloneqq(\P^5)^\vee$ associated to a 
smooth cubic fourfold $X \subset \P^5$ (see \cref{def:HKcompIJ})
constructed as follows.
Let $p\colon \YY \to B $ be the universal hyperplane section of 
the cubic fourfold $X$. 
In \cite{DonMar}, Donagi and Markman considered the relative compactified intermediate Jacobian $f_{U_1}\colon \overline{J}_{U_1}\to U_1$ over the open locus $U_1$ parameterizing hyperplane sections with at worst one nodal singularity. 
Lagrangian compactifications of $f_{U_1}$
were constructed
by
\cite{LSV} for $X$ general, and by \cite{Sacbirational} for any smooth cubic fourfold. See also \cite{Sac25} for a more general construction
and \cite{DutMarq} for an explicit generality assumption under which \cite{LSV} applies.
See \cite{MongardiOnorati}, \cite{MongardiOnorati-erratum} for additional properties of $M$.
The resulting Lagrangian fibration $f\colon M \to B$ is  a hyperk\"ahler variety of OG10 type \cite{LSV}, \cite{KLSV} and it
the only known
example of a hyperk\"ahler Lagrangian fibration with principally polarized smooth fibers which are not Jacobians of curves, but are intermediate Jacobians of cubic threefolds.

To analyze the sheaves $\Lambda_M$ and $\JJ_{M/B}$ we consider the following morphisms:
\begin{equation}\label{eq:setupXYJB}
\xymatrix{
& \YY \ar[dl]_q \ar[dr]^p & & M \ar[dl]_f \\
X & & B & \\
}
\end{equation}

\subsection{The sheaves \texorpdfstring{$\Lambda$}{Λ} and \texorpdfstring{$\calJ$}{J}}

In this set-up, we
let $\Lambda\coloneqq \rmR^3p_*\bbZ_{\sY}$; it is a constructible sheaf of generic rank $10$, the middle Betti number of a smooth cubic threefold.
We also define 
the \textit{relative intermediate Jacobian sheaf} 
\[
\sJ\coloneqq (\rmR^2p_*\Omega_{\sY}^1)/\Lambda \simeq \Omega^1_B / \Lambda
\]
of the universal family of the hyperplane sections of $X$ (see \cref{def:J} and \cref{prop:pushforward-Omega1}). 
We show the following. 
   
\begin{alphtheorem}\label{thm:mainsequences}
Let $f\colon M \to B$ be an
OG10 
Lagrangian fibration associated to a smooth cubic fourfold $X$.
There are canonical isomorphisms
\begin{equation}\label{eq:Lambda-LambdaM-intro}
\Lambda\simeq \Lambda_M\ \text{ and }\ \JJ \simeq \calA_{M/B} \simeq \JJ_{M/B}.
\end{equation}
Moreover the sheaf $\calJ$ admits
a presentation
\begin{equation}\label{eq:liesequence}
    0 \to \JJ \to \widetilde{\JJ} \to \rmR^4 p_*\Z_\YY \to 0
\end{equation}
    with $\widetilde{\JJ} = \rmR^4 p_*\Z_{\sY}(2)_{\rmD}$,
    where
$\Z_\YY(2)_{\rmD}$ is the second Deligne complex $[\Z_{\YY} \overset{(2\pi i)^2}{\longrightarrow}\OO_{\YY} \longrightarrow \Omega^1_{\YY}]$. 
\end{alphtheorem}
Both claims in \cref{thm:mainsequences} extend 
certain well-known results from the smooth locus $U \subset B$ of $p$ to the whole base.
Namely, on the one hand, the isomorphisms
\eqref{eq:Lambda-LambdaM-intro} extend the isomorphism between the middle cohomology of a smooth cubic threefold and the first cohomology of the corresponding principally polarized intermediate Jacobian \cite{ClemensGriffiths}.
On the other hand, the short exact sequence \eqref{eq:liesequence}
extends the corresponding sequence in \cite[\S 3]{Voitwist} 
from the smooth locus $U$; note that for smooth hyperplane sections $Y \subset X$
it is the exact sequence that involves the cycle class map
\[0\to \CH^2_0(Y)\to \CH^2(Y)\overset{\cl}{\longrightarrow} \rmH^4(Y,\bbZ)\to 0.\]

 Theorem \ref{thm:mainsequences} is a combination of Proposition \ref{prop:jaclatticeinjection}, Proposition \ref{prop:JJ-wtJJ}, Corollary \ref{cor:JisAM}. 
 There are three main steps in the proof of Theorem \ref{thm:mainsequences}, which we recall here, not only as an outline of the proof but also as a way to summarize the main new inputs of the paper.
 The results in each step may be of independent interest.
 \begin{enumerate}
     \item \textit{Integral Decomposition Theorem.} 
     We prove an elementary decomposition of $\rmR f_*\Z_{W}$
     for a proper surjective morphism $f$ of smooth varieties, when certain global cycles exist (Proposition \ref{prop:decomp-Z}). Applied to the complex $\rmR p_*\Z_{\sY}$, this induces an integral analogue of the Decomposition Theorem (\cref{thm:decomposition-cubic}).
     We use this to 
     establish the tcf property for $\Lambda$ which implies the first isomorphism in \eqref{eq:Lambda-LambdaM-intro}. 
\item \textit{Leray spectral 
sequence for Deligne cohomology.} We consider
 the Leray spectral sequence for $\rmR p_*\Z_{\sY}(2)_\rmD$ which does not degenerate at the second page, but only at the third page (\cref{thm:main-cohom-wtJJ}). 
    This allows us to relate $\widetilde{\sJ} = \rmR^4 p_*\Z_{\sY}(2)_\rmD$ to the relative intermediate Jacobian sheaf $\sJ$.
    \item  \textit{Jacobian sheaves for Hodge modules.}   We rely on the theory of Hodge modules and their Jacobians \cite{Zuc76, Clem83, Saiadmissible, GreenGriffithsKerr, BrosnanPearlsteinSaito, SchNeron}. 
    This allows us to construct an isomorphism of analytic sheaves $\JJ\simeq \sJ_{M/B}$ (\cref{thm:comparisonJandJM}). 
 \end{enumerate}

 \subsection{The sections, twists and the Hodge lattice}

We compute the Tate--Shafarevich twists for an
OG10 Lagrangian fibration $f\colon M\to B$ associated to a cubic fourfold $X$. 
The Tate--Shafarevich group is defined as
\[
\Sha(M/B) \coloneqq \rmH^1(M, \calA_{M/B}).
\]
The standard regluing procedure \cite{MarkmanLagrangian, AbashevaRogov} allows for every class $\alpha \in \Sha(M/B)$ to construct a new fibration $M_\alpha \to B$ which is locally over $B$ isomorphic to $M$ but has no section when $\alpha \ne 0$.
Under some mild assumptions, these twists are hyperk\"ahler, and for the torsion classes $\alpha$
the twists are projective.

For the next result, for ease of exposition, we will assume that $X$ is sufficiently general, 
namely that all hyperplane sections of $X$ are $\Q$-factorial. We call this property defect general, see Definition \ref{def:defect-general}. Defect general cubic fourfolds have been explicitly characterized by  Marquand--Viktorova \cite[Theorem 1.2]{MarqVikt} as cubic fourfolds containing no planes and no cubic scrolls.
In the case of defect general cubic fourfolds, the sheaf $\rmR^4 p_*\Z_\YY$ is isomorphic to $\Z_B$ (Proposition \ref{prop:cohom-Y}).

For a smooth cubic fourfold $X$,
we write $h \in \Pic(X)$ for the hyperplane class
and $\rmH^4(X, \Z)_{\pr}$ for
the primitive cohomology $(h^2)^\perp \subset \rmH^4(X, \Z)$.
We express the Tate--Shafarevich group in terms of certain analytic
Brauer groups of Hodge structures of K3 type,
see \cite{Huybrechts-Br}, or Appendix \ref{app:Brauer} for their properties:
\begin{align*}
    \Br_{\an}^4(X) &\coloneqq \mathrm{Coker}(\rmH^4(X, \Z) \to \rmH^{1,3}(X)) \\
\Br_{\an, \pr}^4(X) &\coloneqq \mathrm{Coker}(\rmH^4(X, \Z)_{\pr} \to \rmH^{1,3}(X)).
\end{align*}

\begin{alphtheorem} \label{thm:main} 
Assume that $X$ is a defect general cubic fourfold and $f\colon M \to B$ is a 
Lagrangian fibration as in Theorem \ref{thm:mainsequences}.
We have the following isomorphisms
\[
\begin{aligned}
\rmH^0(B, \widetilde{\JJ}) & \simeq \rmH^{2,2}(X, \Z) \quad &
\rmH^1(B, \widetilde{\JJ}) & \simeq \Br_{\an}^4(X)\\
\rmH^0(B,\JJ) & \simeq \rmH^{2,2}(X, \Z)_{\pr}\quad  &
\rmH^1(B, \JJ) & \simeq \Sha(M/B) \simeq
  \Br_{\an, \pr}^4(X)\\
\end{aligned}
\]
and a short exact sequence
\begin{equation}\label{eq:oggshaf}
    \Z/3\Z \to \Sha(M/B)\to \Br_{\an}^4(X)\to 0.
\end{equation}
    The first map in the sequence (\ref{eq:oggshaf})
    is injective if and only if there is no class $\xi \in \rmH^{2,2}(X, \Z)$ with the property $\xi \cdot h^2 = 1$.
\end{alphtheorem}

The computation of $\rmH^0(B, \JJ)$ has already been done by 
Sacc{\`a} \cite{Sacbirational}
using a different method but the other results are new.
The cohomology groups of $\calJ$ and $\wt{\calJ}$ are computed in \S\ref{sec:JandJtilde}. 
The isomorphism $\Sha(M/B)\simeq \rmH^1(B,\calJ)$ is a consequence of Corollary \ref{cor:JisAM}. 
\cref{thm:main}
should be understood as a replacement for 
\eqref{eq:Shioda-intro}
and 
\eqref{eq:K3-Sha-intro} in higher dimensions. If $X$ is an arbitrary smooth cubic fourfold, then we get analogous results with $\rmH^{2,2}(X,\Z)_{\pr}$ replaced by  $\Sigma^{\perp}\subset \rmH^{2,2}(X, \Z)$, where $ \Sigma$ is the Hodge sublattice generated by algebraic classes contained in the hyperplane sections of $X$, see  \cref{thm:Sha-seq-nonDG}.

Finally, we explain the relationship
between
 the Hodge lattices $\rmH^4(X, \Z)$ and $\rmH^2(M, \Z)$.
 This is given using the constructible sheaf $\Lambda$, which for $X$ defect general underlies a pure Hodge module.
We prove the following fundamental Torelli type  statement.

\begin{alphtheorem}\emph{(=Theorem \ref{thm:pairings})}\label{thm:pairings-intro}
Let $X$ be a smooth cubic fourfold
and $f\colon M \to B$ be an associated OG10 Lagrangian fibration.
Let $\theta \in \rmH^2(M, \Z)$
be the class of the principal polarization
and $\eta \coloneqq f^*H \in \rmH^2(M, \Z)$ be
the  isotropic class.
We have
an isomorphism of Hodge lattices 
\[
\rmH^4(X, \Z)_{\pr}(1) \simeq  \rmH^1(B,\Lambda_M)
\simeq
\langle \theta, \eta\rangle^\perp \subset \rmH^2(M,\bbZ).
\]
In particular, $f\colon M \to B$ uniquely determines $X$.
\end{alphtheorem}

The isomorphisms above
were only known after tensoring with $\Q$ \cite{LSV}.
The proof of this result uses the ingredients that have been already mentioned above such as 
integral constructible sheaves, the tcf property, Hodge modules
and the Decomposition Theorem.
\cref{thm:pairings-intro} can be interpreted as a higher-dimensional analog of \eqref{eq:S-Lambda-intro}.

\newpage

\subsection{Comparison to other works}

\subsubsection*{Beauville-Mukai systems} 
The situation that we described  for the OG10 Lagrangian fibrations is analogous to the case of Beauville--Mukai systems, that is compactified Jacobians of complete families of curves on K3 surfaces, considered in \cite{MarkmanLagrangian}. These are K3$^{[n]}$-type hyperk\"ahler manifolds, i.e., they are deformation equivalent to the Hilbert scheme of  points on K3 surfaces. 
The geometry of these moduli spaces is encoded in the  diagram analogous to \eqref{eq:setupXYJB}:
\begin{equation}\label{eq:setupSM}
\xymatrix{
& \CC \ar[dl]_q \ar[dr]^p & & M \ar[dl]_f \\
S & & B & \\
}
\end{equation}
Here $S$ is a K3 surface, $\CC$ is a family of curves in a fixed linear system $B$ on $S$ and $M$ is a moduli space of torsion sheaves on $M$ which is also a compactified Jacobian for $\CC$ over $B$.
In this situation, the role of the sheaf $\JJ$ (resp.\ $\wt{\JJ}$) from \cref{thm:mainsequences} is played by the $\Pic^0$-sheaf (resp.\ $\Pic$-sheaf) for the curve $\calC\to B$, and indeed the first Deligne complex $\Z_\YY(1)_\rmD := [\Z_{\sC} \overset{2\pi i}{\longrightarrow} \OO_{\sC}]$ is quasi-isomorphic to $\OO_{\sC}^*[-1]$. The corresponding Tate--Shafarevich group was defined and computed
by Markman \cite{MarkmanLagrangian}
and with a different approach in \cite{HuyMat}. See Appendix \ref{app:comparisonBM} for more details on the
analogy between Beauville--Mukai systems and OG10 Lagrangian fibrations.

\subsubsection*{Moduli space of twisted sheaves}

In a recent work by Bottini \cite{bottini} it has been shown that $M$ can be represented by a moduli space of sheaves on $F(X)$, the Fano variety of lines on $X$. We expect the Tate--Shafarevich twists of $M$ to be isomorphic to the corresponding moduli spaces of twisted sheaves.
This is analogous to the Beauville--Mukai system
on a K3 surface $S$ 
where $\Br(S)$ and related groups appear when parameterizing  Tate--Shafarevich twists \cite{MarkmanLagrangian}, see Proposition \ref{prop:5termseqBM}.

The  group $\Br_{\an}^4(X)$ 
is isomorphic to the Brauer group $\Br_{\an}^2(F(X))$ of the Fano variety of lines of $X$, and we prove in Corollary \ref{cor:Br-M} that $\Br_{\an, \pr}^4(X)$ is isomorphic to $\Br^2_{\an}(M)$. In particular, the sequence (\ref{eq:oggshaf}) can be rewritten as
$\bbZ/3\bbZ \to \Br_{\an}^2(M) \to \Br_{\an}^2(F(X)) \to 0$
which is an analog of \cite[Theorem 1.2]{MatteiMeinsma} for  Beauville--Mukai systems.

\subsubsection*{Locus of OG10 Lagrangian fibrations}
For a fixed  defect general
cubic $X$ with $\rho = \rk \, \rmH^{2,2}(X, \Z)$ the periods of the twists given by $\Sha(M/B) \simeq \C / \Z^{24 - \rho}$
parameterize in the 
moduli space of OG10 hyperk\"ahler manifolds 
the so-called degenerate twistor line 
\cite{VerbitskyDegenerateTwistor, AbashevaRogov, VerbitskySoldatenkovKaehlertwists}.
In particular, varying both the cubic $X$ and the twist we obtain a divisor in the moduli space of OG10 manifolds. 
This is analogous to Markman's description of Lagrangian fibrations of K$3^{[n]}$ type \cite{MarkmanLagrangian}
as Tate--Shafarevich twists of Beauville--Mukai systems.

The twist $M^T$ 
corresponding to a generator of $\Z/3\Z$ in 
\eqref{eq:oggshaf} is the one previously considered by Voisin \cite{Voitwist}, see \S\ref{sec:comparisonVoisin}. It was  proved by 
Sacc\`a \cite{Sacbirational} that $M$ and $M^T$ are not birational for general $X$. The twist $M^T$ has been identified as a moduli space of stable objects in the Kuznetsov component $\AA_X$ of $X$
\cite{LPZ20} (this is known to be not the case for $M$ \cite{Sacbirational}),
and its period has been computed in
\cite{Gx2Onorati}.

\subsubsection*{Decomposition Theorem for Lagrangian fibrations}

Perverse sheaves and the Decomposition Theorem have been used extensively in studying Lagrangian fibrations starting from the pioneering work of Ng{\^o} \cite{Ngo}. 
For the OG10 Lagrangian fibration $f\colon M \to B$
the Decomposition Theorem has been used in
\cite{dCRS_HodgeNumberOG10} 
to compute the Hodge numbers of $M$
and in \cite{ACLSRelativeLefschetz} to verify the standard Lefschetz conjecture.
Furthermore, the Hodge module Decomposition Theorem has quite recently also become prominent in studying Lagrangian fibrations \cite{ShenYin-PW, ShenYin, MSY23, schHK}. The main difference in our work is that we need to keep track of the integral structure on the corresponding first nontrivial perverse pushforward sheaf (or Hodge module). This, in particular, allows us to consider the cohomology group $\rmH^1(B, \Lambda)$ as an integral Hodge lattice.

\subsection{Organization of the paper}
In \S \ref{sec:prelim} we recall some well-known results on constructible sheaves with integer coefficients, and introduce the tcf property. 
In \S \ref{sec:cohomologycomputations}
the intermediate Jacobian sheaf $\sJ$ and its extension via Deligne cohomology $\widetilde{\sJ}$ are introduced and studied, in particular, we compute cohomology the groups from \cref{thm:main}.
In \S \ref{sec:latticesHK} we prove \cref{thm:pairings-intro} 
and in
\S \ref{sec:HM} we study Hodge module theoretic Jacobian sheaves.  
In \S \ref{sec:twist}, we prove \cref{intro:JMBisAMppHK}, compute Tate--Shafarevich groups, consider Voisin's twist \cite{Voitwist} and prove the isomorphisms in \cref{thm:mainsequences}.
In the appendices, we compare our results to those of Markman for  Beauville--Mukai systems and discuss Huybrechts' notion of Brauer group for a Hodge structure of K3 type.

\subsection*{Notation and conventions} We work with algebraic varieties over the field $\C$ often considered as complex manifolds. The sheaves and cohomology groups that we work with are considered in the analytic topology.

\subsection*{Acknowledgment}
We thank
Nick Addington,
Marc Andrea de Cataldo,
Daniel Huybrechts,
Yoonjoo Kim,
Janos Koll\'ar,
Andreas Krug,
Radu Laza,
Eyal Markman,
Gebhard Martin,
Lisa Marquand,
Keiji Oguiso,
Mihnea Popa,
Yuri Prokhorov,
Giulia Sacc{\`a},
Justin Sawon,
Andrey Soldatenkov,
Philip Tosteson,
Sasha Viktorova,
Claire Voisin
for various inputs and discussions related to this work.
We would like to thank the Junior Trimester Program `Algebraic Geometry', HIM Bonn and the University of Leiden for their hospitality where the project has been discussed at several stages of preparations. We gratefully acknowledge the financial support and the
excellent working conditions 
at the Complex Algebraic Geometry group of the University of Bonn where
this project started.

E.S. is currently supported by the UKRI Horizon Europe guarantee award `Motivic invariants and birational geometry of simple normal crossing degenerations' EP/Z000955/1
and was supported by the ERC HyperK grant at the start of the project.

\section{Preliminaries}\label{sec:prelim}

In this section, we recall some relevant facts and definitions regarding constructible sheaves and intersection complexes, allowing integral coefficients. We introduce the tcf property of constructible sheaves (Definition \ref{def:tcf}) which plays the key role in our work. 
The main results are Proposition \ref{prop:decomp-Z} and Corollary \ref{thm:R1f-tcf-cor}.
We also recall some basic facts about Deligne cohomology.

\subsection{Constructible sheaves and decomposition theorems}\label{sec:tcf}

We work with constructible sheaves 
on a smooth variety $B$
with coefficients in a regular commutative Noetherian ring $R$ of finite Krull dimension which for us will be either $\Z$ or a field.
For every local system $L$ on a
smooth open subset 
$U \subset Z$ of a closed irreducible
subvariety $Z \subset B$
there is a canonically defined
constructible complex, called the \textit{Deligne--Goresky--MacPherson intersection complex} $\IC(L)$ supported on $Z$ \cite{GoreskyMacphersonII, Dimca-sheaves}.
For example, when $Z$ is smooth and $L=R_U$ is a constant sheaf on an open subset $U \subset Z$, we have $\IC(L)\simeq R_Z[\dim Z]$. More generally we have the following observation.

\begin{lemma}\label{lem:IC-claim}
If $L$ is a local system on a smooth open subset $j \colon U \into Z$, then 
\begin{equation}\label{eq:IC-bottom-coh}
\HH^{-\dim Z} (\IC(L)) \simeq \rmR^0 j_*L, \quad
\HH^{k} (\IC(L)) = 0 \text{ for 
$k < -\dim Z$.}
\end{equation}
\end{lemma}

\begin{proof}
This follows immediately from Deligne's description of the intersection complex \cite[p. 101]{GoreskyMacphersonII}.
\end{proof}

Recall that simple perverse sheaves
with rational coefficients
are precisely the objects of the form $\IC(L)$
for a simple $\Q$-local system
$L$ supported on a locally closed smooth subset of $B$. We will also need the Decomposition Theorem of Beilinson--Bernstein--Deligne--Gabber \cite{BBDG} and more generally of Saito \cite{Sai88} in the following setting.

\begin{citedthm}[(Decomposition Theorem)]\label{thm:decomposition-perverse}
Let $f \colon W \to B$ be a projective
morphism
with smooth $W$ and $B$.
%with maximal fiber dimension $d_f$. 
Let $L$ be a $\Q$-local system on $W$ that underlies a polarized variation of Hodge structure (e.g.\ $L = \Q_W$).
Then we have a quasi-isomorphism
\begin{equation}\label{eq:decomp-thm-perv}
\rmR f_*L[\dim W]\simeq \bigoplus_{i \in \Z} P_i[-i].
\end{equation}
The perverse sheaves $P_i$ further decompose into
a direct sum of complexes $\IC(L)$ for various closed subvarieties $Z\subseteq B$ and simple $\Q$-local systems $L$ supported on open subsets $U\subset Z$.
Furthermore, we have the relative Hard Lefschetz
isomorphism $P_{-i}\simeq P_i$.
\end{citedthm}

\begin{remark}\label{rmk:suppPi}
If the morphism $f \colon W \to B$ in \cref{thm:decomposition-perverse} is, in addition, flat of relative dimension $d$, 
applying proper
base change to \eqref{eq:decomp-thm-perv} for points $i_b\colon b \into B$
and using Lemma \ref{lem:IC-claim}
one concludes that $P_d\simeq \IC(\rmR^d f_{U*} \Q_{W_U})$ where $U\subset B$ is the locus over which the fibers of $f$ are smooth and $f_U$ is the restriction of $f$ to $W_U = f^{-1}(U)$. 
Using $P_{-d}\simeq P_d$, we obtain
\[
P_{-d} \simeq \IC(\rmR^0 f_{U*} \Q_{W_U}).
\]
Moreover, under the flatness assumption, the argument also implies $P_{i} = 0$ for all $|i|>d$.
\end{remark}

Usually, the Decomposition Theorem
does not work with integer coefficients.
We will use the following partial
replacement. 
While considering the push forward of $\Z_W$, unlike in the Decomposition Theorem, we prefer to work with $\Z_W[0]$, as opposed to $\Z_W[\dim W]$.
We write $\bbD(-)$ for the
Verdier duality functor.

\begin{proposition}\label{prop:decomp-Z}
Let $f \colon W \to B$ be a proper
morphism with generic fiber of dimension $d$ 
with smooth $W$ and $B$.
Assume that we have classes
$\xi_i \in \rmH^{2i}(W, \Z)$, $0 \le i \le d$  such that
\[
\deg_W(\xi_i \cdot \xi_{d-i} \cdot [W_b]) = 
\deg_{W_b}(\xi_i|_{W_b} \cdot \xi_{d-i}|_{W_b}) = 
1
\]
where $W_b$ is a smooth fiber of $f$.
Then $\rmR f_*\Z_W$ admits a canonical decomposition
\begin{equation}\label{eq:Lambda-direct-summand}
\rmR f_*\Z_W = \bigoplus_{i = 0}^d \Z_B[-2i] \oplus \Lambda^\bullet[-d]
\end{equation}
for some constructible
complex $\Lambda^\bullet$
which satisfies $\bbD(\Lambda^\bullet) \simeq \Lambda^\bullet[2\dim B]$.
\end{proposition}

\begin{proof}
We have canonical isomorphisms
\[
\Hom(\Z_B[-2i], \rmR f_*\Z_W) = \rmH^0(B, \rmR f_*\Z_W[2i]) = 
\rmH^{2i}(B, \rmR f_*\Z_W) = \rmH^{2i}(W, \Z).
\]
For every $i$, let $\gamma_i \in \Hom(\Z_B[-2i], 
\rmR f_*\Z_W)$ be the morphism corresponding to the class $\xi_i$.
We put the morphisms $\gamma_i$ together as follows
\[
\Gamma \coloneqq (\gamma_0, \ldots, \gamma_d) \in \Hom\left(\bigoplus_{i = 0}^d \Z_B[-2i], \rmR f_*\Z_W\right)
\]
and we set
\[
\Lambda^\bullet = \mathrm{cone}\left(\bigoplus_{i = 0}^d \Z_B[-2i] \xrightarrow{\Gamma} \rmR f_*\Z_W\right)[d].
\]

We will show that $\Lambda^\bullet[-d]$ is a direct summand
of $\rmR f_*\Z_W$.
Consider the Verdier dual morphisms
\[
\gamma_i^{\vee} \in \Hom(\bbD(\rmR f_*\Z_W), \bbD(\Z_B[-2i])) 
= \Hom(\rmR f_*\Z_W, \Z_B[-2d+2i]) 
\]
\[
\Gamma^{\vee} \coloneqq (\gamma_d^{\vee}, \ldots, \gamma_0^{\vee}) \in 
\Hom\left(\rmR f_*\Z_W, \bigoplus_{j = 0}^d \Z_B[-2j]\right).
\]

We now consider the composition
\[
\gamma_j^{\vee} \circ \gamma_i \in \Hom(\Z_B[-2i], \Z_B[-2d+2j]) = \rmH^{2i+2j-2d}(B, \Z).
\]
If $i + j < d$, then this morphism is zero by degree reasons.
On the other hand, to compute $\gamma_{d-i}^{\vee} \circ \gamma_i \in \rmH^0(B, \Z) = \Z$ 
we can make a base change to a general point $b \in B$ to obtain $\gamma_{d-i}^{\vee} \circ \gamma_i = 1$
by the assumption on the classes $\xi_0, \ldots, \xi_d$.

Thus the composition $\Gamma^{\vee} \circ \Gamma$ is a lower triangular matrix with $1$ on the diagonal,
hence it is an isomorphism.
This means that $\Gamma$ is an embedding of a direct summand and we obtain \eqref{eq:Lambda-direct-summand}.
Finally we note that $\Lambda^\bullet$
is Verdier self-dual up to shift by construction,
because up to an appropriate shift $\Lambda^\bullet$ is isomorphic to the cone of $\Gamma$ and to the cone of $\Gamma^{\vee}$.
\end{proof}

\begin{remark}\label{rmk:splittingcohomYY}
    When $d$ is odd, we see from the proof above that the choice of classes $\xi_0,\dots,\xi_{\frac{d-1}{2}}$ induces maps
    \begin{equation}\label{eq:splittingRf}
        \bbZ_B[0]\oplus \cdots \oplus \bbZ_B[- (d-1)] 
        \xrightarrow{\gamma_0\oplus\dots\oplus\gamma_{ \frac{d-1}{2}}} \rmR f_*\bbZ_W \xrightarrow{\gamma^\vee_{ \frac{d-1}{2}}\oplus\dots\oplus\gamma^\vee_0}
        \bbZ_B[- (d+1)] \oplus \cdots \oplus \bbZ_B[-2d],
    \end{equation}
    whose composition is trivial. For any integer $k$, the maps in cohomology induced by (\ref{eq:splittingRf}) are given by:
    \begin{equation}\label{eq:complexcohLambdabullet}
                \begin{tikzcd}[row sep=0.05]
        \bigoplus\limits_{i=0}^{\frac{d-1}{2}} \rmH^{k-2i}(B,\bbZ) \ar[r] & \rmH^k(W,\bbZ) \ar[r] & \bigoplus\limits_{j=\frac{d+1}{2}}^{d} \rmH^{k-2j}(B,\bbZ) \\
        \sum_i \beta_{k-2i} \ar[r,mapsto] & \sum_i f^*(\beta_{k-2i})\cdot \xi_i & \\
        & \alpha \ar[r,mapsto] & \sum_j f_*(\alpha\cdot \xi_{d-j}) 
    \end{tikzcd}
    \end{equation}

    The existence of classes $\xi_{\frac{d+1}{2}},\dots,\xi_{d}$ with $\deg(\xi_i\cdot \xi_{d-i}\cdot [W_b])=1$ ensures that the maps in (\ref{eq:splittingRf}) split off direct summands. For any integer $k$ the cohomology group $\rmH^{k-d}(B,\Lambda^\bullet)$ can be computed as the cohomology of the complex (\ref{eq:complexcohLambdabullet}).
    When the classes $\xi_0,\dots,\xi_{\frac{d-1}{2}}$ are algebraic (which will be the case in our applications), this gives the group $\rmH^{*}(B,\Lambda^\bullet)$ an integral Hodge structure.
\end{remark}

In some situations $\Lambda^\bullet$
will be an intersection complex.
We give a simple example below but
the case of main interest for us is the universal hyperplane section of a cubic fourfold in Theorem \ref{thm:decomposition-cubic}.

\begin{example}\label{ex:intdecompcurve}
Let $\rho\colon \CC \to B$
be a flat projective morphism of relative dimension $1$ with connected fibers. Assume that $\CC$ and $B$ are smooth.
Assume that there exists a class $\xi \in \rmH^2(\CC, \Z)$ which restricts as a generator of the second cohomology group on
the smooth fibers of $f$.
Then we can apply Proposition
\ref{prop:decomp-Z}
to the classes $\xi_0 = 1\in \rmH^0(\sC,\bbZ)$ and $\xi_1 = \xi\in \rmH^2(\sC,\bbZ)$
 to obtain
\[
\rmR\rho_*\Z_\CC = \Z_B \oplus \Lambda^\bullet[-1] \oplus \Z_B[-2].
\]
The complex $\Lambda^\bullet$ has  cohomology sheaves only in degrees $0$ and $1$.
The sheaf $\HH^1(\Lambda^\bullet)$
is supported on the set of points
of $B$ parameterizing  reducible fibers.

Assume that smooth fibers have positive genus so that the support of $\HH^0(\Lambda^\bullet)$ is dense in $B$.
Then using the axiomatic description of the intersection complex \cite[p. 107]{GoreskyMacphersonII}
it is easy to see that
$\Lambda^\bullet$
is an intersection complex
if and only if the locus of reducible
fibers of $\rho$ is contained in a subset
of codimension  two in $B$. \end{example}

\subsection{The tcf property}

Let $F$ be a  sheaf of abelian groups on a smooth connected
variety $B$. Given an 
open subvariety $j\colon U \into B$ it will be important for us to know
when $F$ extends itself from the open set $U$, i.e., it is
isomorphic to $\rmR^0 j_*F|_U$. 

\begin{lemma}\label{lem:tcf-defined}
Let $Z \subset B$ be a proper Zariski closed subset
with open complement $U = B \setminus Z$. 
Let $i\colon Z \into B$
and $j\colon U \into B$ be the embedding morphisms.
For a sheaf $F$ of abelian groups on $B$
the following conditions are equivalent:
\begin{enumerate}
    \item[(a)] The unit of adjunction $F \to \rmR^0 j_* F|_U$ is an isomorphism
    \item[(b)] $\HH^0 i^!F = \HH^1 i^!F = 0$

\end{enumerate}
\end{lemma}

\begin{proof}\phantom{\qedhere}
Equivalence between (a) and (b) follows from
the standard exact sequence
\[
0 \to i_* \HH^0 i^! F \to F \to \rmR^0 j_* F|_U
\to i_* \HH^1 i^! F \to 0. \tag*{\qed}
\]
\end{proof}

Since 
$\HH^0 i^!F$ is the torsion 
of $F$ with respect to $Z$,
we define
the \emph{cotorsion} $F$ with respect to  $Z$
to be the sheaf
$\HH^1 i^!F$.

\begin{definition}\label{def:tcf}
We say that $F$ is a \emph{tcf (torsion and cotorsion free) sheaf with respect to $Z$}
if the equivalent conditions of Lemma \ref{lem:tcf-defined}
are satisfied.
We say that $F$ is \emph{tcf everywhere on $B$}
if $F$ is tcf for every proper closed subset $Z \subset B$.
\end{definition}

If $Z$ is fixed or clear from context we just say that $F$ is tcf.
It follows immediately from the definition that tcf is a local property in the analytic topology on $B$.

\begin{example}\label{lem:tcf-local-system}
    Let $L$ be a  local system on a smooth variety $B$.
    Then it is tcf with respect to any $Z \subset B$.
Indeed, for every complex ball $U_\varepsilon \subset B$ and
$R$-module $M$ we have
the identification 
\[
M = \rmH^0(U_\varepsilon, M) = \rmH^0(U_\varepsilon \setminus Z, M),
\]
which
implies that $M$ satisfies condition (a) of \cref{lem:tcf-defined}.
\end{example}

It is clear that if $F$ is tcf with respect to $Z$, then it is so with respect to any closed subset of $Z$.
Furthermore, we also have the 
following 
induction properties. 

\begin{lemma}\label{lem:tcf-support}
Let $F$ be a sheaf of abelian groups on a smooth variety $B$.

(i) Let $T \subset Z \subset B$ be Zariski closed subsets.
Assume that $F$ is tcf with respect to $T$
and $F|_{B \setminus T}$ is tcf with respect to $Z \setminus T$.
Then $F$ is tcf with respect to $Z$.

(ii) If \ $T \subset B$ is a Zariski closed subset such
that $F$ is tcf with respect to $T$
and $F|_{B \setminus T}$ is tcf everywhere on 
$B \setminus T$, then $F$ is tcf everywhere on $B$.
\end{lemma}
\begin{proof}
(i) The easiest way to prove is this is to use the property (a) of Lemma \ref{lem:tcf-defined}.
Writing $j_1\colon B \setminus Z \into B \setminus T$, $j_2 \colon B \setminus T \into B$
and $j\colon B \setminus Z \to B$
for the open embeddings we have
$\rmR^0 j_* = \rmR^0 j_{2*} \rmR^0 j_{1*}$ and the result follows.

(ii) Let $Z \subset B$ be a proper closed subset.
Since $F$ is tcf with respect to $T$ it is also tcf
with respect to $Z \cap T$.
Therefore, we can apply (i) to 
$Z \cap T \subset Z \subset B$
to deduce that $F$ is tcf with respect to $Z$.
\end{proof}

\begin{proposition}\label{prop:IC-bottom}
If $L$ is a local system on $U \subset B$ then
$\HH^{-\dim B} (\IC(L))$
is tcf everywhere on $B$.
\end{proposition}

\begin{proof}
Since $L$ is a local system, 
by Example \ref{lem:tcf-local-system}, we know that $L$ is tcf
everywhere on $U$.
From \eqref{eq:IC-bottom-coh} we see that
$\HH^{-\dim B} (\IC(L))\simeq \rmR^0j_*L$ and the latter is by definition tcf with respect to $B \setminus U$.
Thus, we deduce from Lemma \ref{lem:tcf-support}
that $\rmR^0j_*L$ is tcf everywhere on $B$.
\end{proof}

We now come to the main result of this section, that under certain conditions on the fibers the first pushforward of the integral constant sheaf is tcf.

\begin{definition}\label{def:strictly-mult}
Let $f \colon W \to B$ be a proper morphism.
We say that a scheme fiber $W_b = f^{-1}(b)$ is \emph{strictly multiple} if $W_b$ is nowhere reduced and
for the generic multiplicities $d_1, \ldots, d_m$ of irreducible components of $W_b$ we have $\gcd(d_1, \ldots, d_m) > 1$.

We say that
$f$ has no strictly 
multiple fibers in codimension one
if the set of points
 of $B$ with strictly multiple fibers
 is contained in a Zariski closed subset
 of codimension at least two.
\end{definition}

\begin{theorem}\label{thm:R1f-tcf}
Assume that $f \colon W \to B$ is a flat proper surjective morphism with connected fibers and with smooth $W$ and $B$.
Then the following statements hold
for $F = \rmR^1 f_* \Z_W$.
\begin{enumerate}
    \item $F$ is tcf with respect to any closed subset $Z \subset B$ of codimension at least two.
    \item If $Z \subset B$ is an irreducible divisor such that  fibers of $f$ over general points in $Z$ are not strictly multiple, then $F$ is tcf with respect to $Z$.
 \end{enumerate}
\end{theorem}

\noindent In particular, we have the following.

\begin{corollary}\label{thm:R1f-tcf-cor}
    Let $f\colon W\to B$ as in \cref{thm:R1f-tcf}. If $f$ has no strictly multiple fibers in codimension one, 
 then $\rmR^1f_*\bbZ_W$ is tcf everywhere on $B$. In particular, let $j\colon U\into B$ be the open locus over which the fibers of $f$ are smooth, then $\rmR^0j_*\left(\rmR^1f_{U*}\bbZ_{W_U}\right)\simeq \rmR^1f_*\bbZ_W$
\end{corollary}
\begin{proof}
The sheaf $\rmR^1f_*\bbZ_W$
is tcf everywhere
since by Lemma \ref{lem:tcf-support} and the assumption on $f$,
to 
check the tcf property
it suffices to consider irreducible subvarieties of type (1) or (2) in Theorem \ref{thm:R1f-tcf}.~\qedhere
\end{proof}

While the proof of Theorem \ref{thm:R1f-tcf} requires some work, the similar statement for $F_{\bbQ}$ follows almost immediately from the Decomposition Theorem \ref{thm:decomposition-perverse}.

\begin{lemma}\label{lem:Q-R1-tcf}
Let $f \colon W \to B$ be a flat projective surjective morphism between smooth varieties. Then
the $\Q$-constructible sheaf $\rmR^1 f_*\Q_W$ is a tcf sheaf.    
\end{lemma}

\begin{proof}
Let $n = \dim B$ and $d$ be the dimension of the fibers.
Let $U\subset B$ denote the locus where $f$ is smooth. We consider the  decomposition of $\rmR f_*\bbQ_W[n+d]$ as in Theorem \ref{thm:decomposition-perverse}. 
The sheaves $\calH^j(P_i)$ vanish when $|i|>d$ by Remark \ref{rmk:suppPi} and also for  $j<-n$ by Lemma \ref{lem:IC-claim}. Therefore, we obtain
\begin{equation}\label{eq:R1fQ}
    \rmR^1 f_*\Q_W \simeq
\bigoplus_{i=-d}^d \HH^{-n-d -i+1}(P_i) = 
\HH^{-n+1}(P_{-d})\oplus\HH^{-n}(P_{-d+1})
\end{equation}
We claim that $\HH^{-n+1}(P_{-d}) = 0$. Indeed, since $P_{-d}\simeq P_d$, we have
 \[
 0=\rmR^{2d+1}f_*\bbQ_B\simeq \bigoplus_{i=-d}^{d-1}\sH^{-n-i+d+1} (P_i) \oplus \sH^{-n+1}(P_d).
 \]
The remaining term $\HH^{-n}(P_{-d+1})$ in the right hand side of \eqref{eq:R1fQ} is the lowest cohomology sheaf of an intersection complex with dense support. 
This is because, by \cref{lem:IC-claim} intersection complexes with proper support have trivial cohomology sheaves in degree $-n$. The result follows by \cref{prop:IC-bottom}.
\end{proof}

To prove Theorem \ref{thm:R1f-tcf}, we need the following 
result which will help us  check the tcf property.

\begin{lemma}\label{lem:cohom-dualizing}
For any closed
embedding $i\colon Z \into B$ with smooth $B$ the complex
$i^!R_B$ sits in cohomological degrees $\ge 2\codim(Z)$.
\end{lemma}
\begin{proof}
Since $B$ is smooth we have
$i^!R_B \simeq 
i^!\DD^\bullet_B[-2\dim B] \simeq 
\DD^\bullet_Z[-2\dim B]$, 
where $\DD^\bullet_B$ and
$\DD^{\bullet}_Z$ are
the dualizing complexes of $B$ and $Z$ respectively. 
We have
\cite[p. 91]{GoreskyMacphersonII}
\begin{equation}\label{eq:DD-vanishing}
 \HH^j(\DD^\bullet_Z) = 0
 \text{ for } j < -2 \dim Z
\end{equation}
so that $i^! R_B$ is acyclic in degrees below $2(\dim B - \dim Z)$ which implies the result.
\end{proof}

\begin{proof}[Proof of Theorem \ref{thm:R1f-tcf}]
For both statements we use the spectral sequence
\begin{equation}\label{eq:ss-Hi}
E_2^{p,q} = \HH^p i^! \rmR^q f_*\Z_W \Longrightarrow \HH^{p+q}i^! \rmR f_* \Z_W
\end{equation}
The terms of the spectral sequence that
we need to control
are
\begin{equation}\label{eq:HiRf-ss-terms}
E_2^{0,1} = \HH^0 i^! \rmR^1 f_*\Z_W
\text{ and }
E_2^{1,1} = \HH^1 i^! \rmR^1 f_*\Z_W.
\end{equation}
We will make use of the fact that by base change we have 
$
i^! \rmR f_* \Z_W \simeq \rmR f'_* i'^! \Z_W 
$
where $i'\colon f^{-1}(Z) \to W$ and $f'\colon f^{-1}(Z) \to Z$ are the natural morphisms.
Note also that we
have $\rmR^0 f_*\Z_W = \Z_B$ because $f$ has connected fibers
hence
\[
E_2^{p,0} = \HH^p i^!\Z_B = 0 \text{ for $p < 2\codim(Z)$ }
\]
by Lemma \ref{lem:cohom-dualizing}.

(1) 
In this case
 $E^{p,0}_2$ vanish
for $p \le 3$, which implies
that the terms
\eqref{eq:HiRf-ss-terms}
are not affected by the differentials of the spectral sequence, i.e., $E_2^{0,1}=E_{\infty}^{0,1}$, and similarly, $E_2^{1,1}=E_{\infty}^{1,1}$.

Using Lemma \ref{lem:cohom-dualizing} again, 
we see that $\sH^{p+q}i'^! \Z_W = 0$ for all $p+q<2\codim(Z)$. Since $f'_*$ is left exact, we have
\[
\HH^{p+q}i^! \rmR f_* \Z_W = \rmR^{p+q}f'_* i'^! \Z_W  = 0 \text{ for $p + q < 2 \codim(Z)$. }
\]
By the $\codim(Z) \ge 2$ assumption
we have, in particular, that
$E_\infty^{0,1} = 
E_\infty^{1,1} = 0$.
This implies that the terms 
\eqref{eq:HiRf-ss-terms} also must
vanish.

(2) Applying Lemma \ref{lem:tcf-support} 
to $\Sing(Z) \subset Z \subset B$,
we see that 
it suffices to show that the restriction $F|_{B \setminus \Sing(Z)}$ is tcf with respect to the divisor $Z \setminus \Sing(Z)$.
Thus, replacing $B$ by $B \setminus \Sing(Z)$ we can assume that $Z$ is smooth. 
By further shrinking $B$, we can assume that fibers of $f$ over points in $Z$ are not strictly multiple.
Now in the spectral sequence \eqref{eq:ss-Hi}
the bottom row only has one nonzero term $E_2^{2,0} = \Z_Z$.
Indeed, since both $Z$ and $B$ are smooth, we have $i^!\Z_B = \Z_Z[-2]$. In particular, as in (1), the differentials do not affect the  $\HH^1 i^!\rmR^1 f_*\Z_W =E_2^{1,1}$ term.
Let us show that
the differential 
\[
\HH^0 i^!\rmR^1 f_*\Z_W =E_2^{0,1} \to E_2^{2,0} = \Z_Z
\]
is also trivial.
For that, it suffices to check that
the boundary map of the spectral sequence
\[
\partial\colon 
E_2^{2,0} = \Z_Z \to  
\HH^2 i^! \rmR f_* \Z_W
\]
is injective.
The sheaf on the right computes the top  homology groups of the fibers $f^{-1}(z)$, $z \in Z$.
Using the flat pullback formalism \cite[\S3.2]{Verdier-cycleclass} one can see that the morphism 
$\partial$ corresponds to the global section of this sheaf given by the cycle class of the fiber
and so it is injective.
Thus,
\[
E_2^{2,0} = E_\infty^{2,0} = \Z_Z = \rmL_0 \HH^2 i^! \rmR f_* \Z_W, 
\]
where $\rmL_0$ stands for the lowest term of the Leray filtration.

Furthermore, we have
\[
\HH^1 i^! \rmR^1 f_* \Z_W = 
E_2^{1,1} = E_\infty^{1,1} = \rmL_1 \HH^2 i^! \rmR f_* \Z_W  / \rmL_0 \HH^2 i^! \rmR f_* \Z_W \subset \Coker(\partial).
\]
By Lemma \ref{lem:Q-R1-tcf} the sheaf in the left-hand side
is a $\Z$-torsion sheaf.
On the other hand, 
since $\partial$
is given by the cycle class of the fiber, and 
the fibers are not strictly multiple, $\partial$
has a torsion-free cokernel.
Thus, we see that the terms \eqref{eq:HiRf-ss-terms}
are trivial.~\qedhere
\end{proof}

\subsection{Deligne cohomology groups}

Recall that for any smooth variety $W$ and $m \in \bbZ_{\geq 0}$
there is a Deligne complex
\cite{EsnaultViehweg}, \cite[13.3.1]{VoiHTtwo}
\[
\Z_W(m)_\rmD \coloneqq [\bbZ_W \overset{(2\pi i)^m}\longrightarrow \sO_{W}\overset{d}\longrightarrow \Omega_{W}^1\overset{d}\longrightarrow \ldots
\overset{d}\longrightarrow \Omega_{W}^{m-1}]
\]
where the first term of the complex is in cohomological degree zero. 
Deligne complexes can also be defined to be a one-term complex $\Z_W$ for $m < 0$ but we will not need that.

Deligne cohomology groups are defined as hypercohomology groups of the complexes $\Z_W(m)_\rmD$.
We will mostly need Deligne cohomology with coefficients in $\Z_W(2)_\rmD$, however some results
are naturally explained in a more general context.

We have
\[
\Z_W(0)_\rmD = \Z_W, \quad \Z_W(1)_\rmD \simeq \OO_W^*[-1]
\]
and
\begin{equation}\label{eq:Z2D-def}
\bbZ_W(2)_\rmD\coloneqq [\bbZ\overset{(2\pi i)^2}\longrightarrow \sO_{W}\overset{d}\longrightarrow \Omega_{W}^1] \simeq [\OO^*_{W} \xrightarrow{d\log} \Omega_{W}^1][-1]
\end{equation}
The only nonvanishing
cohomology sheaves of the
complex $\Z_W(2)_\rmD$ are 
\[
\HH^{1} \Z_W(2)_\rmD= \C^*_W, \quad
\HH^2 \Z_W(2)_\rmD = \Coker(d\colon \OO_W \to \Omega^1_W).
\]
In particular, the connecting morphism
$\HH^2 \Z_W(2)_\rmD[-2] \to \HH^1 \Z_W(2)_\rmD$ induces
a long exact sequence,
which we will use later:
\begin{equation}\label{eq:DelCoh-gluing}
\ldots \to {\mathrm H}^{k-1}(W, \C^*) \to
{\mathrm H}^k(W, \Z_W(2)_\rmD) \to
{\mathrm H}^{k-2}(W, \Coker(d)) \to
{\mathrm H}^k(W, \C^*) \to \ldots
\end{equation}

Now we explain how to compute Deligne cohomology groups in various cases.
These results are well-known, however, we include them in order to freely use them for the computations in Proposition \ref{prop:dbcohomcurlyY}.

\begin{lemma}\label{lem:DelCoh-Tate}
Let $W$ be a smooth projective variety.
Assume that for some $k \ge 0$ we have
$\rmH^{2k-1}(W, \Z) = \rmH^{2k+1}(W, \Z) = 0$
and $\rmH^{2k}(W, \Z)$
 is a torsion-free Hodge structure of Hodge--Tate type.
Then for all $m \ge 0$ we have
    \[
    {\mathrm H}^{2k}(W, \Z_W(m)_\rmD) = \begin{cases}
        0 & m > k \\
        \rmH^{2k}(W, \Z) & m \le k \\
    \end{cases}
    \]
    \[
    {\mathrm H}^{2k+1}(W, \Z_W(m)_\rmD) = \begin{cases}
        \rmH^{2k}(W, \Z) \otimes \C^* & m > k \\
        0 & m \le k \\
    \end{cases}
    \]
\end{lemma}
\begin{proof}
Using \cite[Proposition 12.26]{VoiHTone} we get an exact sequence
\begin{equation}\label{eq:DelCoh-Voisin-seq}
0 \to {\mathrm H}^{2k}(W, \Z_W(m)_\rmD) \to 
{\mathrm H}^{2k}(W, \Z) \to {\mathrm H}^{2k}(W, \C) / F^m \to {\mathrm H}^{2k+1}(W, \Z_W(m)_\rmD) \to 0.
\end{equation}
The result follows immediately since the Hodge filtration on the Hodge--Tate structure
is a filtration such that\phantom{\qedhere}
\[
{\mathrm H}^{2k}(W, \C) / F^m = \begin{cases}
    {\mathrm H}^{2k}(W, \C) & m > k \\
    0 & m \le k \\
\end{cases}\tag*{\qed}
\]
\end{proof}

\begin{example}
 Deligne cohomology groups of the projective space
 $\P^n$
 are nonzero in precisely $n + 1$ degrees and
 are given by 
\[
{\mathrm H}^k(\P^n, \Z(m)_\rmD) =   
\begin{cases}
            \C^* & k \text{ odd, } 1 \le k \le 2m-1\\
            \Z & k \text{ even, } 2m \le k \le 2n \\
            0 & \text{otherwise}
\end{cases}
\]
This can be computed 
using Lemma \ref{lem:DelCoh-Tate} or
directly 
from the definition by induction on $m$, or 
 using the projective bundle formula 
\cite[Proposition 8.5]{EsnaultViehweg}.
Later we will need the following particular case
\begin{equation}\label{eq:DelCohP5}
{\mathrm H}^*(\P^5, \Z(2)_\rmD) = \C^*[-1] \oplus \C^*[-3] \oplus \Z[-4] \oplus   \Z[-6] 
 \oplus   \Z[-8]  \oplus   \Z[-10].
\end{equation}
\end{example}

If $W$ is a smooth projective variety with $\rmH^{2k}(W, \Z)$ a Hodge structure of K3 type we set 
\[
\Br_{\an}^{2k}(W) := H^{k-1,k+1}(W, \C) / \image(H^{2k}(W, \Z) \to H^{2k}(W, \C) \to H^{k-1,k+1}(W)),
\]
see Appendix \ref{app:Brauer} for the definitions and the details of this construction.

\begin{lemma}\label{lem:DelCoh-K3}
Let $W$ be a smooth projective variety.
Assume that for some $k \ge 0$ we have
${\mathrm H}^{2k-1}(W, \Z) = {\mathrm H}^{2k+1}(W, \Z) = 0$
and ${\mathrm H}^{2k}(W, \Z)$
 is a Hodge structure of K3 type.
Then we have
canonical isomorphisms
    \[
    {\mathrm H}^{2k}(W, \Z_W(k)_\rmD) \simeq
        {\mathrm H}^{k,k}(W, \Z), \quad
    {\mathrm H}^{2k+1}(W, \Z_W(k)_\rmD) \simeq \Br_{\an}^{2k}(W)
    \]
\end{lemma}
\begin{proof}\phantom{\qedhere}
The proof is similar to that of Lemma \ref{lem:DelCoh-Tate}
using \eqref{eq:DelCoh-Voisin-seq}
and 
\[
\rmH^{2k}(W, \C) / F^k = \rmH^{k-1, k+1}(W).\tag*{\qed}
\]
\end{proof}

\begin{example}\label{ex:DeligneCohK3Surface}
If $S$ is a K3 surface then using Lemma \ref{lem:DelCoh-Tate} and Lemma \ref{lem:DelCoh-K3} we obtain 
\[
    {\mathrm H}^k(S, \bbZ_{S}(1)_\rmD) = 
    \begin{cases}
        0 & k = 0\\
        \C^* & k = 1\\
        \NS(S) & k = 2\\
        \Br_{\an}^2(S) & k = 3 \\
        \Z & k = 4 \\
    \end{cases}
\]
\end{example}

\begin{remark}\label{rmk:IJ-DelCoh}
Deligne cohomology groups are closely related to algebraic cycles and intermediate Jacobians. When $Y$ is a smooth projective variety, by \cite[(7.9)]{EsnaultViehweg}, for all $m\in \bbZ_{\geq 1}$, there is a short exact sequence
\begin{equation}\label{eq:IJ-DelCoh}
0 \to J^m(Y) \to {\mathrm H}^{2m}(Y, \Z_Y(m)_\rmD) \to {\mathrm H}^{2m}(Y, \Z) \to 0
\end{equation}
where $J^m(Y)$ is the $m$-intermediate Jacobian
${\mathrm H}^{2m-1}(Y, \C) / ({\mathrm H}^{2m-1}(Y, \Z) + F^m {\mathrm H}^{2m-1}(Y, \C))$.
Since $\bbZ_\rmD(1)\simeq \sO_Y^*[-1]$, for $m=1$ this recovers the exact sequence
\[0\to \Pic^0(Y)\to \Pic(Y)\to \rmH^2(Y,\bbZ)\to 0.\]
Moreover, in the case when $Y$ is a
smooth cubic 
threefold, we let $J(Y)\coloneqq J^2(Y)$.
The cycle class map 
$\CH^2(Y)\to {\mathrm H}^4(Y, \Z)$ factors 
through ${\mathrm H}^4(Y, \Z_Y(2)_\rmD)$ \cite[\S 7]{EsnaultViehweg} in way compatible with the Abel--Jacobi map. Since the group of homological trivial cycles $\CH^2_0(Y)$ is isomorphic to $J(Y)$ \cite[(8)]{Murre}, 
\eqref{eq:IJ-DelCoh} is nothing but the short exact sequence
\[0\to \CH^2_0(Y)\to \CH^2(Y) \xrightarrow{{\rm deg}} \Z \to 0.\]

\end{remark}

\section{The universal hyperplane section of a cubic fourfold}\label{sec:cohomologycomputations}

Given a smooth cubic threefold $Y$, its Jacobian $J(Y)\coloneqq {\rmH^2(Y,\Omega_Y^1)}/{\rmH^3(Y,\bbZ)}$ is a principally polarized abelian variety of dimension 5. 
In this case, the exact sequence \eqref{eq:IJ-DelCoh} reads as
\begin{equation}
\label{eq:J-wtJ-Z}
0 \to J(Y) \to {\mathrm H}^{4}(Y, \Z_Y(2)_\rmD) \to  \Z \to 0.
\end{equation}
In this section, we sheafify and extend this short exact sequence to the case of the universal hyperplane section
of an arbitrary smooth cubic fourfold $X \subset \P^5$. 
 We denote the two sheaves corresponding to the first two terms in \eqref{eq:J-wtJ-Z} by $\JJ$ and $\wt{\JJ}$ respectively, see  \cref{def:J}
 and \cref{def:wtJJ}. 
We introduce some other objects, most importantly the constructible
sheaf $\Lambda$ \eqref{eq:def-LambdaYY}
and a 
closely related integral intersection complex $\Lambda^\bullet$,
see \cref{thm:decomposition-cubic}.
We compute 
their cohomology groups in
\cref{thm:cohom-Lambda},
\cref{thm:main-cohom-wtJJ} and
\cref{thm:Sha-seq-nonDG}.

\subsection{Singular cubic threefolds}\label{sec:singularcubicthreefold}

We start by recalling some results about cubic threefolds with isolated singularities. The following is a well-known result. We provide a complete proof for reader's convenience.

\begin{lemma}\label{lem:Y-terminal} If $Y$ is a cubic threefold with isolated singularities, then its singularities are
Gorenstein and canonical, in particular, rational.
Furthermore, if $Y$ is not a cone, then the singularities are terminal.
\end{lemma}
\begin{proof}
Singularities of $Y$ are Gorenstein because $Y$ is a hypersurface.
If $Y$ is a cone, then the standard computation  of  the canonical class of the resolution 
\cite[Corollary 3.4]{KK13}
shows that the singularity is canonical.  

Assume that $Y$ is not a cone, and let $p \in Y$ be a singular point. Let $S$ be a general hyperplane section of $Y$ through $p$.
By Bertini's theorem $S$ is a cubic surface with an isolated singularity at $p$.
Since $Y$ is not a cone, $S$ is not a cone either.
Indeed if a general hyperplane section through $p$ was a cone, then the Hilbert scheme of lines on $Y$ passing through $p$ would be $2$-dimensional which would force $Y$ to be a cone with vertex $p$.
A singular cubic surface $S$ with 
isolated singularities and 
which is not a cone must have
du Val singularities \cite{BruceWall}.
Thus, $Y$ has a so-called compound du Val singularity at $p$, which is terminal by~\cite{Reid-models}.

 Finally note that canonical singularities are rational, see e.g. \cite[(3.8)]{Reid-young}.
\end{proof}

We have the 
class group $\Cl(Y)$ of rational equivalence
classes of Weil divisors
and the Picard group $\Pic(Y)$ is its subgroup parameterizing Cartier divisors.
We  consider the so-called \emph{defect}
\begin{equation}\label{eq:def-defect}
\delta_Y = \rk (\Cl(Y) / \Pic(Y)) = b_4(Y) - b_2(Y).
\end{equation}
Here the second equality is a result by 
Namikawa--Steenbrink \cite[Theorem 3.2]{NamikawaSteenbrink} which uses the fact that
$Y$ has rational isolated
hypersurface singularities (Lemma \ref{lem:Y-terminal}).
Thus, $Y$ is $\Q$-factorial if and only if the defect is zero. 

\begin{example}\label{ex:nodal-defect}
It is well-known that if $Y$ has a single ordinary double point then $Y$ is $\Q$-factorial and
$\delta_Y = 0$.
On the other hand, if $Y$ contains a plane $\Pi \subset Y$ then $\Pi$ is a Weil divisor. It can not be Cartier by degree reasons 
and so $\delta_Y \ge 1$. 
\cite[Theorem 1.1]{MarqVikt} gives a general explicit formula for the defect $\delta_Y$ of a cubic threefold.
\end{example}

\begin{proposition}\label{prop:cohom-Y}
Let $Y \subset \P^4$ be a cubic threefold with isolated singularities. 
The odd degree integral cohomology groups of
$Y$ are trivial except for $\rmH^3(Y, \Z)$
which is a torsion-free group.
Write $h \in \rmH^2(Y, \Z)$ for the hyperplane class.
The even degree cohomology groups of $Y$ are
\[
\rmH^0(Y, \Z) = \Z, \quad
\rmH^2(Y, \Z) = \Z \cdot h, \quad
\rmH^4(Y, \Z) = \Z^{\oplus (1 + \delta_Y)}, \quad
\rmH^6(Y, \Z) = \Z
\]
where $\delta_Y \ge 0$ is the defect \eqref{eq:def-defect} of $Y$.

Assume that $\delta_Y = 0$. Then $\rmH^4(Y, \Z) = \Z[\ell]$ for any line $\ell \subset Y$
contained in the smooth locus of $Y$.
In particular, in this case the even cohomology
 groups $\rmH^{2*}(Y, \Z)$
 satisfy the integral Poincar\'e duality.
\end{proposition}
\begin{proof}
The result about $\rmH^j(Y, \Z)$ for $j \le 2$
follows from the Weak Lefschetz Theorem.
The Betti numbers of $Y$ except for $b_3(Y)$ and $b_4(Y)$ are determined in \cite[Chapter 5, Theorem 4.3]{Dimca-book}.
By the same result all integral homology and cohomology groups are torsion free except possibly $\rmH_3(Y, \Z)$ and $\rmH^4(Y, \Z)$
which have isomorphic torsion by the Universal Coefficient Theorem. 

Let us show that $\rmH^4(Y, \Z)$ is torsion-free.
If $Y$ is smooth, this follows from the integral Poincar\'e duality, see \cite[Chapter 5, Lemma 3.1]{Dimca-book}.
Assume now that $Y$ is singular. We use an
argument similar  to that of  \cite[Proposition 3.2]{MarqVikt}, where rational cohomology groups of $Y$ are computed.
Take any $p \in \Sing(Y)$
and consider the blow up $\Bl_p(Y)$ with exceptional divisor $E \subset \Bl_p(Y)$.
The singular point $p$ is a double point
or $Y$ is a cone over a smooth cubic surface. 
Indeed, since $Y$ is a cubic, if $p$ is not a double point, all the terms in an affine equation of $Y$ near $p$ are cubic, in which case $Y$ is a cone over a cubic surface. 
This surface has to be smooth since we assumed that $Y$ has isolated singularities.
Thus, $E$ is either a (possibly singular) quadric surface or a smooth cubic surface, and in each case
$\rmH^3(E, \Z) = 0$ 
and part of the standard exact sequence for a blow-up square gives an exact sequence
\begin{equation}\label{eq:H4-E}
0 = \rmH^3(E, \Z) \to \rmH^4(Y, \Z) \to \rmH^4(\Bl_p(Y), \Z).
\end{equation}

It is well known that $\Bl_p(Y)$ is isomorphic to the blow up of $\P^3$ at the possibly singular sextic curve, given by the
complete intersection of a quadric and a cubic \cite[3.1]{CML}.
The standard blow-up formula for cohomology applies in this case, showing that $\rmH^*(\Bl_p(Y), \Z)$ is torsion-free. Indeed, curves have torsion-free cohomology groups.
Thus, $\rmH^4(Y, \Z)$ is also torsion-free by \eqref{eq:H4-E}

Now assume that $\delta_Y = 0$.
For any line $\ell \subset Y$ in the smooth locus of $Y$ we have $h \cdot \ell = 1$,
so the class $\ell$ is primitive,
hence it is a generator of $\rmH^4(Y, \Z) = \Z$.
It is now clear that in this case $\rmH^{2*}(Y, \Z)$ satisfy integral Poincar\'e duality.
\end{proof}

\subsection{Cohomology groups and pushforwards}\label{sec:CohandPushforwards}

Let $X \subset \P^5$ be a smooth cubic fourfold
and $\YY \subset X \times (\P^5)^\vee$
be the universal hyperplane section of $X$.
We write $p\colon \YY \to B \coloneqq (\P^5)^\vee$ 
for the natural projection,
and
 $q\colon \YY \to X$
 for the second projection,
which is a Zariski locally trivial
$\P^4$-fibration.
Fibers of $q$ are irreducible cubic threefolds with at worst isolated singularities. 
Indeed, if $X = \rmV(F)$ and $Y= \rmV(x_0)\cap X$ is a coordinate hyperplane section of $X$, then 
\[
 \Sing(Y) \cap \rmV\left(\frac{\partial F}{\partial x_0}\right) \subset \Sing(X).
\]

We denote by $h \in \Pic(X)$
and $H \in \Pic(B)$  the  hyperplane section classes and their pullbacks to $\YY$.

\begin{lemma}\label{lem:cohom-Y}
The integral cohomology groups of $X$ and of $\YY$ are given by
\begin{equation}\label{eq:cohomYY}
{\mathrm H}^k(X, \bbZ) =
    \begin{cases}
        \bbZ & \text{if } k=0,2, 6, 8\\
        \bbZ^{\oplus 23}& \text{if } k=4\\
        0 & \text{ otherwise}\\
    \end{cases}
\quad
\quad
{\mathrm H}^k(\YY, \bbZ) =
    \begin{cases}
        \bbZ & \text{if } k=0,16\\
        \bbZ^{\oplus 2}& \text{if } k=2, 14\\
        \bbZ^{\oplus 25}& \text{if } k=4, 12\\
        \bbZ^{\oplus 26}& \text{if } k=6, 10\\
        \bbZ^{\oplus 27}& \text{if } k=8\\
        0 & \text{ otherwise}\\
    \end{cases}
\end{equation}
\end{lemma}

\begin{proof}
Cohomology groups of $X$
are well known, see e.g. \cite[\S1.3]{Huycubic}.
For the cohomology groups of $\YY$ we use
projective bundle formula for $q\colon \YY \to X$:
\begin{equation}\label{eq:proj-bundle}
{\mathrm H}^*(\YY, \Z) = \bigoplus_{i=0}^4
q^* {\mathrm H}^*(X, \Z) \cdot p^*H^i.
\end{equation}
\end{proof}

By the integral Poincar\'e duality
the intersection pairing on $\rmH^4(X,\bbZ)$ is unimodular, in particular since $h^2$ is primitive, there exists a class 
\begin{equation}\label{eq:b-class}
b \in \rmH^4(X, \Z)
\end{equation}
such that $\deg_X(h^2 \cdot b) = 1$. 
This class is algebraic for some special cubic fourfolds, but if $X$ is very general, $b$ in not an algebraic class.
We will need this class later.

\begin{lemma}
We have
    \begin{equation}
    \label{eq:cohom-OmegaX}
    \rmH^*(X, \Omega^1_X) = \C[-1] \oplus \C[-3],
    \quad
    \rmH^*(X, \Omega^1_X(-1)) = \C^6[-3]
    \end{equation}
\end{lemma}
\begin{proof}
This is a standard calculation, see \cite[\S1.2]{Huycubic}. 
\end{proof}

\noindent For the next result we define
\begin{equation}\label{eq:def-BX}
\Br^4_{\an}(X) \coloneqq 
\rmH^{1,3}(X) / \image(\rmH^4(X, \Z) \to \rmH^4(X, \C) \to \rmH^{1,3}(X))
\end{equation}
This is a higher-dimensional analog of the analytic 
Brauer group, see Appendix \ref{app:Brauer}, in particular, Example \ref{ex:Br2-X} for more details about this notion.

\begin{proposition}\label{prop:dbcohomcurlyY}
    The cohomology groups of the second Deligne complex on $X$ and on $\YY$ 
    are given as follows:
    \begin{equation}
        {\mathrm H}^k(X, \bbZ_{X}(2)_{\rmD}) = 
        \begin{cases}
            0 & k = 0, 2\\
            \C^* & k = 1, 3\\
            {\mathrm H}^{2,2}(X, \Z) & k = 4\\
            \Br^4_{\an}(X) & k = 5 \\
            {\mathrm H}^k(X, \Z) & k \ge 6
        \end{cases}
    \end{equation}
    
    \begin{equation}\label{eq:CohDeligneComplexY}
        {\mathrm H}^k(\YY, \bbZ_{\YY}(2)_{\rmD}) = 
        \begin{cases}
            0 & k = 0, 2\\
            \bbC^* = {\mathrm H}^0(\YY, \sO^*) & k=1\\
            \bbC^*\times \bbC^* = {\mathrm H}^{1,1}(\YY, \C) / {\mathrm H}^2(\YY, \Z) & k = 3\\
            {\mathrm H}^{2,2}(\YY, \Z) = \Z^2 \oplus {\mathrm H}^{2,2}(X, \Z)  & k = 4\\
              \Br^4_{\an}(\calY)= \Br^4_{\an}(X) & k = 5\\
            {\mathrm H}^{k}(\YY, \Z) & k \ge 6 \\
        \end{cases}
    \end{equation}
\end{proposition}

\begin{proof}
Deligne cohomology groups of $X$
can be computed immediately using Lemma \ref{lem:DelCoh-Tate} and Lemma \ref{lem:DelCoh-K3}.
Deligne cohomology groups of $\YY$
can be computed in the same way, or using the projective bundle formula \cite[Proposition 8.5]{EsnaultViehweg} applied to $q\colon \YY \to X$.
We omit the details of these computations.~\qedhere
\end{proof}

\begin{proposition}\label{prop:pushforward-Omega1}
For any smooth cubic fourfold
$X$ we have
$\rmR  p_*\OO_\YY \simeq \OO_B$
and
\[
\rmR  p_*\Omega^1_\YY \simeq \Omega^1_B[0] \oplus \OO_B[-1] \oplus \Omega^1_B[-2], \quad
\rmR  p_*\Omega^1_{\YY/B} \simeq \OO_B[-1] \oplus \Omega^1_B[-2].
\]
Furthermore, for the complex $\Omega^{\le 1}_\YY = [\OO_\YY \xrightarrow{d} \Omega^1_\YY]$ (with the first term in degree zero) we have
\begin{equation}\label{eq:pushforward-OmegaLE1}
\rmR  p_* \Omega^{\le 1}_\YY \simeq \Omega^{\le 1}_B \oplus \OO_B[-2] \oplus \Omega^1_B[-3].
\end{equation}
\end{proposition}
\begin{proof}

    Recall that $\YY$ is a smooth divisor on $\XX \coloneqq X \times B$
    and $\OO_\XX(\YY) = \OO_X(1) \boxtimes \OO_B(1)$.   
    Let us write $i\colon \YY \into \XX$ for the closed embedding.
    We compute derived pushforwards 
using the formula $\rmR p_{2*}(\LL \boxtimes \LL') = \rmH^*(X, \LL) \otimes \LL'$. 

Applying $\rmR p_{2*}$ to the short exact sequence
for the divisor $\YY \subset \XX$
\[
0 \to \OO_\XX(-\YY) \to \OO_\XX \to i_*\OO_\YY \to 0
\]
we get $\rmR p_*\OO_\YY = \OO_B$.
Now pushing forward the same exact sequence 
twisted by $\OO_\XX(-\YY)$
we obtain
$\rmR p_*\OO_\YY(-\YY) = 0$.
Therefore, applying $\rmR p_{2*}$ to
    the conormal exact sequence
    \[0 \to \OO_\YY(-\YY) \to i^*\Omega^1_{\XX} \to \Omega^1_\YY \to 0\]
we get
    \begin{equation}\label{eq:Rp-Omega-splitting}
    \rmR p_*\Omega^1_\YY = \rmR p_*i^*\Omega^1_{\XX}
    = \rmR p_*i^*p_1^*\Omega^1_X \oplus 
     \rmR p_*i^*p_2^*\Omega^1_B
    \end{equation}
 Since $\rmR p_*\sO_{\YY}\simeq \sO_B[0]$, the second term in the right-hand side of \eqref{eq:Rp-Omega-splitting}
 is quasi-isomorphic to $\Omega^1_B[0]$ by the projection formula.
Let us now consider the first term 
\[
\FF^{\bullet}\coloneqq \rmR p_*i^*p_1^*\Omega^1_X 
=    \rmR p_{2*}i_*i^*p_1^*\Omega^1_X 
=
    \rmR p_{2*}(p_1^*\Omega^1_X \otimes_{\OO_\XX} \OO_\YY).
    \]
Using the short exact sequence for the divisor $\sY\subset \XX$, it is easy to obtain a distinguished triangle
    \begin{equation}\label{eq:seqOmegaXpushforward}
    {\mathrm H}^*(X, \Omega^1_X(-1)) \otimes \OO_B(-1) \to 
    {\mathrm H}^*(X, \Omega^1_X) \otimes \OO_B \to \FF^\bullet.
    \end{equation}

    Applying \eqref{eq:cohom-OmegaX},
    the triangle becomes
    \[
    \OO_B^{\oplus 6}(-1)[-3] \to \OO_B[-1] \oplus \OO_B[-3] \to 
    \FF^\bullet.
    \]  
   Thus, $\FF^\bullet$ is the direct sum of the shifted
   cohomology sheaves $\HH^1 \FF^\bullet \simeq \OO_B$ and
   $\HH^2 \FF^\bullet$ satisfying
    \[
    0 \to \HH^2 \FF^\bullet \to \OO_B^{\oplus 6}(-1) \to \OO_B \to 0
    \]
    which is the Euler exact sequence,
    giving $\HH^2 \FF^\bullet \simeq \Omega^1_B$. This proves the decomposition for 
    $\rmR p_*\Omega^1_\YY$.
To compute $\rmR p_*\Omega_{\sY/B}^1$
we consider the sequence
 \[
 0\to p^*\Omega_B^1\to \Omega_{\sY}^1\to \Omega_{\sY/B}^1\to 0
 \]
 where the first morphism is injective because $B$ is smooth. 
 Applying $\rmR p_*$ to this sequence gives the result.
 
    It remains to prove \eqref{eq:pushforward-OmegaLE1}. We compute
    \[
    \rmR  p_* \Omega^{\le 1}_\YY \simeq
    \mathrm{cone}(\rmR  p_* \OO_\calY \to \rmR  p_* \Omega^{1}_\YY)[-1]
    \simeq 
    \mathrm{cone}(\OO_B[0] \to \Omega^1_B[0] \oplus \OO_B[-1] \oplus \Omega^1_B[-2])[-1].
    \]
    For degree reasons, out of the three components of the morphism inside the cone, only the first one can be nontrivial. 
    The first morphism is the differential. So,
    \[
    \rmR  p_* \Omega^{\le 1}_\YY \simeq \mathrm{cone}(\OO_B \xrightarrow{d} \Omega^1_B)[-1] \oplus \OO_B[-2] \oplus \Omega^1_B[-3]
    \]
    which gives \eqref{eq:pushforward-OmegaLE1}.
\end{proof}

\begin{definition}\label{def:defect-general}
We say that a smooth cubic fourfold
$X$ is \emph{defect general} if 
all hyperplane sections 
$Y \subset X$
have defect \eqref{eq:def-defect}
$\delta_Y = 0$.
\end{definition}

It can be deduced from the results of \cite{LSV} that if $X$ is a very general cubic fourfold then
it is defect general.
It follows from the recent results of Marquand and Viktorova \cite[Theorem 1.2]{MarqVikt}
that $X$ is defect general
 if and only if 
$X$ does not contain a plane or a cubic scroll.

We consider a constructible sheaf on $B$ defined as
\begin{equation}\label{eq:def-LambdaYY}
    \Lambda \coloneqq \rmR^3 p_*\Z_\YY.
\end{equation}
At the general point $b \in B$ it has a rank
$10$, the middle Betti number of a smooth cubic threefold.
This sheaf and its cohomology play the key role in the paper.
We also have another constructible sheaf $\QQ$ fitting in the short exact sequence
\begin{equation}
\label{eq:Q-seq}
0 \to \QQ \to \rmR^4 p_*\Z_\YY \to \Z_B \to 0
\end{equation}
where the last morphism is the fiberwise intersection with a hyperplane class. The sheaf $\QQ$ is a torsion sheaf.
By \cref{prop:cohom-Y}, for every $b \in B$
we have $\QQ_b \simeq \Z^{\delta_{Y_b}}$.

The Leray spectral sequence for $\rmR p_* \Z_\YY$ does not degenerate in general, however we have the following result.
We write $U \subset B$ for the locus parameterizing smooth hyperplane sections so that $\Lambda_U \coloneqq \Lambda|_U$ is a local system of rank $10$.

\begin{theorem}\label{thm:decomposition-cubic}
    Let $X$ be a  
    smooth cubic fourfold.
    We have a  decomposition
  \begin{equation}\label{eq:general-splitting}
    \rmR p_*\Z_\YY \simeq \bigoplus_{i=0}^3 \Z_B[-2i] \oplus \Lambda^\bullet[-3]
\end{equation}
    where $\Lambda^\bullet[5]\simeq {\rm IC}_B(\Lambda_{U})$ is an intersection complex, which is simple over $\bbQ$.
    Furthermore, 
    $\HH^i(\Lambda^\bullet) = 0$ for $i \ne 0, 1$.
    We have
    and $\HH^0(\Lambda^\bullet) \simeq \Lambda$
    and it is a tcf sheaf.
    We have $\sQ \simeq \HH^1(\Lambda^\bullet)$ and it
    is supported precisely along the locus
    of hyperplane sections which have positive defect, which is a locus of codimension at least two. 
\end{theorem}

\begin{remark}    
With rational coefficients Theorem \ref{thm:decomposition-cubic} follows from more general results on Decomposition Theorem applied to universal hyperplane sections 
\cite{BFNP},
\cite[Theorem 1]{Schresidues} and \cite[\S 3.3 Lemma]{Bei15}, however for our applications
it is crucial to work with integer coefficients.
\end{remark}

It follows from Theorem \ref{thm:decomposition-cubic} that
\[
\text{$X$
    is defect general $\iff$  $\sQ = 0$ $\iff$
    $\Lambda^\bullet \simeq \Lambda[0]$.
}\]
Furthermore, for any smooth $X$, \eqref{eq:general-splitting} provides
a splitting $\rmR^4p_*\bbZ_{\sY}\simeq \bbZ_B\oplus \sQ$ which can also be obtained from \eqref{eq:Q-seq} directly.

\begin{proof}[Proof of \cref{thm:decomposition-cubic}]
Consider a class $b \in \rmH^4(X, \Z)$ defined in  \eqref{eq:b-class}.
We apply Proposition \ref{prop:decomp-Z} with classes
$\xi_i \in \rmH^{2i}(\YY, \Z)$ given as
\[
\xi_0 = 1, \quad \xi_1 = q^*h, \quad 
\xi_2 = q^*b, \quad 
\xi_3 = q^*(bh)
\]
where $q\colon \sY\to X$ is the usual projection.
On a smooth fiber $Y_b \subset \YY$
we have
\begin{equation}\label{eq:xi-duality}
\rmH^*(Y_b, \Z) = \displaystyle\bigoplus_{k=0}^3 \Z \xi_k|_{Y_b} \oplus \rmH^3(Y_b, \Z)\
\text{ and }\
\deg(\xi_i|_{Y_b} \cdot \xi_{3-j}|_{Y_b}) = 
\delta_{ij}
\end{equation}
The splitting \eqref{eq:general-splitting}
follows.

By taking cohomology sheaves in
\eqref{eq:general-splitting}, using Proposition \ref{prop:cohom-Y}
we see that the only nontrivial cohomology sheaves of $\Lambda^\bullet$ are
 $\HH^0(\Lambda^\bullet) \simeq \Lambda$ and $\HH^1(\Lambda^\bullet) \simeq \QQ$.
Since general singular hyperplane sections are 1-nodal and have trivial defect by Example \ref{ex:nodal-defect}, it follows $\QQ$ is supported in codimension at least two.

Let us prove that $\Lambda^\bullet[5]$ is quasi-isomorphic to the intersection complex $\IC(\Lambda_U)$.
We check
that it satisfies the  conditions 
\cite[p. 107]{GoreskyMacphersonII}.
The support condition
is satisfied because $\QQ  = \HH^{-4}(\Lambda^\bullet[5])$ has support of dimension less than $4$. 
Since $\sH^i(\Lambda^{\bullet})$ is $\bbZ$-torsion free and $\Lambda^{\bullet}[5]$ is self-dual, the sheaf $\Lambda^\bullet[5]$ also satisfies the cosupport condition (see e.g.\ the argument of \cite[p. 120]{GoreskyMacphersonII}).
Then the quasi-isomorphism
$\Lambda^\bullet[5]\simeq \IC(\Lambda_U)$
follows from uniqueness of the extension satisfying these conditions (see \cite[Theorem 5.2.8]{Dimca-sheaves}, \cite[p.\ 107, Uniqueness theorem]{GoreskyMacphersonII}).

The sheaf $\Lambda = \HH^0(\Lambda^\bullet)$
is tcf by Proposition \ref{prop:IC-bottom}. 
Finally, the complex $\IC(\Lambda_{U,\bbQ})$ is a simple perverse sheaf because $\Lambda_{U,\bbQ}$
    is an irreducible local system \cite[Corollary 3.28]{VoiHTtwo}.
\end{proof}

\begin{theorem}
\label{thm:cohom-Lambda}
For any smooth cubic fourfold $X$, the cohomology groups of the intersection complex $\Lambda^\bullet$ constructed in \cref{thm:decomposition-cubic}
are given by
\begin{equation}\label{eq:QcohomLambda}
\rmH^k(B, \Lambda^\bullet) = 
    \begin{cases}
        \Z^{\oplus 22} & \text{if } k=1,3,7, 9\\
        \Z^{\oplus 23} & \text{if } k= 5\\
        0 &  \text{otherwise}\\
    \end{cases}
\end{equation}
Furthermore, we have
a canonical isomorphism of Hodge structures
\[
\rmH^1(B, \Lambda^\bullet) \simeq
\rmH^4(X, \Z)_{\rm pr}
\]
where the Hodge structure on
$\rmH^i(B, \Lambda^\bullet)$ is induced from the Hodge structure on $\rmH^4(\YY, \Z)$ via  decomposition \eqref{eq:general-splitting}. Moreover,
the pairing $\deg_\YY(\alpha \beta H^4)$ for $\alpha, \beta \in \rmH^1(B, \Lambda^\bullet)$
where $H \in \rmH^2(B, \Z)$ is the ample generator, corresponds to the intersection pairing on
$\rmH^4(X, \Z)_{\rm pr}$.
\end{theorem}

\begin{proof}
As in Remark \ref{rmk:splittingcohomYY}, consider the morphisms
\[
\Z \oplus \Z[-2] \xrightarrow{(\gamma_0, \gamma_1)}
\rmR p_* \Z_\YY \xrightarrow{(\gamma_0^{\vee}, \gamma_1^{\vee})} \Z[-6] \oplus \Z[-4].
\]
where $\gamma_i$ are corresponding to
the classes $\xi_0 = 1$, $\xi_1 = q^* h$
used in the proof of  \cref{thm:decomposition-cubic}. 
The first morphism is an embedding of a direct summand, the second one is a projection onto another direct summand, and the composition of the two is zero. The cohomology groups $\rmH^*(B, \Lambda^\bullet[-3])$
are canonically isomorphic to the middle cohomology of the following complex:
\[
\rmH^*(B, \Z) \oplus \rmH^{*-2}(B, \Z) \xrightarrow{(p^*(-),h\cdot p^*(-))}
 \rmH^*(\YY, \Z) \xrightarrow{(p_*(-),p_*(h\cdot -))}
\rmH^{*-6}(B, \Z) \oplus \rmH^{*-4}(B, \Z).
\]
Here $h$ denotes the hyperplane class on $X$ pulled back to $\sY$.

In particular, $\rmH^1(B, \Lambda^\bullet)$
is the middle cohomology of the complex of integral Hodge structures
\begin{equation}\label{eq:H1Lambda-formula}
\rmH^4(B, \Z) \oplus \rmH^{2}(B, \Z) \xrightarrow{(p^*(-),h\cdot p^*(-))}
 \rmH^4(\YY, \Z) \xrightarrow{p_*(h\cdot -)}
\rmH^{0}(B, \Z).
\end{equation}
From \eqref{eq:proj-bundle} we get
\[
\rmH^4(\YY, \Z) = 
\Z H^2 \oplus \Z hH \oplus q^*\rmH^4(X, \Z)
\]
with the first two terms corresponding to the images of the first arrow in \eqref{eq:H1Lambda-formula}.
Thus, we obtain a canonical identification of Hodge structures
\[
\rmH^1(B, \Lambda^\bullet) \simeq \Ker(q^*\rmH^4(X, \Z) 
\xrightarrow{p_*(h\cdot -)}
\rmH^{0}(B, \Z)) = q^*\rmH^4(X, \Z)_{\rm pr}.
\]

We can do a similar computation
for $\rmH^9(B, \Lambda^\bullet)$, using instead the middle cohomology of
\[
\rmH^{10}(B, \Z) \xrightarrow{h\cdot p^*(-)}
 \rmH^{12}(\YY, \Z) \xrightarrow{(p_*(-),p_*(h\cdot -))}
 \rmH^{6}(B, \Z) \oplus
\rmH^{8}(B, \Z).
\]
The resulting cohomology is $\rmH^9(B, \Lambda^\bullet) = q^*\rmH^4(X, \Z)_{\rm pr} \cdot H^4 \subset \rmH^{12}(\YY, \Z)$.

Finally, note that the isomorphism $\rmH^1(B, \Lambda^\bullet) \overset{\cdot H^4}{\simeq} \rmH^9(B, \Lambda^\bullet)$ induces a pairing on $\rmH^1(B, \Lambda^{\bullet})$ by sending $\alpha, \beta\in \rmH^1(B, \Lambda^{\bullet})$ to $\deg_{\sY}(\alpha\cdot \beta H^4)$, which corresponds to the intersection pairing on $\rmH^4(X, \Z)_{\rm pr}\simeq \rmH^1(B, \Lambda^\bullet)$.

The ranks of the other cohomology groups 
can be computed similarly using
Lemma \ref{lem:cohom-Y}.
\end{proof}

 By Theorem \ref{thm:decomposition-cubic}  we have the following distinguished triangle useful for accessing $\Lambda$:
\begin{equation}\label{eq:LambdaLambdabulletQ}
\Lambda\to \Lambda^{\bullet}\to \sQ[-1]\to \Lambda[1].
\end{equation}
Denote by $\Sigma\subset \rmH^4(X,\bbZ)$ the sublattice consisting of Poincaré duals of pushforward of classes in $\rmH_4(Y,\bbZ)$ for all hyperplane sections $Y_b$, $b \in B$. Equivalently, $\Sigma$ is generated by algebraic classes in $\rmH^4(X,\bbZ)$ contained in hyperplane sections (see \cite[\S 3]{NamikawaSteenbrink}, or \cite[\S 3.1]{MarqVikt}).
Note that if $X$ is defect general, then $\Sigma = \Z h^2$.

\begin{lemma}\label{lem:H1LambdaSigma}
    We have $\rmH^1(B,\Lambda) \simeq \Ker\left(\rmH^4(X,\bbZ)\to \rmH^0(B, \rmR^4p_*\bbZ_{\sY})\right)\simeq \Sigma^\perp$. 
\end{lemma}
\begin{proof}
     It follows from (\ref{eq:LambdaLambdabulletQ}) that $\rmH^1(B,\Lambda)$ is the kernel of the map $\rmH^1(B,\Lambda^\bullet) \to \rmH^0(B,\calQ)$.

     Observe from \eqref{eq:H1Lambda-formula} and the decomposition of $\rmH^4(\sY,\bbZ)$ as cohomology of projective bundle over $X$, 
     that the map
     \[\rmH^4(X,\bbZ)\xrightarrow{p_*(q^*h\cdot -)} \rmH^0(B,\bbZ)\]
     has kernel $\rmH^1(B,\Lambda^{\bullet})$ and factors through 
     $\rmH^0(B, \rmR^4p_*\bbZ_{\sY})$.  Therefore, by Snake Lemma applied to 
     \[
     \begin{tikzcd}
         0\rar& \rmH^1(B,\Lambda^{\bullet})\rar\dar& \rmH^4(X,\bbZ)\rar["p_*(q^*h\cdot -)"]\dar& \rmH^0(B,\bbZ)\rar\dar[equal]& 0\\
     0\rar & \rmH^0(B,\sQ)\rar&\rmH^0(B, \rmR^4p_*\bbZ_{\sY})\rar &\rmH^0(B,\bbZ)\rar&0 
     \end{tikzcd}
     \]
     we obtain the first isomorphism. Now, note that
    \[\Ker\left(\rmH^4(X,\bbZ)\to \rmH^0(B, \rmR^4p_*\bbZ_{\sY})\right)= \{\alpha\in \rmH^4(X,\bbZ) \ | \ \alpha|_{Y_b}=0 \ \forall \ b\in B\}.\]
    For any hyperplane section $Y\subset X$, since $\rmH^4(Y,\bbZ)$ is torsion-free by Proposition \ref{prop:cohom-Y}, there is a perfect pairing $(-,-)_Y\colon \rmH^4(Y,\bbZ)\times \rmH_4(Y,\bbZ)\to \bbZ$ on $Y$. Therefore, for a class $\alpha\in \rmH^4(X,\bbZ)$ we have
    $$\alpha|_Y=0 \iff ( \alpha|_Y,\beta)_Y=0 \ \forall \beta\in \rmH_4(Y,\bbZ) \iff  (\alpha,i_{Y*}\beta)_X=0 \ \forall \beta\in \rmH_4(Y,\bbZ),$$
    where $i_{Y}\colon Y\hookrightarrow X$ denotes the embedding. This proves the second isomorphism.
\end{proof}

\subsection{Abelian sheaves \texorpdfstring{$\JJ$}{J} and \texorpdfstring{$\wt{\JJ}$}{J̃}
and their cohomology groups}\label{sec:JandJtilde}

Let $X$ be a smooth cubic fourfold.
Over the locus $U \subset B$
parameterizing  smooth hyperplane sections of $X$ 
we have a well-defined intermediate Jacobian sheaf $\JJ_U$ defined
as follows.
Hodge theory provides us 
with an injective
morphism of sheaves of abelian groups (cf.\ \cite[\S 7.1.1]{VoiHTtwo})
\begin{equation}\label{eq:LambdaU-emb}
\Lambda|_U = \rmR^3 p_*\Z_{\YY_U} \into
\rmR^3 p_*\C_{\YY_U} \to 
\rmR^3 p_*\Omega_{\YY_U/U}^{\le 1} \to 
\rmR^2 p_*\Omega_{\YY_U/U}^1
\simeq  \rmR^2p_*\Omega_{\sY_U}^1.
\end{equation}
where we have used Proposition \ref{prop:pushforward-Omega1} in the last isomorphism.
The relative intermediate Jacobian $\JJ_U$ is defined as the cokernel of \eqref{eq:LambdaU-emb}. We will now extend this sheaf over $B$.

\begin{proposition}\label{prop:jaclatticeinjection}
There exists a unique morphism
$
\epsilon\colon \Lambda \to \rmR^2p_*\Omega_{\sY}^1
$
extending \eqref{eq:LambdaU-emb},
and 
this morphism is also injective.
\end{proposition}
\begin{proof}
Injectivity of any such morphism follows since the kernel
must be a torsion subsheaf
$\Lambda$ which has to be zero
because $\Lambda$ is tcf (in particular, torsion free)
by Theorem \ref{thm:decomposition-cubic}.
The same argument proves uniqueness.

For existence we argue as follows.
Using the exponential exact sequence 
and $\rmR p_*\OO_\YY = \OO_B$ (see Proposition \ref{prop:pushforward-Omega1}) we get
\[
\rmR^2 p_*\OO_\YY^* = 
\rmR^3 p_*\Z_\YY = \Lambda.
\]
We use $d\log\colon \OO_\YY^* \to \Omega^1_{\sY}$
to define $\epsilon$ to be the
induced
morphism
\[
\rmR^2 p_*\OO_\YY^* \xrightarrow{\rmR^2 p_*(d\log)} \rmR^2 p_*\Omega^1_{\sY}
\]
and this extends \eqref{eq:LambdaU-emb}
because that morphism is also induced
by $d \log$.
\end{proof}

\begin{definition}\label{def:J}
We define the \textit{relative
intermediate Jacobian sheaf} of $\YY$
as $\JJ \coloneqq \dfrac{\rmR^2p_*\Omega^1_{\sY}}{\epsilon(\Lambda)}$.
\end{definition}
\noindent Note that by Proposition \ref{prop:pushforward-Omega1}, we have
$\rmR^2p_*\Omega^1_{\sY}\simeq \Omega_B^1$ and therefore the following exact sequence of sheaves of abelian groups
which we will frequently use:
\begin{equation}\label{eq:LambdaOmegaJ}
        0 \to \Lambda \to \Omega^1_B \to \calJ \to 0.
    \end{equation}

\begin{lemma}\label{lem:CohJCohLambda}
    For any $k\geq 1$, we have 
    $$\rmH^{2k}(B,\calJ)\simeq \rmH^{2k+1}(\Lambda),\ \text{ and }\ \rmH^{2k+1}(B,\calJ)=0.$$
\end{lemma}

\begin{proof}
    Follows by computing cohomology of the sequence (\ref{eq:LambdaOmegaJ}). 
\end{proof}

We will compute the remaining cohomology groups $\rmH^{0}(B,\calJ)$ and $\rmH^{1}(B,\calJ)$
in \cref{thm:Sha-seq-nonDG}.
In order to do that
we define another important sheaf of abelian groups $\wt{\JJ}$
using
an idea of Voisin \cite{Voitwist}
to push forward the second Deligne
complex \eqref{eq:Z2D-def}.
While cohomology groups of $\JJ$ are directly relevant for the computation of the Tate--Shafarevich group, cohomology groups of $\wt{\JJ}$ are accessible via the Leray spectral sequence and are easier to describe.

\begin{definition}
\label{def:wtJJ} 
    For a smooth cubic fourfold $X$ we define the \textit{extended relative intermediate Jacobian sheaf} of $\YY$
    to be $\wt{\JJ} \coloneqq \rmR^4 p_*\Z_{\sY}(2)_{\rmD}$.
\end{definition}

\begin{proposition}\label{prop:JJ-wtJJ}

The sheaves $\JJ$ and $\wt{\JJ}$
are related by the exact sequence
\begin{equation}
\label{eq:JJ-wtJJ}
    0\to \sJ\to \wt{\JJ} \to \rmR^4p_*\Z_\YY \to 0.
\end{equation}
\end{proposition}
\begin{proof}
Let $\Omega^{\le 1}_{\YY}$ be the complex
$[\OO_\YY \to \Omega^1_\YY]$.
From \eqref{eq:Z2D-def} 
we get a distinguished triangle 
\begin{equation}\label{eq:distinguished-deligne}
    \Omega^{\le 1}_{\YY}[-1] \to \Z_\calY(2)_{\rmD} \to \Z_{\YY} \to \Omega^{\le 1}_\YY.
\end{equation}
We claim that pushing it forward by $p$
we get an exact sequence
\begin{equation}\label{eq:exactJtildeJ}
 0\to \Lambda \xrightarrow{\epsilon} \rmR^2p_*\Omega^1_{\sY} \to
\wt{\JJ} \to \rmR^4 p_* \Z \to 0.
\end{equation}
Indeed, by \cref{prop:pushforward-Omega1} we have $\rmR^3 p_* \Omega^{\le 1}_{\YY} \simeq \rmR^2p_*\Omega^1_{\sY}$
and $\rmR^4 p_* \Omega^{\le 1}_{\YY} = 0$.
Furthermore, the first morphism in the exact sequence coincides with the injective map $\epsilon$ as in \cref{prop:jaclatticeinjection},
and we obtain 
\eqref{eq:JJ-wtJJ}.
\end{proof}

\begin{lemma}\label{lem:cohom-wtJJ}
    For any $k\geq 1$, we have a short exact sequence
    $$0\to  \rmH^{2k+1}(\Lambda^{\bullet})\to \rmH^{2k}(B,\wt{\calJ})\to \bbZ\to 0,\ \text{and }\ \rmH^{2k+1}(B,\tilde{\calJ})=0.$$
\end{lemma}

\begin{proof}

Using the decompositions from Proposition \ref{prop:pushforward-Omega1} and Theorem \ref{thm:decomposition-cubic}, we get a morphism $\Lambda^\bullet \to \Omega^1_B$. 
Denote its cone by $\calJ'$. 
 Using truncation of \eqref{eq:distinguished-deligne},
we obtain a diagram (with distinguished triangles in the rows)

\begin{equation*}
    \begin{tikzcd}
        \Lambda^\bullet[-3] \ar[r] \ar[d] & \Omega^1_B[-3] \ar[r] \ar[d,"\simeq"] & \calJ'[-3] \ar[r] \ar[d] & \Lambda^\bullet[-2] \ar[d] \\
        \tau^{\geq 3}\rmR p_*\bbZ_\calY \ar[r,"\varphi"] & \tau^{\geq 3}\rmR p_*\Omega^{\leq 1}_\calY \ar[r] 
        & {\rm cone}(\varphi) \ar[r]
        & \tau^{\geq 3}\rmR p_*\bbZ_\calY[1]
    \end{tikzcd}
\end{equation*}
The long cohomology sequences, after identifying $\calJ=\Omega^1_B/\Lambda$, give the diagram:
\begin{equation}
    \begin{tikzcd}\label{eq:diagJ'}
        0 \ar[r]  & \calJ \ar[r] \ar[d] & \calJ' \ar[r] \ar[d] & \calQ \ar[r] \ar[d] & 0\\
        0 \ar[r] & \calJ \ar[r] & \widetilde{\calJ} \ar[r] & \rmR^4p_*\bbZ_\calY \ar[r] & 0.
    \end{tikzcd}
\end{equation}
In other words, $\calJ'$ identifies with the kernel of the composition $\widetilde{\calJ} \to \rmR^4p_*\bbZ_\calY \to \bbZ_B$,
where the second morphism is (\ref{eq:Q-seq}). 
The cohomology of $\wt{\sJ}$ can now be computed from the short exact sequence $0 \to \calJ' \to \widetilde{\calJ} \to \bbZ_B \to 0$, after computing the cohomology of $\sJ'$. The latter is done using the triangle $\Lambda^\bullet \to \Omega^1_B \to \calJ' \to \Lambda^\bullet[1]$ and using Theorem \ref{thm:cohom-Lambda} where cohomology of $\Lambda^\bullet$ is computed.

\end{proof}

We will now 
calculate the remaining cohomology groups $\rmH^0(B, \wt{\sJ})$ and $\rmH^1(B, \wt{\sJ})$ 
Recall the definition $\Br^4_{\an}(X)\coloneqq \rmH^{1,3}(X)/\im(\rmH^4(X,\bbZ)\to \rmH^{1,3}(X))$.

\begin{theorem}\label{thm:main-cohom-wtJJ}
Let $X$ be a smooth cubic fourfold. 
We have canonical isomorphisms
\begin{equation}
\rmH^0(B, \widetilde{\sJ})\simeq \rmH^{2,2}(X,\bbZ)
\end{equation}
\begin{equation}
\rmH^1(B, \wt{\JJ}) \simeq \rmH^5(\YY, \Z_\YY(2)_{\rmD}) \simeq \Br^4_{\an}(\calY) \simeq \Br^4_{\an}(X).
\end{equation}
\end{theorem}

The proof of Theorem \ref{thm:main-cohom-wtJJ}
uses the Leray spectral sequence for $\rmR p_*\Z(2)_{\rmD}$.
We first compute the corresponding higher derived image sheaves.

\begin{lemma}
\label{lem:dbpushforward}
We have $\tau^{\le 2} \rmR  p_*\bbZ_{\YY}(2)_{\rmD} \simeq
\Z_B(2)_{\rmD}$ and
    \begin{equation}
        \rmR^i p_*\bbZ_{\YY}(2)_{\rmD} = 
        \begin{cases}
            0 & i=0\\
            \C^*_B & i=1\\
            \Coker(\sO_B\overset{d}{\to} \Omega_B^1) & i=2\\
            \calO^*_B & i=3 \\
            \wt{\JJ} & i = 4 \\
            0 & i= 5\\
             \bbZ_B & i= 6
        \end{cases}
    \end{equation}
\end{lemma}
\begin{proof}
We use the cohomology sheaves
exact sequence
for the distinguished triangle
\begin{equation}\label{eq:DerivedPushSeqofDeligneCoh}
\rmR p_* \Z_\YY(2)_{\rmD} \to \rmR p_* \Z_\YY \to \rmR p_* \Omega^{\le 1}_\YY \to
\rmR p_* \Z_\YY(2)_{\rmD}[1]
\end{equation}
induced by \eqref{eq:distinguished-deligne}.
The cohomologies of the second 
term can be computed from the decomposition in Theorem \ref{thm:decomposition-cubic} and
the cohomologies of the third term
follow easily from 
\eqref{eq:pushforward-OmegaLE1}, 
in particular, $\rmR^0 p_* \Omega^{\le 1}_B \simeq \C_B$
and 
$
\CC \coloneqq \rmR^1 p_* \Omega^{\le 1}_B = \Coker(d)
$
for the de Rham differential $d\colon \OO_B \to \Omega^1_B$.
Writing $\HH^i$ for $\rmR^i p_*  \Z_\YY(2)_{\rmD}$
to simplify the notation we obtain the following exact sequence:
\[
\hspace*{-1.1cm}
\begin{tikzcd}[row sep = small]
    0 \rar & \HH^0 \rar & \bbZ_B \rar & \bbC_B \ar[out=0, in=180, looseness=3]{dll}\\
    & \HH^1 \rar & 0 \rar & \CC \ar[out=0, in=180, looseness=3]{dll}\\
    & \HH^2 \rar & \bbZ_B \rar & \OO_B \ar[out=0, in=180, looseness=3]{dll}\\
    & \HH^3 \rar & \Lambda \rar & \Omega^1_B \ar[out=0, in=180, looseness=3]{dll}\\
    & \wt{\JJ} \rar & \rmR^4p_*\bbZ_\calY \rar & 0 \ar[out=0, in=180, looseness=3]{dll}\\
    & \HH^5 \rar & 0 \rar & 0 \ar[out=0, in=180, looseness=3]{dll}\\
    & \HH^6 \rar & \Z_B \rar & 0 \rar & 0 \\
\end{tikzcd}
\]

The result follows immediately since all nontrivial maps between the terms in the second and third columns
are injective.
\end{proof}

\begin{proof}[Proof of Theorem \ref{thm:main-cohom-wtJJ}]
    We use the Leray spectral sequence 
    \begin{equation}\label{eq:LeraySSforDeligneCoh}
    E_2^{p,q}\coloneqq \rmH^p(B, \rmR^q p_*\bbZ_{\YY}(2)_{\rmD})
    \Longrightarrow \rmH^{p+q}(\YY, \Z_{\sY}(2)_{\rmD}).
    \end{equation}
    The sheaves 
    $\rmR^q p_*\bbZ_{\YY}(2)_{\rmD}$ are computed in  \cref{lem:dbpushforward}.
   The spectral sequence  takes the form: 
\[\arraycolsep=4pt
\begin{array}{|c|c|c|c|c|c|c|c|c|c|c|c|}
\hline
6 & \mathbb{Z} & 0 & \mathbb{Z} & 0 & \mathbb{Z} & 0 & \mathbb{Z} & 0 & \mathbb{Z} & 0 & \mathbb{Z} \\
\hline
5 & 0 & 0 & 0 & 0 & 0 & 0 & 0 & 0 & 0 & 0 & 0 \\
\hline
4 & \rmH^0 & \rmH^1 & \Z^{23} & 0 & \Z^{24} & 0 & \Z^{23} & 0 & \Z^{23} & 0 & \Z \\
\hline
3 & \bbC^*& \mathbb{Z} & 0 & \mathbb{Z} & 0 & \mathbb{Z} & 0 & \mathbb{Z} & 0 & \mathbb{Z} & 0 \\
\hline
2 & 0 & 0 & \mathbb{C} & 0 & \mathbb{C} & 0 & \mathbb{C} & 0 & \mathbb{C} & 0 & 0 \\
\hline

1 & \mathbb{C}^* & 0 & \mathbb{C}^* & 0 & \mathbb{C}^* & 0 & \mathbb{C}^* & 0 & \mathbb{C}^* & 0 & \mathbb{C}^* \\
\hline
0 & 0 & 0 & 0 & 0 & 0 & 0 & 0 & 0 & 0 & 0 & 0 \\
\hline
\tikz[baseline=(X.base)]{
    \node[minimum size=2em] (X) {\strut};
    \draw (X.north east) -- (X.south west);
    \node[anchor=north west] at (X.north west) {$q$};
    \node[anchor=south east] at (X.south east) {$p$};
} & 0 & 1 & 2 & 3 & 4 & 5 & 6 & 7 & 8 & 9 & 10 \\
\hline
\end{array}
\]
In the fourth row we filled in
the cohomology groups from  \cref{lem:cohom-wtJJ}. We used the notation
$\rmH^0$ and $\rmH^1$ for the first two cohomology groups of $\wt{\JJ}$ which we need to compute.

Let us explain how to obtain the third page of this spectral sequence.
Using the first claim of Lemma \ref{lem:dbpushforward} we see that
the second differentials
between the $p = 2$ and $p = 1$ 
rows
are part of the long exact sequence
\eqref{eq:DelCoh-gluing}
\[
\rmH^{p}(B, \Z_B(2)_{\rmD}) \to E_2^{p-2, 2} \xrightarrow{d_2} E_2^{p,1} \to \rmH^{p+1}(B, \Z_B(2)_{\rmD}).
\]
Plugging in the Deligne cohomology groups
$\rmH^{p}(B, \Z_B(2)_{\rmD})$ from \eqref{eq:DelCohP5}
we see that
each nontrivial differential 
$d_2\colon E_2^{2k,2} \to E_2^{2k+2,1}$
is the exponential map $\C \to \C^*$, which
is surjective with kernel isomorphic to $\Z$, 
in particular, $E_3^{2k,2}\simeq \bbZ$, except possibly when $k=1$
because in that case we also have a differential 
$d_2\colon E_2^{0,3} \to E_2^{2,2}$.
However this differential factors via
\[
\C^* = E_2^{0,3} \to \Ker(E_2^{2,2} \xrightarrow{d_2} E_2^{4,1}) = \Z \subset \C = E_2^{2,2}
\]
and this map is zero because $\C^*$ is a divisible group.
Thus,  the differential $d_2\colon E_2^{0,3} \to E_2^{2,2}$ is  
trivial and $E_3^{2,2}\simeq \bbZ$.

Finally, the differential $d_2\colon E_2^{1,4} \to E_2^{3,3}$ is trivial for the following reason.
The rank of the group $\rmH^6(\YY, \Z(2)_{\rmD})$ is $26$
by Proposition \ref{prop:dbcohomcurlyY} and Lemma \ref{lem:cohom-Y},
and so it is
equal to the sum of the ranks of the groups on the sixth diagonal. Therefore, the differential with the target $E_2^{3,3}$ must be trivial.
We get the third page of the spectral sequence:
\[\arraycolsep=4pt
\begin{array}{|c|c|c|c|c|c|c|c|c|c|c|c|}
\hline
6 & \mathbb{Z} & 0 & \mathbb{Z} & 0 & \mathbb{Z} & 0 & \mathbb{Z} & 0 & \mathbb{Z} & 0 & \mathbb{Z} \\
\hline
5 & 0 & 0 & 0 & 0 & 0 & 0 & 0 & 0 & 0 & 0 & 0 \\
\hline
4 & \rmH^0 & \rmH^1 & \Z^{23} & 0 & \Z^{24} & 0 & \Z^{23} & 0 & \Z^{23} & 0 & \Z \\
\hline
3 & \bbC^*& \mathbb{Z} & 0 & \mathbb{Z} & 0 & \mathbb{Z} & 0 & \mathbb{Z} & 0 & \mathbb{Z} & 0 \\
\hline
2 & 0 & 0 & \mathbb{Z} & 0 & \mathbb{Z} & 0 & \mathbb{Z} & 0 & \mathbb{Z} & 0 & 0 \\
\hline
1 & \mathbb{C}^* & 0 & \mathbb{C}^* & 0 & 0
& 0 & 0 & 0 & 0 & 0 & 0 \\
\hline
0 & 0 & 0 & 0 & 0 & 0 & 0 & 0 & 0 & 0 & 0 & 0 \\
\hline
\tikz[baseline=(X.base)]{
    \draw (X.north east) -- (X.south west);
    \node[anchor=north west] at (X.north west) {$q$};
    \node[anchor=south east] at (X.south east) {$p$};
} & 0 & 1 & 2 & 3 & 4 & 5 & 6 & 7 & 8 & 9 & 10 \\
\hline
\end{array}
\]

The only possible nontrivial third differential
is $d_3\colon E_3^{1,4} \to E_3^{4,2}$.
However, this differential must be
zero
by the same argument involving the rank of $\rmH^6(\sY, \bbZ_{\sY}(2)_{\rmD})$ as above.

All higher differentials are zero by degree reasons and the spectral sequence degenerates at the third page.
In particular, we
see that $\rmH^1 = \rmH^1(B, \wt{\JJ})$
is canonically isomorphic to $\rmH^5(\YY, \Z(2)_{\rmD})$, and this group was computed in Proposition \ref{prop:dbcohomcurlyY}.

To compute $\rmH^0 = \rmH^0(B, \wt{\JJ})$ term
we recall 
\begin{equation}\label{eq:H4Y-D-split}
\rmH^4(\YY, \Z(2)_{\rmD}) = \rmH^{2,2}(\YY, \Z) = \Z \cdot H^2 \oplus \Z \cdot h H \oplus \rmH^{2,2}(X, \Z).
\end{equation}
We claim that the first two terms in the right-hand side of \eqref{eq:H4Y-D-split}
are contained in the Leray filtration
term
\[
\rmL_{3} \rmH^4(\YY, \Z(2)_{\rmD}) = 
\Ker\left(\rmH^4(\YY, \Z(2)_{\rmD}) \to \rmH^0(B, \rmR^4 p_*\Z_\YY(2)_{\rmD})\right).
\]
To see the claim, note that both $H^2$ and $hH$ are induced from the (trivial) projective bundle
$X \times B \to B$ and restrict trivially to the fibers of this projective bundle, so the claim follows by functoriality of the Leray filtration.
Since we see from the spectral sequence above
that $\rmL_{3} \rmH^4(\YY, \Z(2)_{\rmD})$
has rank two and $\Z \cdot H^2 \oplus \Z \cdot h H$
is saturated inside $\rmH^4(\YY, \Z(2)_{\rmD})$
we get
$
\rmL_{3} \rmH^4(\YY, \Z(2)_{\rmD}) = \Z H^2 \oplus \Z h H.
$
Therefore, we have \phantom{\qedhere}
\begin{equation}
\rmH^0(B, \wt{\JJ}) = 
\rmH^4(\YY, \Z(2)_{\rmD}) / \rmL_{3} \rmH^4(\YY, \Z(2)_{\rmD}) = \rmH^{2,2}(X, \Z).\tag*{\qed}
\end{equation}
\end{proof}

The following theorem relates the cohomology groups of $\JJ$ and $\wt{\JJ}$.
We write $\Br_{\rm an}(\Sigma^\perp)$ for the Brauer group of the Hodge substructure $\Sigma^\perp \subset \rmH^4(X,\bbZ)$, see Appendix \ref{app:Brauer}.

\begin{theorem}\label{thm:Sha-seq-nonDG}
        There is a commutative diagram with exact rows and columns
    \[
\begin{tikzcd}[column sep=small]
& & 
& 0
& 0
& & \\
& & 
& \rmH^2(B, \Lambda) \ar[r,"\simeq"] \ar[u]
& \rmH^2(B, \Lambda) \ar[u]
& & \\
0 \ar[r] & \rmH^0(B, \JJ) \ar[r] 
& \rmH^0(B, \wt{\JJ}) \ar[r] 
&\rmH^0(B, \rmR^4p_*\bbZ_{\sY}) \ar[u] \ar[r]  
& \rmH^1(B, \JJ) \ar[u] \ar[r] 
& \rmH^1(B, \wt{\JJ}) \ar[r] 
& 0 \\
0 \ar[r] & (\Sigma^\perp)^{2,2} \ar[r] \ar[u, "\simeq"]
& \rmH^{2,2}(X, \Z) \ar[r] \ar[u,  "\simeq"] 
& \dfrac{\rmH^4(X,\bbZ)}{\Sigma^\perp} \ar[r] \ar[u] 
& \Br_{\rm an}(\Sigma^\perp) \ar[r] \ar[u]
& \Br^4_{\an}(X) \ar[r] \ar[u, "\simeq"]
& 0\\
& & 
& 0 \ar[u]
& 0 \ar[u]
& & 
\end{tikzcd}
\]

\end{theorem}

In particular, we get an exact sequence involving
$\rmH^1(B,\calJ)^0\coloneqq \Ker\left(\rmH^1(B,\calJ)\to \rmH^2(B,\Lambda)\right)$:
\begin{equation}\label{eq:exseqH1J0}
    0 \to \rmH^4(X,\bbZ)/(\rmH^{2,2}(X,\bbZ)+\Sigma^\perp) \to \rmH^1(B,\calJ)^0 \to \Br^4_{\an}(X) \to 0.
\end{equation}
The group $\rmH^1(B,\calJ)^0$ is related to Tate--Shafarevich twists of certain Lagrangian fibrations, see Corollary \ref{cor:JisAM}. 

\begin{proof}

We start with a distinguished triangle \eqref{eq:distinguished-deligne}.
Via the Leray spectral sequence  
\[
E^{p,q}_2(-)\coloneqq \rmH^{p}(B,\rmR^qp_*(-)) \Longrightarrow \rmH^{p+q}(\calY,-)
\]
we obtain a commutative diagram with exact rows
involving the third term $\rmL_3\rmH^4(-)$ of the Leray filtration on $\rmH^4(-)$:
\[\begin{tikzcd}
	0 \ar[r] & \rmL_3\rmH^4(\bbZ_\calY(2)_\rmD) \ar[r] \ar[d] & \rmH^4(\bbZ_\calY(2)_\rmD) \ar[r] \ar[d] & E^{0,4}_\infty(\bbZ_\calY(2)_\rmD) \ar[d] \ar[r] & 0 \\
	0 \ar[r] & \rmL_3\rmH^4(\bbZ_\calY) \ar[r] \ar[d] & \rmH^4(\bbZ_\calY) \ar[r] \ar[d] & E^{0,4}_\infty(\bbZ_\calY) \ar[d] \ar[r] & 0 \\
	0 \ar[r] & \rmL_3\rmH^4(\Omega^{\leq 1}_\calY) \ar[r]  & \rmH^4(\Omega^{\leq 1}_\calY) \ar[r]  & E^{0,4}_\infty(\Omega^{\leq 1}_\calY)  \ar[r] & 0.
\end{tikzcd}\]

The spectral sequences giving the first two
rows
were
analyzed in \cref{thm:main-cohom-wtJJ} and \cref{thm:decomposition-cubic} respectively.
In particular, \cref{thm:decomposition-cubic}
implies that 
\[E^{0,4}_\infty(\bbZ_\calY)=K\coloneqq \Ker \left(\rmH^0(B,\rmR^4p_*\bbZ_\calY)\to \rmH^2(B,\Lambda)\right).\] 
On the other hand, the proof of \cref{thm:main-cohom-wtJJ} shows that 
$\rmL_3\rmH^4(\sY,\bbZ_\calY(2)_\rmD)\simeq \bbZ H^2\oplus \bbZ hH$ and this group 
maps as a direct summand to the respective terms in the top left square. 
From \cref{prop:pushforward-Omega1} we know that 
$E^{0,4}_\infty(\Omega^{\leq 1}_\calY) =0$.
Omitting the $\bbZ H^2\oplus \bbZ hH$ summand from the top left square we get the following diagram:

\begin{equation}\label{eq:threelerays}
\begin{tikzcd}
0 \ar[r] & 0 \ar[r] \ar[d] & \rmH^{2,2}(X,\bbZ) \ar[r,"\sim"] \ar[d] & \rmH^0(B,\wt{\calJ}) \ar[r] \ar[d] & 0\\
      0\ar[r] &   \rmH^1(B,\Lambda) \ar[r] \ar[d,"\varphi"] & \rmH^4(X,\bbZ) \ar[r]
      \ar[ur,phantom,"{\color{purple} \circled{1}}"] 
      \ar[d,"\psi"] & K \ar[d]  \ar[r] & 0\\
       0 \ar[r] &  \rmH^1(B,\rmR^3p_*\Omega^{\leq 1}_\calY) \ar[r,"\sim"] & \rmH^4(\calY,\Omega^{\leq 1}_\calY) 
       \ar[r] & 0 \ar[r] & 0
\end{tikzcd}
\end{equation}

Note that the map $\varphi$ is induced by (\ref{eq:LambdaOmegaJ}), in particular, $\Ker(\varphi)=\rmH^0(B,\calJ)$ and $\coker(\varphi)=\rmH^1(B,\calJ)^0\coloneqq \Ker \left(\rmH^1(B,\calJ)\to \rmH^2(B,\Lambda)\right)$. On the other hand, since $\rmH^5(\calY,\bbZ)=0$, we have 
$\coker(\psi)=\rmH^5(\calY,\bbZ_\calY(2)_\rmD)\simeq \rmH^1(B,\widetilde{\calJ})$. Applying the Snake Lemma to the bottom two rows of (\ref{eq:threelerays}) gives an exact sequence
$$0 \to \rmH^0(B,\calJ) \to \rmH^0(B,\widetilde{\calJ}) \to K \xrightarrow{\delta} \rmH^1(B,\calJ)^0 \to \rmH^1(B,\widetilde{\calJ}) \to 0.$$
where $\delta$ is the connecting morphism. 

Now, we identify $\rmH^1(B,\Lambda)\simeq \Sigma^\perp$ via Lemma \ref{lem:H1LambdaSigma}, $\rmH^4(\calY,\Omega^{\leq 1}_\calY)\simeq \rmH^3(X,\Omega^1_X)$ via the proof of Proposition \ref{prop:pushforward-Omega1}. 
 We obtain a diagram with exact rows
\begin{equation}\label{eq:bigdiagramSHA}
\begin{tikzcd}[column sep=small]
0 \ar[r] & \rmH^0(B, \JJ) \ar[r] 
& \rmH^0(B, \wt{\JJ}) \ar[r] 
& K \ar[r] 
& \rmH^1(B, \JJ)^0 \ar[r] 
& \rmH^1(B, \wt{\JJ}) \ar[r] 
& 0 \\
0 \ar[r] & (\Sigma^\perp)^{2,2} \ar[r] \ar[u, "\simeq"]
& \rmH^{2,2}(X, \Z) \ar[r] \ar[u, "\simeq"] 
\ar[ur,phantom,"{\color{purple} \circled{2}}"]
& \dfrac{\rmH^4(X,\bbZ)}{\Sigma^\perp} \ar[r] \ar[u, "\simeq"]
\ar[ur,phantom,"{\color{purple} \circled{3}}"  pos=0.4]
& \Br_{\rm an}(\Sigma^\perp) \ar[r] \ar[u, "\simeq"]
& \Br^4_{\an}(X) \ar[r] \ar[u, "\simeq"]
& 0
\end{tikzcd}
\end{equation}
To conclude the proof of the theorem, it remains to show that (\ref{eq:bigdiagramSHA}) is commutative. It is enough to show commutativity of the squares {\color{purple} \circled{$2$}} and {\color{purple} \circled{$3$}}. The former follows from the commutativity of {\color{purple} \circled{$1$}}, and the latter by construction of $\delta$.
\end{proof}

\begin{corollary}\label{cor:Sha-sequence}
Assume $X$ is defect general. We have 
\[\rmH^0(B,\sJ)\simeq \rmH^{2,2}(X,\bbZ)_{\pr}\ \text{ and }\ \rmH^1(B,\sJ)\simeq 
\Br_{\an,\pr}^4(X) \coloneqq
\BrAn(\rmH^4(X,\bbZ)_{\pr})\] and exact sequence
\begin{equation}\label{eq:exseq-twists}
\Z/3\Z \to \rmH^1(B, \JJ) \to 
\Br^4_{\an}(X) \to 0.
\end{equation}
The first map in \eqref{eq:exseq-twists} is injective if and only if $h^2$ has divisibility one in $\rmH^{2,2}(X,\bbZ)$.
\end{corollary}
\begin{proof}
Since $X$ is defect general, $\Lambda=\Lambda^\bullet$ by Theorem \ref{thm:decomposition-cubic}, so $\rmH^2(B,\Lambda)=0$ by Lemma \ref{lem:CohJCohLambda} and hence $\rmH^1(B,\calJ)=\rmH^1(B,\calJ)^0$. Moreover, in this case $\Sigma^\perp=(h^2)^\perp$. Thus, the first two isomorphisms
follow from \cref{thm:Sha-seq-nonDG}.
Define $d$ to be the positive integer such that
\[
\image(\rmH^{2,2}(X, \Z) \xrightarrow{\cdot h^2} \Z) = 
d\Z.
\]
Since $(h^2)^2 = 3$, the divisibility $d$ of $h^2$ can only be equal to $1$ or $3$.
The short exact sequence
and the final claim
follow from \eqref{eq:exseqH1J0}.
\end{proof}

The next result allows constructing torsion classes in 
$\rmH^1(B, \JJ)$.

\begin{lemma}\label{lem:Jm-cohomology}
Let $X$ be a defect general cubic fourfold. 
    For any $m \ge 1$ consider the sheaf of torsion abelian groups
    $\calJ[m]\coloneqq \Ker(\calJ \xrightarrow{\cdot m}\calJ)$.
    Then
    \begin{equation}\label{eq:Jm-cohom}
    \rmH^4(X, \Z/m)_{\pr} \coloneqq
    \Ker(\rmH^4(X, \Z/m) \xrightarrow{h^2} \Z/m)
    \simeq \rmH^1(B, \JJ[m]).
    \end{equation}
    Let $b \in \rmH^4(X, \Z)$
    be any class such that $\deg_X(h^2 \cdot b) = 1$, see  \eqref{eq:b-class}.
    Then
    the image of the first map in \eqref{eq:exseq-twists} is generated by the image of the element $\ol{h^2 - 3b} \in \rmH^4(X, \Z/3)_{\pr}$ under the composition of \eqref{eq:Jm-cohom} for $m = 3$ with the map
    \[
    \rmH^1(B, \JJ[3]) \to \rmH^1(B, \JJ)[3] \subset \rmH^1(B, \JJ).
    \]
    \end{lemma}

We will 
explain the role of the class $\ol{h^2 - 3b}$ in  \cref{lem:degree-twists}.
    
\begin{proof}
    Apply multiplication by $m$ and the Snake Lemma to the short exact sequence \eqref{eq:LambdaOmegaJ}
to obtain
\[
\JJ[m] \simeq \Lambda/m\Lambda.
\]
Since $\rmH^2(B,\Lambda) = 0$ for defect general $X$ (\cref{thm:cohom-Lambda}), we observe that
\[
    \rmH^1(B, \JJ[m]) \simeq 
    \rmH^1(B, \Lambda/m\Lambda) \simeq \frac{\rmH^1(B, \Lambda)}{m\rmH^1(B, \Lambda)} \simeq
    \frac{\rmH^4(X, \Z)_{\pr}}{m\rmH^4(X, \Z)_{\pr}} \simeq
    \rmH^4(X, \Z/m)_{\pr}.
\]
 For the final claim, we note that 
 the kernel of the map between analytic Brauer groups
 \[
 \BrAn(\rmH^4(X,\bbZ)_{\pr}) \to \BrAn(\rmH^4(X,\bbZ)).
 \]
 is generated by an explicit class 
 given in \cref{lem:Bprim}. This class is precisely the image of $\ol{h^2 - 3b}$.~\qedhere
\end{proof}

\section{OG10 Lagrangian fibration associated to a cubic fourfold}\label{sec:latticesHK}

We recall some basic results on the Decomposition Theorem for a Lagrangian fibration
and the relation to the Beauville--Bogomolov--Fujiki form, see \cref{prop:BBF-dCM}.
In \cref{thm:pairings} we compute this form for the OG10 Lagrangian fibration $M \to B = (\P^5)^\vee$
associated to a smooth cubic fourfold $X$.
This uses our results about the universal hyperplane sections of cubic fourfolds from \S \ref{sec:cohomologycomputations}, namely
\cref{thm:decomposition-cubic} and \cref{thm:cohom-Lambda}.

\subsection{Perverse filtration and BBF form on a Lagrangian fibration}

For basic facts about hyperk\"ahler manifolds and Lagrangian fibrations see \cite{HuycptHKbasic, Beauville-HK-survey} and references therein.
Recall that any hyperk\"ahler variety $M$
admits the so-called Beauville--Bogomolov--Fujiki (BBF) form $q$
on $\rmH^2(M, \Z)$ which satisfies
\begin{equation}\label{eq:Fujiki}
c_M \cdot q(u)^n = \int_M u^{2n}.
\end{equation}
Here $\dim M = 2n$
and
$c_M$ is the Fujiki constant.
There is a more general version of this relation: for a set of classes $\alpha_1, \ldots, \alpha_{2n} \in \rmH^2(M, \Z)$, we can apply \eqref{eq:Fujiki}
to a linear combination $u = t_1\alpha_1+\cdots + t_{2n}\alpha_{2n}$. Expanding both sides as polynomials in $t_1, \ldots, t_{2n}$
and comparing the 
coefficients of $t_1\cdots t_{2n}$
we obtain the sum over all permutations $\sigma \in \mathfrak{S}_{2n}$
\begin{equation}\label{eq:Fujiki-2}
 c_M\sum_{\sigma\in \mathfrak{S}_{2n}} q(\alpha_{\sigma(1)}, \alpha_{\sigma(2)})\cdots q(\alpha_{\sigma(2n-1)}, \alpha_{\sigma(2n)}) = (2n)!\int_M\alpha_1\cdots\alpha_{2n} 
\end{equation}

Let $f\colon M \to B$ be a Lagrangian fibration of relative dimension $n$ over $B \simeq \P^n$.
We will use \eqref{eq:Fujiki-2} to compute the BBF form on the primitive part of $\rmH^2(M, \Q)$, see Proposition \ref{prop:BBF-dCM}.

\medskip

We start by applying the Decomposition Theorem
\ref{thm:decomposition-perverse} to decompose
 $\rmH^2(M, \Q)$ into simpler pieces. 
 We have
\[
\rmH^2(M, \Q) = 
\rmH^{2-2n}(M, \Q[2n]) \simeq
\rmH^{2-n}(B, P_{-n}) \oplus
\rmH^{1-n}(B, P_{-n+1}) \oplus
\rmH^{-n}(B, P_{-n+2}).
\]
All other cohomology terms vanish by Lemma \ref{lem:IC-claim}.
We have $P_{-n} \simeq \Q[n]$ by Remark \ref{rmk:suppPi} and
for the remaining two complexes, we use the notation
\[
P^\bullet \coloneqq P_{-n+1}[-n], 
\quad
R^\bullet \coloneqq P_{-n+2}[-n]. 
\]
One can show that $P^\bullet$ is supported in cohomological degrees $0$ and $1$.
The splitting above can be rewritten~as
\begin{equation}
\label{eq:H2M-splitting}
\rmH^2(M, \Q) \simeq
\rmH^{2}(B, \Q) \oplus
\rmH^{1}(B, P^\bullet) \oplus
\rmH^{0}(B, R^\bullet).
\end{equation}
We write $\eta \coloneqq f^*c_1(\sO(1)) \in \rmH^2(M, \Q)$, so that 
by \eqref{eq:Fujiki} we have $q(\eta) = 0$.
The first term in decomposition \eqref{eq:H2M-splitting} is $\Q\eta$.
By \cite[Theorem 0.2]{ShenYin-PW}
the last piece is also one dimensional.
The splitting is not canonical
but we have the following  observation.

\begin{lemma}\label{lem:H2M-R}
The kernel of the natural morphism $\rmH^2(M, \Q) \to \rmH^0(B, R^\bullet) \simeq \Q$ equals $\eta^\perp$
and we have a canonical isomorphism
\[
\rmH^{1}(B, P^\bullet) \simeq \eta^\perp / \Q \eta.
\]
\end{lemma}
\begin{proof}
Let $W = \Ker(\rmH^2(M, \Q) \to \rmH^0(B, R^\bullet))$.
Consider the fiber $M_b$ over a general point $b \in B$.
As in \cite[Proof of Theorem 0.4(b)]{ShenYin-PW}
the  restriction of cohomology classes to the fiber factors as
\[
\rmH^2(M, \Q) \to \rmH^0(B, R^\bullet) \to 
\rmH^2(M_b, \Q)
\]
and so 
\begin{equation}\label{eq:WetaperpQ}
W \subset 
\Ker(\rmH^2(M, \Q) \to \rmH^2(M_b, \Q))
= \eta^\perp
\end{equation}
where the last equality is \cite[Lemma 2.2]{Matsushita-deformations}.
Since $W$ and $\eta^{\perp}$ have codimension one in $\rmH^2(M, \Q)$, these two subspaces coincide.
The final claim follows because $\rmH^1(B, P^\bullet)$ is a quotient of $W = \eta^\perp$ by the image of $\rmH^2(B, \Q)$ given by $\Q \eta$.
\end{proof}

Let $\theta \in \rmH^2(M, \Q)$ be any class
satisfying $q(\theta,\eta) = 1$.
Then we have
\[
\eta^\perp = \Q\eta \perp \langle \eta, \theta\rangle^\perp \text{ and }
\eta^\perp / \Q\eta = \langle \eta, \theta\rangle^\perp.
\]
We thank Radu Laza for explaining us the following result.

\begin{proposition}\label{prop:BBF-dCM}
The BBF form on $\langle \eta, \theta\rangle^\perp \subset \rmH^2(M, \Q)$ is given as follows:
\begin{equation}
\label{eq:BBF-dCM}
c_M \cdot q(u,v) = \binom{2n}{n}\frac{n}{2^n}  \int_M uv \theta^{n-1} \eta^{n-1}.
\end{equation}
Thus,  we have a similitude (that is an isomorphism which preserves the bilinear form up to scaling) of pure polarized rational Hodge structures
\begin{equation}\label{eq:theta-eta-perp}
\rmH^{1}(B, P^\bullet) \simeq \eta^\perp / \Q \eta \simeq \langle \eta, \theta\rangle^\perp. 
\end{equation}
\end{proposition}
\begin{proof}
In \eqref{eq:Fujiki-2} let us take
$
\alpha_1 = u$, $\alpha_2 = v
$ 
with $u, v \in \langle \eta, \theta\rangle^\perp$
and
\[
\begin{aligned}
\alpha_3 = \ldots = \alpha_{n+1} &= \theta
\\
\alpha_{n+2} = \ldots = \alpha_{2n} &= \eta \\
\end{aligned}
\]
to obtain
\[
c_M n! \cdot (n-1)! \cdot 2^n q(u,v) = (2n)! \int_X uv \theta^{n-1} \eta^{n-1}
\]
which is equivalent to \eqref{eq:BBF-dCM}.
The isomorphism \eqref{eq:theta-eta-perp}
coming from Lemma \ref{lem:H2M-R} is an isomorphism of pure Hodge structures \cite[Corollary 2]{Sai88}, and as we have just shown, it preserves the lattice structure up to a multiple.~\qedhere
\end{proof}

\begin{example}\label{ex:K3n-OG10-Fujiki}
If $M$ is of K3$^{[n]}$ or of OG10 type, then
the Fujiki constant
has the form $c_M = (2n)!/(2^n n!)$
\cite{Rapagnetta-OG10}
and
\eqref{eq:BBF-dCM}
simplifies to
\[
q(u,v) = \frac1{(n-1)!}  \int_M uv \theta^{n-1}\eta^{n-1}.
\]
\end{example}

\begin{remark}
Decomposition \eqref{eq:H2M-splitting} can be rewritten 
as in \cite[Proposition 1.1]{ShenYin-PW}
\[
\rmH^2(M, \Q) \simeq 
\Q\eta \oplus
\rmH^{1}(B, P^\bullet) \oplus
\Q\theta.
\]
If $\theta \in \rmH^2(M, \Q)$ is an ample class such that $q(\theta, \eta) = 1$ then 
this is the $(\eta, \theta)$-Lefschetz decomposition of de Cataldo--Migliorini \cite[Corollary 2.1.7]{dCM-HT}
and
the form 
\[
\rmS_{\theta,\eta} \colon (u, v) \mapsto \int_M uv \theta^{n-1} \eta^{n-1}
\]
in the right-hand side of \eqref{eq:BBF-dCM} is the de Cataldo--Migliorini polarization of the biprimitive piece $\rmH^1(B, P^{\bullet})$ of $\rmH^2(M, \Q)$ \cite[Theorem 2.1.8]{dCM-HT}.
\end{remark}

\subsection{Hodge lattice of OG10 Lagrangian fibration}

Given a smooth cubic fourfold $X$, one can associate a relative intermediate Jacobian to the family of at worst $1$-nodal hyperplane sections $\calY_{1} \to U_1\subset B\coloneqq (\bbP^5)^\vee$ of $X$, with a natural symplectic form on it \cite{DonMar, LSV}.
 
 \begin{definition}\label{def:HKcompIJ}
    By \textit{an OG10 Lagrangian fibration associated to} $X$, we mean 
    a Lagrangian fibration $f\colon M\to B$ of a 
    projective hyperk\"ahler manifold $M$ compactifying the relative intermediate Jacobian fibration of the family $\calY_1 \to U_1$ of hyperplanes of $X$ with at worst one nodal singularity.
 \end{definition}

 Such a compactification has been constructed in \cite{LSV} for general cubic fourfolds, and in \cite{Sacbirational} for all smooth cubic fourfolds. 
  As a consequence of the definition, the fibers of $f\colon M\to B$ are not strictly multiple in codimension $1$. 
 If $X$ is general, then there is an associated hyperk\"ahler Lagrangian fibration where all fibers of $f$ are integral.
 This generality condition on $X$ is made explicit in \cite{DutMarq}.

Let $\eta \in \rmH^2(M, \Z)$ be the pullback of the ample generator $H \in \rmH^2(B, \Z)$ so that $q(\eta) = 0$.
Following \cite[\S 3]{Sacbirational}, there is an ample
divisor class $\theta \in \rmH^2(M, \Z)$
such that $q(\theta, \eta) = 1$.
We think of $\langle\theta, \eta \rangle^\perp$ as the primitive part of $\rmH^2(M, \Z)$.
Since 
$\langle \theta, \eta \rangle \simeq U$ (the hyperbolic plane), see \cite[Lemma 3.5]{Sacbirational},
we have canonical isomorphisms
\begin{equation}\label{eq:H2M-U}
\rmH^2(M, \Z) \simeq \langle \theta, \eta \rangle^\perp \perp U, \quad
\eta^\perp / \Z \eta \simeq \langle \theta, \eta \rangle^\perp.
\end{equation}

\begin{theorem}\label{thm:pairings}
Let $X$ be a smooth cubic fourfold
and $f\colon M \to B$ be an associated OG10 Lagrangian fibration.
Let $\eta \coloneqq f^*H \in \rmH^2(M, \Z)$ be
the  isotropic class.
We have
a Hodge isometry
\begin{equation}\label{eq:M-X-lattice}
\rmH^4(X, \Z)_{\pr}(1) \simeq 
\eta^\perp / \Z\eta \simeq \langle \theta, \eta \rangle^\perp   \subset \rmH^2(M,\bbZ)
\end{equation}
where the Hodge structure and the
bilinear form
 on $\eta^\perp / \Z\eta$
are induced from $\rmH^2(M, \Z)$ with its BBF~form. In particular, $X$ is uniquely determined by the Lagrangian fibration $f \colon M \to B$.
\end{theorem}

Note that here $\rmH^4(X, \Z)_{\pr}(1)$
means that we apply the Tate twist to the Hodge structure and change the sign of the intersection form on $\rmH^4(X, \Z)$.

\begin{proof}
First, let us assume that $X$ is general. 
In this case, by Theorem \ref{thm:decomposition-cubic}
we have $\Lambda \simeq \Lambda^\bullet$, where $\Lambda^\bullet$ is a shift of a certain integral intersection complex on $B$. 
By construction of $M$, over the smooth locus $U \subset B$, we have $\rmR^1 f_* \Z_{M_U} \simeq \Lambda_U$.
When $X$ is general,
the fibers of $M$ are integral
so that by the Ng\^{o} support theorem \cite{Ngo}, \cite{AriFed}, \cite[Proposition 9.5]{ACLSRelativeLefschetz}, \cite[Corollary 5.2]{AnconaFratila},
$P^\bullet[5]$ is an intersection complex with full support on $B$;
in particular, 
\[
P^{\bullet}\simeq \IC(\rmR^1f_{U*}\bbQ_{M_U})[-5] \simeq
\IC(\Lambda_{U, \bbQ})[-5] \simeq \Lambda^\bullet_\Q \simeq \Lambda_\Q.
\]
and $P^\bullet$ is a sheaf.

In \cref{thm:cohom-Lambda} we have already observed the canonical isomorphism of Hodge structures
\begin{equation}
\label{eq:hodge-isom1}
\rmH^4(X, \Z)_{\pr} \simeq \rmH^1(B, \Lambda).
\end{equation}
To compare $\rmH^1(B, \Lambda)$ to $\rmH^2(M, \Z)$ we 
use the Leray spectral sequence for $\rmR f_*\Z_M$ as in \cite[\S4.2]{AbashevaRogov}.
Let
\[
W_\Z \coloneqq \Ker\left(\rmH^2(M, \Z) \to \rmH^0(B, \rmR^2 f_* \Z_M)\right)
\]
so that the spectral sequence gives
an isomorphism of Hodge structures
\begin{equation}\label{eq:hodge-isom2}
\rmH^1(B, \Lambda) \simeq W_\Z /\Z \eta.
\end{equation}
From \eqref{eq:WetaperpQ}, it follows that 
$W_\Z$ is a finite index subgroup of $\eta^{\perp_{\rmH^2(M, \Z)}}$.

We will prove that the two induced bilinear forms on $\rmH^1(B,\Lambda)$ match and $W_\Z\simeq \eta^{\perp_{\rmH^2(M,\bbZ)}}$.
To compare bilinear forms it suffices to work rationally. Tensoring \eqref{eq:hodge-isom1} and \eqref{eq:hodge-isom2} with $\Q$ we get
isomorphisms of
Hodge structures
\begin{equation}\label{eq:both-hodge-isom}
\rmH^4(X, \Q)_{\pr} \simeq \rmH^1(B, \Lambda_\Q) \simeq \eta^{\perp_{\rmH^2(M, \Q)}} / \Q \eta
\end{equation}

The Hodge structure 
$\rmH^1(B, \Lambda_\Q)$ has two natural bilinear forms:
the pairing from \cref{thm:cohom-Lambda} and the pairing from 
the right hand side of \eqref{eq:BBF-dCM}.
Both of them are induced by polarizations of a Hodge module structure on $\Lambda_\Q[5]$
and must be proportional because $\Lambda_\Q[5]$ is a simple perverse 
sheaf hence it admits a unique polarization up to scaling.
Explicitly we can compare the two pairings as follows. For any smooth hyperplane section $Y \subset X$ we have an isomorphism of Hodge structures $\rmH^1(J(Y), \Z)(-1) \simeq \rmH^3(Y, \Z)$
\cite{ClemensGriffiths}
preserving polarizations
\[
-\int_{J(Y)} uv \theta^4/4! = \int_{Y} uv
\]
for all $u, v \in \rmH^3(Y, \Z)$;
for the appearance of the minimal class $\theta^4/4!$ see
\cite[(1), (5)]{Voisin-ppav}.
Thus, the two polarizations on $\Lambda_\Q$ differ by a factor $-1/4!$ and
for all $u, v \in \rmH^1(B, \Lambda)$
we have:
\[
\langle u, v \rangle_X =
\int_\YY u v \cdot p^*H^4 =
 -\frac1{4!} \int_M uv \cdot \theta^4 \eta^4 =-\langle u,v \rangle_M.
\]
Here the first equality is by \cref{thm:cohom-Lambda} and the last equality comes from
Proposition \ref{prop:BBF-dCM}
and Example~\ref{ex:K3n-OG10-Fujiki}.

We deduce that composing \eqref{eq:hodge-isom1} with \eqref{eq:hodge-isom2} gives a Hodge isometry
\[
\rmH^4(X, \Z)_{\pr} \simeq W_\Z / \Z\eta(-1).
\]

Let us prove that
$W_\Z = \eta^{\perp_{\rmH^2(M, \Z)}}$.
Since the discriminant of $\rmH^2(M, \Z)$ equals three (see e.g. \cite{Rapagnetta-OG10}), the
same is true for $\eta^\perp/\Z\eta$,
hence
discriminant of the quadratic
form on $W_\Z/\Z\eta$ equals $3d^2$,
where $d$ is index of $W_\Z \subset \eta^\perp$.
On the other hand, discriminant of $\rmH^4(X, \Z)_{\pr}$ also equals three. Thus, we must have $d = 1$ and 
$W_\Z = \eta^{\perp_{\rmH^2(M, \Z)}}$. This concludes the proof when $X$ is general.

Now let us assume $X$ is any smooth cubic and take a family of cubic fourfolds
$$\pi_\calX\colon  \calX\to \Delta$$
over a disk $\Delta$ such that $\calX_t$ is sufficiently general
for $t\in \Delta^*$ and $\calX_0\simeq X$. Following \cite[Proposition 1.18]{Sacbirational}, up to a finite base change one can construct a family $\pi_\calM\colon \calM\to \Delta$ of hyperk\"ahler manifolds equipped with a fiberwise Lagrangian fibration $f\colon \calM\to \bbP^5_\Delta$
such that for all $t\in \Delta$, $\calM_t \to \bbP^5_{t}$ is associated to $\calX_t$ in the sense of Definition \ref{def:HKcompIJ}. 

Since $\Delta$ is simply connected, the local systems $\rmR^4_{\pr}\pi_{\calX *}\bbZ_{\calX}$ and $\rmR^2_{\pr}\pi_{\calM *}\bbZ_{\sM}$ are trivial; here the latter local system has fiber $\eta_t^{\perp}/\bbZ\eta_t$ at $t\in \Delta$. We constructed, for all $t\in \Delta^*$, a Hodge isometry $\varphi_t\colon \rmH^4(\calX_t,\bbZ)_{\rm pr}(1) \to \eta_t^\perp/\bbZ\eta_t$ and this construction is compatible with small deformation of $\XX_t$. In particular,
the fibers of the local systems at some fixed $t_0\in \Delta^*$ are isometric, and we denote the correspondence lattice by $L$.
We fix a marking 
$\gamma \colon \rmH^4(\calX_{t_0},\bbZ)_{\rm pr}(1) \xrightarrow{\sim} L$, which induces a marking on $\calX_t$ for all $t\in \Delta$. Via the isometry $\varphi_{t_0}$, we obtain a marking $\delta\colon \eta_{t_0}^\perp/\bbZ\eta_{t_0}\simeq L$, which induces one for all $t\in \Delta$. 
In particular, these markings induce two period maps
\[
\calP_\calX, \calP_\calM \colon \Delta\longrightarrow \calD 
\]
where $\calD\subset \bbP\left(L\otimes\bbC \right)$ denotes the period domain associated to $L$.
By construction the two maps
$\calP_\calX$ and $\calP_\calM$
coincide
on all 
sub-disks $\Delta'\subset \Delta^*$, and so 
they must coincide on 
$\Delta^*$.
By continuity, we have $\calP_\calX(0)=\calP_\calM(0)$, and we obtain a Hodge isometry
\[\rmH^4(X,\bbZ)_{\rm pr}(1) \xrightarrow{\sim} \eta_{0}^\perp/\bbZ\eta_{0}.\]

For the final statement, let $X$ and $X'$ be smooth cubic fourfolds, and
$M$ and $M'$ be the respective associated OG10 Lagrangian fibrations. If $M$ and $M'$ are isomorphic as Lagrangian fibrations, then from \eqref{eq:M-X-lattice} we deduce that $\rmH^4(X, \Z)_{\pr}$ and $\rmH^4(X', \Z)_{\pr}$ are Hodge isometric. A standard lattice argument shows that there exists a Hodge isometry $\rmH^4(X, \Z) \simeq \rmH^4(X', \Z)$ mapping $h^2$ to $(h')^2$.
 Voisin's Torelli theorem   
\cite{VoisinTorelliCubic, VoisinTorelliCubicErratum} implies
that $X$ and $X'$ are isomorphic.
\end{proof}

\begin{corollary}\label{cor:Br-M}
We have a Hodge isometry $\rmH^2(M, \Z) \simeq \rmH^4(X, \Z)_{\pr}(1) \perp U$ and an isomorphism
\[
\Br^2_{\an}(M) \simeq 
\Br^4_{\rm an, pr}(X).
\]
\end{corollary}

Here $\Br_{\rm an, pr}(X)\coloneqq \BrAn(\rmH^4(X,\bbZ)_{\pr})$, the analytic Brauer group of the K3 type Hodge structure $\rmH^4(X,\bbZ)_{\pr}$ as defined in Appendix \ref{app:Brauer}.

\begin{proof}
The first isomorphism follows from Theorem \ref{thm:pairings} and \eqref{eq:H2M-U}.
The second isomorphism follows from the first as we 
have \phantom{\qedhere}
\[
\BrAn(M) = \BrAn(\rmH^2(M, \Z)) \simeq \BrAn(\rmH^4(X, \Z)_{\pr}(1) \perp U) \simeq \BrAn(\rmH^4(X, \Z)_{\pr}).\tag*{\qed}
\]
\end{proof}

\begin{remark}
Consider the so-called
degenerate twistor deformation 
$f_\alpha\colon M_\alpha \to B$ of $f\colon M \to B$,
that is a Lagrangian fibration with the symplectic form
 $\sigma_\alpha = \sigma + \alpha \eta$, $\alpha \in \C$ 
\cite[Remark 4.6]{MarkmanLagrangian},
\cite[Theorem 1.10]{VerbitskyDegenerateTwistor}, 
\cite[\S2.3]{AbashevaRogov}. 
Then the primitive Hodge lattice $\eta^\perp / \Z\eta$
does not change
and the isomorphism $\rmH^4(X, \Z)_{\pr}(1) \simeq 
\eta^\perp / \Z\eta$ 
from Theorem \ref{thm:pairings} holds true for $M_\alpha$
instead of $M$. This is analogous  
to the case of Beauville--Mukai systems where $\eta^\perp/\Z\eta$ is the Hodge lattice of an algebraic K3 surface \cite[\S4]{MarkmanLagrangian}.

However, Corollary \ref{cor:Br-M} will not   hold for $M_\alpha$ with $\alpha \ne 0$, 
because $\theta$ 
will not be
a $(1,1)$-class on $M_\alpha$.
See \cite[Example 3.13]{Gx2Onorati} and \cite[Proposition 2.3]{MongardiOnorati-erratum} for the description of the
Hodge lattice $\rmH^2(M^T, \Z)$ of the Voisin twist $M^T$ defined in \cite{Voitwist} and recalled in \S \ref{sec:comparisonVoisin}.
\end{remark}

\section{Jacobian sheaves via Hodge modules}\label{sec:HM}

The main goal of this section is to justify $\sJ\coloneqq \rmR^2p_*\Omega_{\sY}^1/\rmR^3p_*\Z_\YY$ (Definition \ref{def:J}) 
as a suitable candidate for the Hodge theoretic degeneration of intermediate Jacobian of smooth cubic threefolds in $p\colon \sY\to B$.
For that, we reconstruct it as the intermediate
Jacobian sheaf of a Hodge module. Such Jacobian sheaves
have been defined by
Schnell \cite{SchNeron}, generalizing earlier work
on Jacobians of variations of Hodge structure (vHs)
and Hodge modules 
\cite{Zuc76, Clem83, Saiadmissible, GreenGriffithsKerr, BrosnanPearlsteinSaito}.
In particular, this shows that $\JJ$
is naturally reconstructed from the locus of smooth cubics $U \subset B$.
The main results of this section are
\cref{thm:HMreljac}
and \cref{thm:comparisonJandJM}.

\subsection{Preliminaries on Hodge modules}\label{sec:prelimHM}

We need to recall some important conventions and results from the theory of Hodge modules.
The  main reference for the theory of Hodge modules is Saito's work \cite{Sai88};
for a quick overview of the theory, consult \cite{Schoverview, PopaHM, SabSch}. For an overview of the general theory of $D$-modules, consult \cite{HTT08}.

The category of Hodge modules is constructed by a complicated inductive procedure. For our purposes,
it suffices to know that
a Hodge module on a smooth variety $B$
can be understood as a triple $(\MM, K, F_\bullet)$
where 
$\MM$ is a quasi-coherent sheaf with 
a (left) $D_B$-module structure
(here $D_B$ is the non-commutative algebra of differential forms on $B$),
$K$ is a $\bbQ$-perverse sheaf on $B$,
and
$F_{\bullet}\calM$ is an increasing Hodge filtration of $\MM$
by quasi-coherent subsheaves.
The connection $\nabla$
giving $\MM$ a structure of a $D_B$-module
needs to satisfy
$\nabla(F_k\sM) \subseteq \Omega_B^1\otimes F_{k+1}\calM$.
There are various other compatibility constraints; most importantly, there is
an isomorphism $\dr(\sM)\simeq K_\C$, where $\dr$ denotes the de Rham complex, and existence of a polarization, 
which we will not specify.

Any polarized variation of Hodge structure (vHs) 
$\HH = L \otimes \OO_B$ 
for a $\Q$-local system $L$
on a smooth variety $B$ 
gives an example of a Hodge module with 
$\MM = \HH$,
$K = L[\dim B]$, and
$F_{-k} \MM \coloneqq F^k \HH$.

In general, the structure of a Hodge module $(\MM, K, F_\bullet)$ can be thought of enriching the structure of the underlying perverse sheaf $K$. In particular, the six functor formalism on perverse sheaves extends canonically to Hodge modules.

 In what follows, we will simplify
the notation and write $\sM$ to denote a Hodge module, instead of the triple $(\sM, K, F_{\bullet})$, when it does not lead to confusion.
 It is worth mentioning that several sources use right $D_B$-modules, 
 and one needs to be careful when
 translating between the two notions;
 for example, the Hodge filtration gets shifted by $\dim B$ when passing between left and right Hodge modules.
 
\medskip

Below we list certain properties of Hodge modules that we will use. Let $B$ be a smooth variety of dimension $n$.

\begin{enumerate}
    \item \textbf{The minimal extension.} 
    Let $L_{U}$ be a polarized $\Q$-local system on a Zariski open subset $U \subset B$ which underlies
    a
    vHs $\HH_{U} = L_{U} \otimes \OO_B$.
    Then $\HH_U$ admits the so-called \emph{minimal extension} Hodge module $(\sM, \IC(L_{U}), F_\bullet)$. 
    Thus, we may think of $\IC(L_{U})$ as a perverse sheaf underlying the Hodge module $\MM$. 

Furthermore, any
Hodge module $\sM$ decomposes as $\sM\simeq \oplus_Z\sM_Z$, for proper subvarieties 
$Z\subseteq B$ and such that for each $Z$, $\MM_Z$ is a minimal extension of a vHs.   

    \item \textbf{Weight of a Hodge module.}
    The category of Hodge modules is graded by integer weights. 
    By Saito's conventions 
    (see \cite[Th\'eor\`eme 2]{Sai88}) 
    the minimal Hodge module 
    extension of a vHs of weight $m$
    has Hodge module weight $n + m$.

    \item \textbf{Filtered de Rham Complex.}  Let $(\sM,K,  F_{\bullet})$ be a Hodge module. The associated de Rham complex of $\sM$ is defined as
\[\dr(\sM)\coloneqq [\sM\to \Omega_B^1\otimes \sM\to \dots \to \omega_B\otimes \sM][n]\]
a complex living in degrees $-n,\dots, 0$, with the differential induced by the connection $\nabla$. The de Rham complex, $\dr(\sM)$ is quasi-isomorphic to the perverse sheaf $K_\C$.
Furthermore, we have the following filtered version of the de Rham complex
\[F_k \dr(\sM) = [F_k \sM \to \Omega_B^1 \otimes F_{k+1} \sM \to \dots \to \omega_B\otimes F_{k+n}\sM][n]\]
and the complex of coherent sheaves of $\sO_B$-modules
\begin{equation}\label{eq:grFM}
    \gr^F_k \dr(\sM) = [\gr^F_k \sM \to \Omega_B^1 \otimes \gr^F_{k+1} \sM \to \dots \to \omega_B\otimes \gr^F_{k+n}\sM][n]
\end{equation}
    \item  \textbf{Trivial Hodge module $(\sO_B, \bbQ_B[n], F_{\bullet})$.} This is nothing but the rank one constant variation of Hodge structure on $B$ of weight 0.
    It has Hodge module weight $n$.
    We have the Hodge filtration $F_k\sO_B = \sO_B$ if $k\geq 0$ and $0$ otherwise. The de Rham complex recovers the usual holomorphic de Rham complex
    \[
    \DR(\sO_B) = [\sO_B\to \Omega_B^1\to\dots\to \omega_B][n] \simeq \C_B[n].
    \]
    Furthermore, we have $\gr^F_{-k}\dr(\sO_B) \simeq \Omega_B^k[n-k]$ when $0 \leq k\leq n$ and $0$ otherwise. 
    \item\label{item:smallestpiece}\textbf{Lowest piece.} 
    Given any Hodge module $\sM$, following \cite[(2.1.4)]{SaiKolConj}, we denote by $p(\sM)$ the minimal  $p \in \Z$ with $F_{p} \sM \ne 0$ and $s(\sM) \coloneqq F_{p(\sM)} \sM$.
    By construction, we obtain
\begin{equation}\label{eq:smallestgr}
    \gr^F_{-n+p(\sM)}\dr(\sM) = \omega_B\otimes s(\sM)[0]
\end{equation}
Furthermore, let $\sM$ be the minimal extension of a vHs $\sH_U$, then (see  \cite[(2.25)]{SaiKolConj} for a similar expression for right $D_B$-modules) 
\begin{equation}\label{eq:gradingsmallestpiece}
    p(\sM) = -\max\{p\in \bbZ\mid F^p\sH\neq 0\}
\end{equation}
and the lowest piece $s(\sM)$ extends the smallest filtered piece $F^{-p(\sM)}\sH_U$.
For instance, when $\sM\simeq \sO_B$, $p(\sM) =0$, $s(\sM)\simeq \sO_B$ and $\gr^F_{-n}\dr(\sO_B)\simeq \omega_B[0]$.

\item \textbf{Strict support.} A Hodge module on $B$ is said to be have \textit{strict support along $B$} if $\sM$ is supported on $B$ and it does not have any nontrivial sub- or quotient-Hodge module supported along a proper subvariety. Such Hodge modules are uniquely determined by the associated vHs on some open subset of $B$.

   \item \textbf{Tate twist. }Given any Hodge module $\sM$ and $m \in \Z$, the Tate-twisted Hodge module $\sM(-m)$ 
   has the same $D_B$-module structure and the shifted Hodge filtration  
    \begin{equation}\label{eq:filtshift}
    F_{\bullet}\sM(-m) = F_{\bullet+m}\sM.
\end{equation}
Tate twist does not change the underlying perverse sheaf $K$. If $\sM$ has weight $w$, then
$\sM(-m)$ has weight $w+ 2m$.
Also note that by definition $p(\sM(-m)) = p(\sM) - m$ and
$s(\sM(-m)) = s(\sM)$.

In particular, we have the Tate-twisted Hodge module
$(\sO_B(-m), \bbQ_B[n], F_{\bullet})$.
This Hodge module is nothing but the variation of 1-dimensional trivial $(m,m)$-Hodge--Tate structure. We have $\gr^F_{-k}\sO_B(-m)\simeq \sO_B$ when $k=m$ and $0$ otherwise. Consequently, the graded pieces of the de Rham complex have the form
     \begin{equation}\label{eq:trivialtate}
    \gr^F_{-k}\dr(\sO_B(-m))\simeq 
    \Omega_B^{k-m}[n-k+m] \text{ for } m \le k \le m+n
    \end{equation}
    and are trivial outside of this range.

\item \textbf{Duality.} Let $\sM$ be a Hodge module of weight $w$,
then the dual $\scalebox{0.9}{$\bbD \sM\coloneqq \omega_B^{-1}\otimes R\sH om_{D_B}(\sM, D_B)$}$ is again a Hodge module (left $D_B$-module) of weight $-w$. 
Any polarization on $\sM$ induces an isomorphism $\bbD\sM\simeq \sM(w)$ which implies that
there is a quasi-isomorphism (see e.g.\ \cite[Lemma 7.4]{Schvanish})
\[\rmR\sH om_{\sO_B}(\gr^F_{-k}\dr(\sM),\omega_B[n])\simeq \gr^F_{k-w}\dr(\sM).\]
For instance, applying this to  $k = n-p(\sM)$ and using \eqref{eq:smallestgr}, we obtain a quasi-isomorphism 
\begin{equation}\label{eq:smallestdual}
    \rmR\sH om_{\sO_B}(s(\sM),\sO_B[n])\simeq \gr^F_{n-p(\sM)-w}\dr(\sM).
\end{equation}
\noindent Moreover, since $\gr^F_{-k}\dr(\sM) \simeq 0$ for %$k> n-p(\sM) 
$-k < -n + p(\MM)$, by duality we obtain
\begin{equation}\label{eq:gracyclic}
    \gr^F_{-k}\dr(\sM) \text{ is acyclic for all }
-k > n - p(\MM) - w.
\end{equation}

\end{enumerate}

\noindent The category of Hodge modules is semisimple. A Hodge module theoretic version of the decomposition \cref{thm:decomposition-perverse} exists and is one of the main inputs in what follows, so we recall it here.
\begin{citedthm}[{\cite[Th\'eor\`eme 5.3.1]{Sai88}}]\label{thm:HMdecomposition}
    Let $f\colon W\to B$ be a projective morphism of smooth varieties.
    Let $\sM$ be a Hodge module on $W$ of weight $w$. Then the Hodge module derived direct image $f_+ \sM$
splits non-canonically into the following direct sum
\[
f_+ \sM\simeq \bigoplus_{i\in \bbZ}\sP_i[-i],
\]
 where for all $i$, $\sP_i$ is a Hodge module on $B$ of weight $w+i$ with an underlying perverse sheaf structure given by $P_i$, defined in Theorem \ref{thm:decomposition-perverse}. Furthermore, there is an isomorphism $\sP_{-i}(-i)\simeq \sP_i$ induced via the relative Hard Lefschetz isomorphism.
\end{citedthm}

It is rather difficult to get hold  of the Hodge filtration on Hodge modules. Nonetheless, Saito in \cite[\S 2.3.7]{Sai88} proved a very useful property for the induced filtration on the de Rham complex (see also \cite[Theorem 28.1]{Schoverview}).
\begin{citedthm}[(Saito's Formula)]\label{thm:Saito}
    Under the assumptions of Theorem \ref{thm:HMdecomposition} for any $k$ we have 
    \[\rmR f_* \gr^F_k \DR(\sM) \simeq \bigoplus_{ i\in \bbZ}\gr_k^F \DR(\sP_i)[-i].\]
\end{citedthm}

Now we explain a corollary of Saito's formula. First we need the following.
\begin{lemma}\label{lem:saito-smallest}
    Under the assumptions of Theorem \ref{thm:HMdecomposition}, assume that $f$ is  surjective, and has connected general fibers of dimension $d$. Denote by $\sP_i'$ the direct summand of $\PP_i$ that does not have
    strict support along $B$. Then we have
    \[p(\PP_i')>\max\left(p(\MM)-d,\; p(\MM)-d-i\right).\]
\end{lemma}
\begin{proof}
        \cite[(2.6) Proposition]{SaiKolConj} translated into the language of left Hodge modules, implies that $p(\PP_{i}')>p(\MM)-d$ for all $i$. Moreover, using $\PP_i\simeq \PP_{-i}(-i)$ we have \phantom{\qedhere}
        \[
        p(\PP_i')=p(\PP_{-i}')-i>p(\MM)-d-i.\tag*{\qed}
        \]
\end{proof}

A similar formulation of the following statement can be found in \cite[Corollary 7.8]{KoldualI} and can be deduced from the more general result \cite[(3.2.2)]{SaiKolConj}. Here it is stated in a format useful to the current set-up.

\begin{corollary}\label{lem:R1fO}
Under the assumptions of Theorem \ref{thm:HMdecomposition}, assume that $f$ is  surjective, has connected general fibers of dimension $d$ and $\dim B = n$. Let
$\sM_{-d+1}$ be the Hodge module minimal extension of the vHs $\rmR^1 f_{U*} \Q_{W_U}$ where $U \subset B$ is the smooth locus of $f$. 
Then
\begin{equation}\label{eq:RkfO-formula}
\rmR^1 f_* \OO_W \simeq \sH^{-n}\left(\gr^F_0\dr(\sM_{-d+1})\right)\simeq
\cHom_B(s(\MM_{-d+1}), \OO_B),
\end{equation}
in particular, it is a reflexive sheaf.
\end{corollary}
\begin{proof}
To prove the first isomorphism in \eqref{eq:RkfO-formula}, we begin by using \cref{thm:HMdecomposition}  and \cref{thm:Saito} with $\sM=\sO_W$ and apply $\gr^F_0\dr(-)$ to both sides of the decomposition to get
    \begin{equation}\label{eq:R1f-sum-proof}
    \rmR^1 f_* \OO_W \simeq \bigoplus_{i\geq 0} \HH^{-n+i}(\gr^F_{0}\dr(\PP_{-d-i+1})).
    \end{equation}
    Indeed, by definition of $\gr^F_{k}\dr(-)$, for $i<0$, $\sH^{-n+i}\gr^F_{0}\dr(\PP_{-d-i+1})=0$.

Let us write $\PP_j \simeq \MM_j \oplus \PP_j'$ where $\MM_j$
has strict support along $B$ and $\PP_j'$
is supported on a proper closed subset of $B$. Here $\MM_{-d+1}$ is the Hodge module from the statement of the corollary
and $\MM_{-d} \simeq \OO_B$
by the connectedness
assumption on the fibers.
Furthermore, $\MM_{-d-i+1} = 0$
for $i > 1$
by the assumption on the dimension of the general fibers.

By \cref{lem:saito-smallest},
using $p(\OO_W) = 0$, we have $p(\PP_j') > -d - j$. 
Since the weight of $\sP_{j}'$ is $n+d+j$, \eqref{eq:gracyclic} implies that $\gr^F_0\dr(\PP_{j}')= 0$ for all $j \in \Z$. 
Thus, only the terms $i = 0$ and $i = 1$ can give nontrivial contribution to \eqref{eq:R1f-sum-proof}, hence we obtain
    \[
    \rmR^1 f_* \OO_W \simeq 
    \HH^{-n}\left(\gr^F_{0}\dr(\MM_{-d+1})\right)
    \oplus
    \HH^{-n+1}\left(\gr^F_{0}\dr(\OO_B)\right).
    \]
Using \eqref{eq:trivialtate}, we get
\[\sH^{-n+1}(\gr^F_0\dr(\sO_B))\simeq \sH^{-n+1}(\sO_B[n]) = 0
\]
and the first isomorphism in \eqref{eq:RkfO-formula} follows.

For the second isomorphism, assume that $\MM_{-d+1} \ne 0$, otherwise there is nothing to prove as both sides are trivial.
The Hodge module $\MM_{-d+1}$ 
has weight $n + 1$
and we have 
$p(\MM_{-d+1}) = -1$ by \eqref{eq:gradingsmallestpiece}.
Thus, the second isomorphism in \eqref{eq:RkfO-formula} follows from  duality \eqref{eq:smallestdual}\phantom{\qedhere}
\[
    \gr^F_{0}\dr(\sM_{-d+1}) \simeq \rmR\sH om(s(\sM_{-d+1}),\sO_B)[n].\tag*{\qed}
\]
\end{proof}

\subsection{Jacobians of Hodge modules}
\label{sec:SheafLagFibr}

Before explaining Schnell's definition of a Jacobian of a Hodge module we recall the classical Jacobian sheaf of a vHs. 
Let $U \subset B$ be an open subset of a smooth projective variety $B$ and $L_U$ be $\Z$-local system underlying a polarized vHs $\sH_U$
on $U$.
We assume that $\sH_U$ has weight $2k+1$ and level 1, i.e., the Hodge filtration is a two-step decreasing filtration with $F^i = \sH_U$ for all $i\leq k$ and $0$ for all $i> k+1$.
In this set-up, one has a family of polarized abelian varieties $J(\sH_U)\to U$
whose sheaf of sections is the Jacobian sheaf $\sJ(\sH_U) \coloneqq (F^{{k+1}})^{\vee}/{L_U}$. 

In \cite{SchNeron}, Schnell constructed an analytic space $J(\sM)\to B$ which extends $J(\sH_U)$, 
where $\sM$ is  the minimal extension as Hodge module  of the vHs $\sH_U$. 
Denote by $j\colon U\into B$  the open immersion. 
Let $s(\sM)$ denote the lowest graded piece of the Hodge filtration on $\sM$ \eqref{eq:smallestgr}. Schnell observed that 
\[
\rmR ^0j_*L_U\into s(\sM)^{\vee} = \cHom_B(s(\sM), \OO_B)
\]
is an injection (because it is generically injective and the source of the morphism is a torsion-free sheaf).
He defined the analytic space $J(\sM)$, as the total space of the analytic sheaf in the following definition.

\begin{definition}\label{def:schnellneron}
    The \textit{Jacobian sheaf of a Hodge module}
    given by the minimal extension $\sM$ of a vHs $\sH_U$ of odd weight and level 1 is defined as
\[\sJ(\sM) \coloneqq  \dfrac{s(\sM)^{\vee}}{\rmR^0j_*L_U}\]
where $s(\sM)$ is the lowest graded piece of the Hodge module $\sM$. 
\end{definition}

Note that $\sJ(\sM)$ depends on the integral structure on the Hodge module $(\sM, K, F_\bullet)$, that is the $\Z$-local system $L_U$, or equivalently, the integral intersection complex $K_\Z = \IC(L_U)$ satisfying $K_\Z \otimes \Q \simeq K$.
It is functorial with respect to certain pullbacks \cite[Prop 2.22]{SchNeron}.

We now give some explicit examples of Jacobians for some Hodge modules arising from geometry.
Recall that, given a projective morphism $f \colon W\to B$ of smooth varieties with connected fibers, by the exponential sequence, one can deduce an injective morphism of sheaves $\rmR ^1f_*\bbZ_{W}\into \rmR ^1f_*\sO_{W}$. We denote $\Lambda_{W}\coloneqq \rmR^1f_*\bbZ_{W}$. 

\begin{definition}\label{def:reljac}
    The \textit{relative Jacobian sheaf} for a projective morphism $f \colon W\to B$ of smooth varieties with connected fibers is defined as the analytic sheaf 
    \[\sJ_{W/B}\coloneqq \dfrac{\rmR^1f_*\sO_{W}}{\Lambda_{W}}.\] 
\end{definition}

\begin{remark}\label{rmk:pic}
    Note that, when $f$ has integral fibers, $\sJ_{W/B}$ is the analytic sheaf of sections of the scheme of relative algebraically trivial line bundles $\Pic^0_{W/B}$.
\end{remark} 

\begin{theorem}\label{thm:HMreljac}
    Let $f\colon W\to B$ be a surjective projective morphism between smooth varieties. 
    Assume that $f$ has connected fibers. 
    Let $\sM$ denote the minimal extension of the vHs corresponding to $\Lambda_{W_U}\coloneqq R^1f_{U*}\bbZ_{W_U}$, where $U$ is the locus parameterizing the smooth fibers of $f$. Then 
    \begin{equation}\label{eq:JMB}
        \sJ(\sM) \simeq \dfrac{\rmR^1f_*\sO_W}{\rmR^0j_*\Lambda_{W_U}}.
    \end{equation}
    If $f$ is flat and has no strictly multiple fibers in codimension one, then we have an isomorphism  
    \begin{equation}\label{eq:JMBPic}
        \sJ(\sM)\simeq \sJ_{W/B},
    \end{equation}
    where $\sJ_{W/B}$ is the relative Jacobian sheaf as defined in \cref{def:reljac}.
\end{theorem}

\begin{proof}

   The statement 
   \eqref{eq:JMB} follows from \cref{lem:R1fO}.
      For the second part, note from Corollary \ref{thm:R1f-tcf-cor} that $\Lambda_W$ is a tcf sheaf, in particular, $\rmR^0j_*\Lambda_{W_U}\simeq \Lambda_W$. Then \eqref{eq:JMBPic} follows from definition.~\qedhere
\end{proof}

\begin{example}\label{ex:groupschemelagr}
 Let
 $f\colon M\to B$ be a Lagrangian fibration of a hyperk\"ahler manifold $M$ of dimension $2n$ with no strictly multiple fibers in codimension one. Let $\sM$ be defined as the minimal extension  of the vHs $\Lambda_{M_U}\coloneqq \rmR^1f_{U*}\bbZ_{M_U}$, where $U$ is, as usual, the locus over which the fibers of $f$ are smooth.
 Since $B$ is smooth and $f$ is equidimensional, $f$ is flat. Therefore by \cref{thm:HMreljac} we have,
 \begin{equation}\label{eq:JM}
        \sJ(\sM) \simeq \sJ_{M/B}\simeq \dfrac{\Omega_B^1}{\Lambda_M}.
    \end{equation}
    Indeed, the isomorphism $\rmR^1f_*\sO_M\simeq \Omega_B^1$ follows from \cite[Theorem 1.3]{Matdirect}. 
    Hodge modules associated to higher direct images under Lagrangian fibrations of hyperk\"ahler manifolds and remarkable symmetries of their Hodge filtrations were studied recently in \cite{ShenYin-PW, ShenYin, MSY23, schHK}.
\end{example}

Now we consider an OG10 Lagrangian 
fibration $f\colon M \to B$ 
associated to a fixed smooth cubic fourfold $X$, see \cref{def:HKcompIJ}.
Recall that in \S \ref{sec:cohomologycomputations}
we have introduced and studied
a constructible sheaf $\Lambda= \rmR^3p_*\bbZ_{\sY}$ where $p\colon \sY \to B$
is the universal hyperplane section of $X$.
As usual, we consider $U\subset B$, the open locus over which the fibers of $p$ (and $f$) are smooth.

 \begin{theorem}\label{thm:comparisonJandJM}
   Given a smooth cubic fourfold $X$, let
   $\sM$ be 
   the minimal extension of the vHs  $\Lambda_{U}= \rmR^3p_{U*}\bbZ_{\sY_U}$. 
   Then we have
        \[\sJ(\sM)\simeq \sJ = \dfrac{\rmR^2p_*\Omega_{\sY}^1}{\Lambda}.\]
Furthermore, for an associated OG10 Lagrangian fibration $f\colon M\to B$ we  have $\sJ_{M/B}\simeq \sJ$.
\end{theorem}
\begin{proof}
The Hodge module $\MM$ has weight $8$ and by \eqref{eq:gradingsmallestpiece} it satisfies
$p(\MM) = -2$. To compute $\JJ(\MM)$ we first verify that $s(\MM)^\vee \simeq \rmR^2p_*\Omega_{\sY}^1$. 
By duality \eqref{eq:smallestdual} we can write
\begin{equation}\label{eq:Jnumerator}
    \rmR\sH om(s(\sM), \sO_B[5]) \simeq \gr^F_{-1}\dr(\sM).
\end{equation}
Now we compute the the graded piece in the right-hand-side of
\eqref{eq:Jnumerator} using Saito's \cref{thm:Saito}.
   The Hodge module theoretic version of the decomposition in \cref{thm:decomposition-cubic} is given as follows:
\[
p_+\sO_{\YY} \simeq \sO_B[3] \oplus \sO_B(-1)[1]\oplus \sM \oplus \sO_B(-2)[-1]\oplus \sO_B(-3)[-3].
\]
Indeed, the trivial pieces in \cref{thm:decomposition-cubic} generically underlie a Hodge--Tate structure and hence as Hodge modules they are isomorphic to appropriate Tate twists of the trivial Hodge module $\sO_B$. 
Considering $\gr^F_{-1}\dr(-)$ in Saito's \cref{thm:Saito} and using \eqref{eq:trivialtate}, we obtain the following quasi-isomorphism:
\[\rmR p_*\Omega^1_{\YY}[7] \simeq \Omega_B^1[7]\oplus \sO_B[6]\oplus \gr^F_{-1}\dr(\sM).\]
We deduce
\begin{equation}
\rmR^2p_*\Omega_{\sY}^1 \simeq
\sH^{-5}\left(\gr^F_{-1}\dr(\sM)\right) \simeq s(\MM)^\vee
\end{equation}
where we have used \eqref{eq:Jnumerator} in the last isomorphism. Since $\Lambda$ is a tcf sheaf (see \cref{thm:decomposition-cubic})
we get the desired description of $\JJ(\MM)$.

    Let $\sM_M$ denote the minimal extension of the vHs corresponding to $\Lambda_{M_U}$.
    From \cite[Lemma 3.2]{Voitwist}, one can deduce that the two vHs corresponding to the local systems $\Lambda_U$ and $\Lambda_{M_U}$ are isomorphic up to a Tate twist.
     Therefore, we get an isomorphism of minimal extensions
    \[\sM(1)\simeq \sM_M.\]
    which induces $\sJ(\sM)\simeq \sJ(\sM_M)$. 
    Since by construction of $f$ the fibers are not strictly multiple in codimension $1$, from \cref{ex:groupschemelagr} we have $\sJ(\sM_M)\simeq \sJ_{M/B}$. 
\end{proof}

\begin{remark}
    Comparing the lowest pieces of the Hodge modules $\sM_M\simeq \sM$ above, one can deduce $\rmR^2p_*\Omega_{\sY}^1\simeq \Omega_B^1$, which was  proved directly in \cref{prop:pushforward-Omega1}. 
\end{remark}

Further interpretations of the Hodge module theoretic Jacobian sheaf in the 
relative Prym setting
is a work in progress and will appear in a separate article. A comparison with the Fano variety of lines set-up appears in \cite[Proposition 8]{DutMarq}.

\section{Twists of  hyperk\"ahler Lagrangian fibrations}\label{sec:twist}

We recall the Tate--Shafarevich group of a Lagrangian fibration, and fully compute this group for the OG10 Lagrangian fibration associated to a smooth cubic fourfold. This relies on computations in \S \ref{sec:CohandPushforwards}.
Furthermore we identify Voisin's twisted OG10 Lagrangian fibration \cite{Voitwist} with a degree twist.
The main results of this section are
\cref{thm:JMBisAMppHK},
\cref{cor:JisAM},
\cref{cor:ShaseqSmoothCubic}
and \cref{thm:voisin=DMS}.

\subsection{The Albanese sheaf and Tate--Shafarevich twists}\label{sec:TateShaTwists}

Let $M$ be a projective hyperk\"ahler manifold of dimension $2n$ and $f\colon M\to B\coloneqq\bbP^n$ be a Lagrangian fibration. We assume that $f$ has no strictly multiple fibers in codimension $1$ (see Definition \ref{def:strictly-mult}).
By a result of Campana \cite[Remarque 1.3]{CampanaMultipleFibers} (see also \cite{HOCharFoliation}), this is equivalent to
the following assumption: 
away from codimension two subset of the base,
every fiber of $f$ has a reduced component. 

We denote by $T_{M/B}$ the relative tangent sheaf $(\Omega_{M/B})^{\vee}$. 
The image of the exponential map \cite{Markushevich,AbashevaRogov}
\begin{equation}\label{eq:exp}
    \exp\colon f_*T_{M/B}\to \Aut_{M/B}
\end{equation}
is a sheaf of abelian groups $\calA_{M/B}$ over $B$, called \textit{relative Albanese sheaf} of $M$. 
 This sheaf was considered by \cite{Markushevich} and \cite{AbashevaRogov} in the analytic setting,
 and in the algebraic setting by
 \cite{AriFed} when fibers are integral, 
 and for non-multiple fibers by \cite{Yoonjoo}. See also
 \cite{AbashevaSTII}, \cite{Sac25} for
 other work on the relative Albanese sheaf.
 
Consider the induced exact sequence
\begin{equation}\label{eq:exseq-expAM}
    0 \to \Gamma_{M} \to  f_*T_{M/B} \to \calA_{M/B} \to 0
\end{equation}
where $\Gamma_{M}\coloneqq \Ker(\exp)$ is a sheaf of torsion-free finitely generated abelian groups on $B$
\cite[Lemma 3.1]{AbashevaRogov}.
The restriction of $\Gamma_M$ to the smooth locus $U \subset B$ of $f$ can be described as follows.
For every $b \in U$ we have $\Gamma_{M,b} = \rmH_1(M_b, \Z) \simeq \rmH^{2n-1}(M_b, \Z)$, and this isomorphism is compatible with small deformation of $b \in U$. Thus, we have a canonical isomorphism of local systems
\begin{equation}\label{eq:GammaU-R2n-1}
\Gamma_M|_U \simeq \rmR^{2n-1} f_{U*} \Z_{M_U}.    
\end{equation}

Note that by construction, $\calA_{M/B}$ comes with an action $\calA_{M/B}\times_B M\to M$ over $B$. The following is \cite[Proposition  2.1]{Markushevich}, see also \cite[Theorem  2]{AriFed}.

\begin{proposition}\label{prop:smlocus-torsor}
  If not empty, 
  the smooth locus $M_b^{\rm sm}$ of any irreducible component of a fiber $M_b$ over $b$ is a torsor under the action of $\calA_{M,b}$. In particular, when all fibers of $f$ are integral and $f$ has a section, $\calA_{M/B}$ identifies with the sheaf of sections of the smooth locus $M^{\rm sm}$ of $f$.
\end{proposition}

The proof of the following result
was explained to us by Andrey Soldatenkov.

\begin{proposition}\label{prop:Gamma-tcf}
The sheaf $\Gamma_M$ is a tcf constructible sheaf. 
\end{proposition}
\begin{proof}
    Let us write $i \colon D \into B$ for the closed embedding of the discriminant locus of $f$
    and $j\colon U \into B$ for the open embedding of the complement $U = B \setminus D$.
    To check that $\Gamma_M$ is tcf everywhere on $B$ it suffices to prove that 
    \begin{equation}\label{eq:Gamma-tcf-vanishing}
    \HH^0 i^! \Gamma_M = \HH^1 i^! \Gamma_M = 0.
    \end{equation}
    Indeed, by definition this would imply that $\Gamma_M$ is tcf with respect to $D$, and then since $\Gamma_M|_{U}$ is a local system, \cref{lem:tcf-local-system} and \cref{lem:tcf-support} imply that $\Gamma_M$ is tcf everywhere on $B$. 
    
    Since $\Gamma_M$ is a subsheaf of $\Omega^1_B$,
    it is a torsion-free sheaf which implies vanishing of the first term in \eqref{eq:Gamma-tcf-vanishing}. 
    Let us compare the sheaf $\Gamma_M$ with $\Lambda_M = \rmR^1 f_* \Z$.
    By \cite[Proposition 4.4]{AbashevaRogov} we know that $\Gamma_M \otimes \Q \simeq \Lambda_M \otimes \Q$, and the latter is a tcf sheaf by \cref{lem:Q-R1-tcf}.
    Therefore, $\HH^1 i^! \Gamma_M$ is a $\Z$-torsion sheaf.
    Consider the diagram 
    \[
    \begin{tikzcd}
    0 \ar[r] & \Gamma_M \ar[r] \ar[d, hook] & R^0 j_*(\Gamma_M|_U) \ar[r]  \ar[d, hook] & \HH^1 i^! \Gamma_M \ar[r]  \ar[d, "=0"] & 0 \\
    0 \ar[r] & \Omega^1_B \ar[r] & R^0 j_*\Omega^1_U \ar[r] & \HH^1 i^! \Omega^1_B \ar[r] & 0
    \end{tikzcd}
    \]
    where the vertical map on the right is zero as it goes from a $\Z$-torsion sheaf to a sheaf of $\C$-vector spaces. 
    By the Snake Lemma, $\HH^1 i^! \Gamma_M$ is a subsheaf of $\AA_{M/B}$, so
    $\HH^1 i^! \Gamma_M = 0$ as $\AA_{M/B}$ is a torsion-free sheaf by definition.

    Finally, $\Gamma_M$ is constructible
    because of the isomorphism $\Gamma_M \simeq \rmR^0 j_* \Gamma_M|_U$, as $\Gamma_M|_U$ is a local system.
\end{proof}

\begin{remark}
Using Proposition \ref{prop:Gamma-tcf} it is easy to see that the Albanese sheaf $\AA_{M/B}$ is isomorphic to the Jacobian sheaf of the Hodge module $\NN$, given by the minimal extension of the vHs induced by \eqref{eq:GammaU-R2n-1} with the underlying integral structure of the
%corresponding to the 
local system $\Gamma_M|_U$. See \cref{thm:JMBisAMppHK} for the arguments that are used in this computation.
\end{remark}

As in \cite{AbashevaRogov} and \cite{MarkmanLagrangian}, we introduce the following.

\begin{definition}\label{def:ShaLagFib}
    We denote
    $$\Sha(M/B)\coloneqq \rmH^1(B,\calA_{M/B})$$
    and we call it the \textit{Tate--Shafarevich group} of the Lagrangian fibration $M\to B$.
\end{definition}

Given a class $\alpha\coloneqq \{U_{ij},\alpha_{ij}\}_{i,j}\in \rmH^1(B,\calA_{M/B})$ we obtain a new complex manifold
$M_\alpha$  by regluing the restrictions $M_{U_i}$ via the action of $\alpha_{ij}$ on the overlaps $U_{ij}$, together with a map $f_\alpha\colon M_\alpha \to B$. By construction, $M$ and $M_\alpha$ are analytic-locally isomorphic over $B$. We call $M_\alpha$ a \textit{Tate--Shafarevich} twist of $M$. 

The long exact sequence associated to (\ref{eq:exseq-expAM}) reads
$$\rmH^1(B,\Gamma_{M}) \to \rmH^1(B,\Omega^1_B) \xrightarrow{\psi} \Sha(M/B) \to \rmH^2(B,\Gamma_{M}) \to 0.$$
Given $t\in \rmH^1(B,\Omega^1_B)\simeq \bbC$, we set $M_t\coloneqq M_{\psi(t)}$,  and one can construct a family
\begin{equation}\label{eq:degtw-family}
\calM\to \bbC
\end{equation}
which is the so-called \textit{degenerate twistor family} introduced by Verbitsky in \cite{VerbitskyDegenerateTwistor}, see \cite[Theorem  3.9]{AbashevaRogov}.

Following \cite{MarkmanLagrangian} and \cite{AbashevaRogov} we consider the ``connected component''
\[
\Sha^0(M/B) \coloneqq \im(\psi) = \Ker(\Sha(M/B) \to \rmH^2(B,\Gamma_{M})).
\]
The next result follows from \cite[Theorem  5.19]{AbashevaRogov}, \cite[Theorem  4.7]{VerbitskySoldatenkovKaehlertwists}, \cite[Theorem A]{AbashevaSTII} (see also \cite{MarkmanLagrangian} for twists of moduli spaces of sheaves on K3~surfaces). 

\begin{citedthm}
    The period map $\calP\colon \bbC \to \calD$ associated to (\ref{eq:degtw-family}) is injective, where $\calD\subset \bbP(\rmH^2(M,\bbC))$ is the period domain of $M$. Furthermore, 
    for all $\alpha\in  \Sha^0(M/B)$, the twist $M_\alpha$ is a hyperk\"ahler manifold. If $\alpha$ is torsion, then $M_\alpha$ is projective.
\end{citedthm}

    Note that if $\alpha$ is not in the image of $\psi$, it can be non-K\"ahler, see \cite[Remark 5.3.13]{AbashevaSTII}.

\subsection{Comparison: \texorpdfstring{$\calJ$}{J} and \texorpdfstring{$\calA_{M/B}$}{A}}\label{sec:compareJJAM}

Recall from \cite[Proposition 4.3]{WieneckPolType} that for any hyperk\"ahler Lagrangian fibration $f\colon M\to B$, there exists a K\"ahler class $\omega\in \rmH^2(B,\bbQ)$ which restricts to an integral primitive polarization on the smooth fibers of $f$. Such a class is called a \textit{special K\"ahler class}. The polarization type $d(\omega)=(d_1,\dots,d_n)$ is the same on all smooth fibers, and it does not depend on the choice of $\omega$. A special K\"ahler class $\omega$ is called \textit{principal} if it is of type $(1,\dots,1)$, i.e. if it restricts to a principal polarization on all smooth fibers.

For the following theorem, recall that $\Lambda_M=\rmR^1f_*\bbZ_M$ and $\calJ_{M/B}$ is the relative Jacobian sheaf introduced in Definition \ref{def:reljac}.

\begin{citedthm}[(= \cref{intro:JMBisAMppHK})]\label{thm:JMBisAMppHK}
    Let $M$ be a projective hyperk\"ahler manifold and $f\colon M\to B=\bbP^n$ be a Lagrangian fibration  with no strictly multiple fibers in codimension one.  
    Let $\omega\in  \rmH^2(M,\bbQ)$ be a special K\"ahler class. There is an injection
        $\widetilde{\omega}\colon \Lambda_M\hookrightarrow \Gamma_{M}$
    and an exact sequence
    $$0 \to \Gamma_{M}/\widetilde{\omega}(\Lambda_M) \to \calJ_{M/B} \to \calA_{M/B} \to 0.$$
    If $\omega$ is principal, then
    $$\Lambda_M\simeq \Gamma_{M}\ \text{ and  }\
    \calJ_{M/B}\simeq \calA_{M/B}.$$
    
\end{citedthm}

\begin{proof}
The exponential map \eqref{eq:exp} fits in the following diagram
    \begin{equation}\label{eq:compJMandAM}
        \begin{tikzcd}
        0 \ar[r]& \Lambda_{M} \ar[r,"j"] \ar[d,"\widetilde{\omega}"] & \Omega^1_B \ar[r] \ar[d,"\simeq"',"(f_*\sigma)"] & \calJ_{M/B} \ar[r] \ar[d,"\varphi"] & 0 \\
        0 \ar[r] & \Gamma_{M} \ar[r] &  f_*T_{M/B} \ar[r, "\exp"] & \calA_{M/B} \ar[r] &  0.
    \end{tikzcd}
    \end{equation}
    Indeed, \cite[Theorem 2.6]{AbashevaRogov} showed that
    the symplectic form $\sigma\in \rmH^0(M,\Omega_M^2)$ induces an isomorphism
    $$f_*\sigma\colon \Omega^1_B \xrightarrow{\sim} f_*T_{M/B}.$$
    The morphism $j$ is defined as the composition of the natural injection $\Lambda_M=\rmR^1f_*\bbZ_M \hookrightarrow \rmR^1f_*\calO_M$ with Matsushita's isomorphism $\rmR^1f_*\sO_M\simeq \Omega_B^1$ \cite{Matdirect}. It is proved in \cite[Proposition 4.1]{AbashevaRogov} that the composition of $f_*\sigma \circ j$ sends $\Lambda_M$ to $\Gamma_{M}$: over a point $b\in B$ with $M_b$ smooth, the map $\widetilde{\omega}_b$ identifies with the natural map induced by the polarization
    $$\begin{tikzcd}
        \Lambda_{M,b}= \rmH^1(M_b,\bbZ)  \ar[d,hookrightarrow] \ar[r,"{\widetilde{\omega}_b}"]& \rmH_1(M_b,\bbZ) = \Gamma_{M,b}\ar[d,hookrightarrow] \\
        (\rmR^1f_*\calO_M)_b=\rmH^1(M_b,\calO_{M_b}) \ar[r,"\omega_b"] & \rmH^0(M_b,T_{M_b})=(f_*T_{M/B})_b
    \end{tikzcd}$$
    The first claim of the theorem follows. For the second, note that when $\omega$ is principal, then $\widetilde{\omega}_b$ is an isomorphism for all $b$ with $M_b$ smooth. In particular, $\widetilde{\omega}\colon \Lambda_M\to \Gamma_{M}$ is an isomorphism over the dense Zariski open subset of $B$ parameterizing  smooth fibers.
Since $\Lambda_M$ is tcf by  \cref{thm:R1f-tcf-cor} and $\Gamma_M$ is tcf by \cref{prop:Gamma-tcf}, $\widetilde{\omega}$ must be an isomorphism. Hence $\varphi$ is an isomorphism as well and the theorem is proved.
\end{proof}

\begin{remark}
    In the set-up of Theorem \ref{thm:JMBisAMppHK}, if $\omega$ is principal, and furthermore, the fibers of $f$ are integral and $f$ admits a section, then it can be shown using Remark \ref{rmk:pic} that 
    \[\Pic^0_{M/B}\simeq M^{\rm sm}.\]
    This generalizes \cite[Theorem B]{Arinkin-autoduality}.
\end{remark}

Recall that for a
smooth cubic fourfold $X$, we can
consider an
associated OG10 Lagrangian fibration, see Definition \ref{def:HKcompIJ}. To compute its associated Tate--Shafarevich group, we use the 
sheaves that we introduced in \S \ref{sec:cohomologycomputations}, namely the
constructible sheaf $\Lambda$ on $B$ (see (\ref{eq:def-LambdaYY})) and the sheaf of abelian groups $\calJ\simeq \Omega^1_B/\Lambda$ (Definition \ref{def:J}).

\begin{corollary}\label{cor:JisAM}
    Given a smooth cubic fourfold $X$ and an
    associated OG10 Lagrangian fibration $f\colon M\to B$, we have an isomorphism of analytic sheaves of abelian groups
    \begin{equation}\label{eq:calJ=AM}
   \Lambda\simeq \Lambda_M\simeq  \Gamma_{M}\  \text{ and }\  \calJ\simeq \calJ_{M/B}\simeq \calA_{M/B} .
    \end{equation}
    In particular, $\Sha^0(M/B)\simeq \rmH^1(B,\sJ)^0$, where the latter is as in \eqref{eq:exseqH1J0}.
\end{corollary}
\begin{proof}
From \cref{thm:comparisonJandJM} we have the isomorphisms $\Lambda\simeq \Lambda_M$ and $\sJ\simeq \sJ_{M/B}$.
Furthermore, $f\colon M\to B$ admits a relative theta divisor on the smooth fibers \cite[\S5]{LSV}. 
      Then \cref{thm:JMBisAMppHK} implies the result.
\end{proof}

\begin{remark}\label{rmk:MW}

     When $f$ has a section, Proposition \ref{prop:smlocus-torsor} allows us to move this section using elements of $ \rmH^0(B,\sA_{M/B})$. 
     Therefore, Corollary \ref{cor:JisAM} implies that the group of regular sections of $f$ has a subgroup isomorphic to $\rmH^0(B,\calJ)\simeq(\Sigma^{\perp})^{2,2}$ (Theorem \ref{thm:Sha-seq-nonDG}).
     In particular, when $f$ has integral fibers, this reproves the special case of  \cite[Theorem 5.1]{Sacbirational}.
 
\end{remark}

Recall that $\Br^4_{\an}(X)= \BrAn(\rmH^4(X,\bbZ)) = \rmH^{1,3}(X) / \image(\rmH^4(X, \Z) \to \rmH^4(X, \C) \to \rmH^{1,3}(X))$.

\begin{corollary}\label{cor:ShaseqSmoothCubic}
    Given a  defect general smooth cubic fourfold $X$, let $f\colon M\to B$ be an associated OG10 Lagrangian fibration as in Definition \ref{def:HKcompIJ}. Then $\Sha(M/B) \simeq \Sha^0(M/B)$ and there is an exact sequence
    \begin{equation}\label{eq:Sha-seqZd}
    0 \to \bbZ/d\bbZ \overset{\partial}\to \Sha(M/B) \to \Br^4_{\an}(X) \to 0,
    \end{equation}
    where $d$ is the divisibility of the complete intersection class $h^2$ in $\rmH^{2,2}(X,\bbZ)$.   
    Furthermore, 
 \[
 \Sha(M/B)\simeq \Br_{\an}(M).
 \]
\end{corollary}

\begin{proof}
    The exact sequence is just a combination of Corollary \ref{cor:JisAM} and Corollary \ref{cor:Sha-sequence}.
For the second isomorphism, we use \cref{cor:Br-M} applied to the Hodge structure $\rmH^4(X, \Z)_{\pr}$.~
\end{proof}

 \subsection{Degree twists and comparison with Voisin's construction}\label{sec:comparisonVoisin}

In this section, we focus on the short exact sequence in \cref{cor:ShaseqSmoothCubic}, in order to geometrically explain the degree twists, as defined below. 
%Recall that $\Sha(M/B) \simeq \rmH^1(B,\sJ)$

\begin{definition}\label{def:degreetwist}
Let $M\to B$ be an OG10 Lagrangian fibration associated to a defect general cubic fourfold $X$. 
    We define a \textit{degree twist} to be the Tate--Shafarevich twist of $M\to B$ given by
     any element of the image of $\bbZ/d\bbZ \to \rmH^1(B,\calJ) \simeq \Sha(M/B)$    
     from Corollary \ref{cor:ShaseqSmoothCubic}.
\end{definition}

We can give an explicit way to compute the cohomology class corresponding to a degree twist.
Namely, if
$b \in \rmH^4(X, \Z)$ is any class such that $\deg_X(b\cdot h^2) = 1$,
then by Lemma \ref{lem:Jm-cohomology}
the residue class
\begin{equation}\label{eq:def-t}
t = \ol{h^2 - 3b} \in \rmH^4(X, \Z/3)_{\pr} 
\end{equation}
gives rise to a generator of the degree twists subgroup.

\medskip

 In \cite{Voitwist}, Voisin studies a particular twist of the intermediate Jacobian fibration $M \to B$ when $X$ is general. A smooth fiber of this twisted fibration over a point $b=[H]\in B$ corresponding to a smooth hyperplane section $Y = X \cap H$
parametrizes $1$-cycles of degree $1$ in $Y$, i.e., cycles in the preimage of $1$ by $\CH^2(Y)\to \rmH^2(Y,\bbZ)\simeq \bbZ$ . 
This twist is constructed as a compactification of the corresponding torsor $[J^T_{U_1}]\in \rmH^1(U_1,\calJ_{U_1})$
where $U_1 \subset B$ is the open subset parameterizing smooth or $1$-nodal hyperplane sections.
The torsor $J_{U_1}^T$ can also be constructed when $X$ is, more generally, a defect general cubic fourfold. We will call any smooth hyperk\"ahler compactification $M^T\to B$ of $J_{U_1}^T\to U_1$, 
a \textit{Voisin twist}.
Such a compactification was constructed by Voisin \cite{Voitwist} when $X$ is general and by Sacc{\`a} \cite{Sacbirational} for any smooth cubic $X$. Note that, if $X$ is very general, 
 then
$M^T$ and $M$ are not birational \cite{Sacbirational}.

We now come to the final main result of the section.

 \begin{theorem}\label{thm:voisin=DMS}
          Let $X$ be defect general. 
          A Voisin twist $M^T$ is birational to a degree twist of $M$.
 \end{theorem}

  We begin with some lemmas
 which will distinguish degree twists as those twists which come from monodromy invariant cohomology classes and thus, exist over the moduli space of cubic fourfolds.
By monodromy invariance of a class,
 we mean that the class is restricted from the universal family of cubic fourfolds.
 
 \begin{lemma}\label{lem:h2-class-A}
 Let $X$ be a smooth cubic fourfold.
 For any coefficient ring $R$, 
 the monodromy invariant classes in $\rmH^4(X, R)$ are generated, as an $R$-module, by the complete intersection class $h^2$.
 \end{lemma}

 \begin{proof}
 Let us take a Lefschetz pencil $g\colon \XX \to \P^1$ with
 $X = \XX_0$.
 Specifically, $\XX$ is obtained by blowing up a smooth cubic fivefold $\ol{\XX} \subset \P^6$ containing $X$ with center in a smooth codimension two complete intersection $X \cap X'$ where $X'$ is another sufficiently general smooth cubic fourfold.
 We use the Leray spectral sequence for $\rmR g_*R_\XX$.
It is clear that
\[
\rmH^0(\rmR^4 f_* R_\XX)  = E_2^{0,4} = E_\infty^{0,4}
\]
because this term is not affected by differentials.
This means that we have a surjective restriction map
\[
\rmH^4(\ol{\XX}, R) \oplus 
\rmH^2(X \cap X', R) \simeq
\rmH^4(\XX, R) \to \rmH^0(\rmR^4 g_* R_\XX)
\]
where we used the blow up formula in the first isomorphism. Since both $\ol{\XX}$ and $X \cap X'$ are complete intersections, by the Weak Lefschetz theorem, we only get $R$-multiples of $h^2$ in the image. In particular, these are the only classes that can be induced from the universal family of cubic fourfolds.~\qedhere
\end{proof}

We will now work with the primitive cohomology of $X$ with coefficients in a ring $R$ defined by
\[
\rmH^4(X, R)_{\pr} = \Ker(h^2 \colon
\rmH^4(X, R) \to R).
\]
 It follows from Lemma \ref{lem:h2-class-A} that there are no 
 nonzero monodromy invariant classes in the primitive cohomology $\rmH^4(X, \Z)_{\rm pr}$. However the situation is different for $3$-torsion coefficients.
 
 \begin{lemma}\label{lem:degree-twists}
 Monodromy invariant classes
 with coefficients in $\Z/3$
 form a subgroup 
 \[
\Z/3 \cdot t \subset \rmH^4(X, \Z/3)_{\pr} \subset \rmH^4(X, \Z/3)
\]
with the class $t$ defined in \eqref{eq:def-t}.
 \end{lemma}

 \begin{proof}
 As $(h^2)^2 = 3$, the subgroup $\Z/3 \cdot h^2$ of monodromy invariant classes (see Lemma \ref{lem:h2-class-A}) is consisting
of primitive classes.
   The reduction $t = \ol{h^2 - 3b}$ is nonzero and independent of the choice of $b$.
   Hence it is a generator for the subgroup of monodromy invariant classes.
 \end{proof}

 \begin{proof}[Proof of \cref{thm:voisin=DMS}]
 By construction, see \cite[\S3]{Voitwist}, any
Voisin twist is birational to the Tate--Shafarevich  twist obtained from a
monodromy invariant class in $\rmH^4(X, \Z/3)_{\pr}$ by the composition
\begin{equation}\label{eq:Voisin-composition}
\rmH^4(X, \Z/3)_{\pr} \simeq 
\rmH^1(B, \JJ[3]) \to
\rmH^1(B, \JJ)[3] = 
\Sha(M/B)[3].
\end{equation}
By \cref{lem:degree-twists}, such classes are proportional to $t=\overline{h^2-3b}$.
 It follows from \cref{lem:Jm-cohomology}
 that the image of $t$ under
 \eqref{eq:Voisin-composition} generates subgroup of degree twists.

 \end{proof}

\appendix

\section{Comparison with Beauville--Mukai systems, after Markman}\label{app:comparisonBM}

The goal of this appendix is to compare our results with the work of Markman \cite[\S7]{MarkmanLagrangian}.
Let $S$ be a K3 surface, $B=|\calL|\simeq \bbP^g$ a base point free linear system on $S$ of genus $g \ge 2$. 
In order to apply Proposition \ref{prop:decomp-Z}, we assume that $|\calL|$ is primitive.

Denote by $U\subset B$ the locus parameterizing smooth curves.
We will now explain the analogs of $\Lambda$, $\Lambda^\bullet$ and other objects that we have introduced in \S \ref{sec:cohomologycomputations}, reusing the same notation.

We start by considering the universal curve $\calC\subset S\times B$ and we denote $p\colon \calC\to B$ and $q\colon \calC \to S$ the projections (see \eqref{eq:setupSM}). 
In particular, $q$ is a Zariski locally trivial $\bbP^{g-1}$-bundle. 
We let $\Lambda := \rmR^1 f_*\Z_\CC$ and also consider the sheaf $\sQ$ fitting in the short exact sequence
\[0\to \sQ\to \rmR^2 p_*\bbZ_{\sC}\overset{\deg}{\to} \bbZ_B\to 0.\]
The sheaf $\sQ$ is supported along the locus of reducible curves.

We have the following analog of Theorem \ref{thm:decomposition-cubic}.
\begin{theorem}\label{thm:decomposition-K3}
    Let $S$ be a  
    K3 surface and $B=|\calL|\simeq \bbP^g$ be a primitive base point free linear system on $S$ of genus $g \ge 2$.  
    Let $p\colon \sC \to B$
    be the universal family of curves.
    We have the  decomposition
\begin{equation}\label{eq:Rp*ZBM}
\rmR p_*\bbZ_\calC=\bbZ_B\oplus \Lambda^\bullet[-1] \oplus \bbZ_B[-2]
\end{equation}
    We have
    $\HH^0(\Lambda^\bullet) \simeq \Lambda$, $\HH^1(\Lambda^\bullet) \simeq \sQ$, and $\HH^i(\Lambda^\bullet) = 0$ for $i \neq 0, 1$. 
    Furthermore, $\Lambda$ is a tcf sheaf.
\end{theorem} 
\begin{proof}
    Since $\calL$ is primitive and $\rmH^2(S,\bbZ)$ is unimodular, there exists a class $\beta\in \rmH^2(S,\bbZ)$ such that $\beta\cdot c_1(\calL)=1$. In particular, the pullback $\xi = q^*(\beta) \in \rmH^2(\calC,\bbZ)$ restricts to $1\in \rmH^2(\calC_b,\bbZ)$ for all $b\in U$. 
    Therefore decomposition $\rmR p_* \Z_\CC$ follows from Example \ref{ex:intdecompcurve}.

    The last statement follows from Corollary \ref{thm:R1f-tcf-cor} since $p$ has no strictly multiple fibers because $\beta\cdot c_1(\LL) = 1$.
\end{proof}

\begin{remark}
    Note from Example \ref{ex:intdecompcurve} that if we assume in addition that the fibers of $p$ are integral in codimension one, then $\Lambda^\bullet[\dim B]\simeq {\rm IC}_B(\Lambda_{U})$ is an intersection complex.
\end{remark}

Denote by $\Sigma\subset \rmH^2(S,\bbZ)$, the sublattice generated by classes of all irreducible components of curves in $B$. 
We have the following analog of \cref{thm:cohom-Lambda} and \cref{lem:H1LambdaSigma}.

\begin{lemma}\label{lem:CohLambdaBMSystem}
We have
    \begin{equation}\label{eq:QcohomLambdaBM}
\rmH^k(B, \Lambda^\bullet) = 
    \begin{cases}
        \Z^{\oplus 21} & \text{if } k=1 \text{ or } k=2g-1,\\
        \Z^{\oplus 22} & \text{if } 1<k<2g-1, \ k \text{ odd,}\\
        0 &  \text{otherwise.}\\
    \end{cases}
\end{equation}
There is a canonical isomorphism of Hodge structures
$$\rmH^1(B,\Lambda)\simeq \Sigma^\perp\subset \rmH^2(S,\bbZ).$$
\end{lemma}

\begin{proof}
    Similarly to the proof of \cref{thm:cohom-Lambda}, from (\ref{eq:Rp*ZBM}) we get that $\rmH^{k-1}(B,\Lambda^\bullet)$ is the middle cohomology of the sequence
    \begin{equation}\label{eq:cohlambdaBMproof}
        \rmH^k(B,\bbZ) \xrightarrow{p^*} \rmH^k(\calC,\bbZ) \xrightarrow{p_*} \rmH^{k-2}(B,\bbZ).
    \end{equation}
    Since $q\colon \calC\to S$ is a $\bbP^{g-1}$-bundle,
    we can immediately compute $\rmH^k(\CC, \Z)$ and the middle cohomology of the sequence \eqref{eq:cohlambdaBMproof}.
    For example, for $k = 2$
    the kernel of the second map in (\ref{eq:cohlambdaBMproof}) is given by 
    \[
    \rmH^2(B,\bbZ)\oplus \Ker(\rmH^2(S,\bbZ) \xrightarrow{(-)\cdot c_1(\calL)} \bbZ)
    \]
    Thus, we get $\rmH^1(B,\Lambda^\bullet)\simeq c_1(\calL)^\perp\subset \rmH^2(S,\bbZ)$ which is isomorphic to $\Z^{\oplus 21}$. The first statement follows by a similar computation for $k > 1$.

    Consider the triangle $\Lambda \to \Lambda^\bullet \to \sQ[-1] \to \Lambda[1]$. 
    Using an argument similar to Lemma \ref{lem:H1LambdaSigma},
    we obtain that 
    $$\rmH^1(B,\Lambda) \simeq \Ker\left(\rmH^2(S,\bbZ)\to \rmH^0(B, \rmR^2p_*\bbZ_{\sC})\right).$$
    Therefore, the classes in $\rmH^1(B,\Lambda)$
    correspond to the classes in $\rmH^2(S,\bbZ)$ that restrict trivially on all components of all fibers. We get $\rmH^1(B,\Lambda)=\Sigma^\perp$. 
\end{proof}

Analogously to Definition \ref{def:wtJJ}, we consider the sheaf $\widetilde{\calJ}\coloneqq \rmR^2p_*\bbZ_{\calC}(1)_\rmD$. From 
\eqref{eq:Z2D-def} we have
\begin{equation}
    \widetilde{\calJ}\simeq \rmR^1p_*\calO_\calC^*,
\end{equation}
 i.e. $\wt{\sJ}$ is isomorphic to the relative Picard sheaf of $\calC$ over $B$.
The long exact sequence associated to the pushforward of the exponential sequence of $\calC$ gives an exact sequence
\begin{equation}\label{eq:4termseqBM}
   0 \to \rmR^1p_*\bbZ_{\calC} \to \rmR^1p_*\calO_\calC \to \widetilde{\calJ} \xrightarrow{\rm deg} \rmR^2p_*\bbZ_{\calC} \to 0.
\end{equation}
Set
$$\Lambda\coloneqq \rmR^1p_*\bbZ_{\calC} \text{ and } \calJ\coloneqq \calJ_{\calC/B} \simeq \Ker({\rm deg}).$$ 
Note that a computation similar to Proposition \ref{prop:pushforward-Omega1} (see \cite[Lemma 7.3]{MarkmanLagrangian}) shows that $\rmR^1p_*\calO_\calC\simeq \Omega^1_B$. 
The sequence (\ref{eq:4termseqBM}) splits into the two sequences
\begin{align}
    &0 \to \Lambda \to \Omega^1_B \to \calJ \to 0 \label{eq:LambdaOmegaJBM}\\
    &0 \to \calJ \to \widetilde{\calJ} \xrightarrow{\rm deg} \rmR^2p_*\bbZ_{\calC} \to 0. \label{eq:JJtildeZBM}
\end{align}
As observed in \cref{rmk:IJ-DelCoh}, \eqref{eq:JJtildeZBM} can be seen as an analog of the sequence 
\[
0 \to \Pic^0(C)\to \Pic(C)\to \rmH^2(C,\bbZ) \to 0
\]
for a family of curves.

\begin{citedprop}[(\cite{MarkmanLagrangian}, \S 7.1)]\label{prop:5termseqBM}
    There is a commutative diagram with exact rows and columns
    $$\begin{tikzcd}[column sep = 0.8cm]
    & & & 0 & 0 & \\
    & & & \rmH^2(B,\Lambda) \ar[u] \ar[r,"\simeq"] & \rmH^2(B,\Lambda) \ar[u] & \\
        0 \ar[r] & \rmH^0(B,\calJ) \ar[r]  & \rmH^0(B,\widetilde{\calJ}) \ar[r, "{\rmH^0({\rm deg})}"] & \rmH^0(B,\rmR^2p_*\bbZ_\calC) \ar[r] \ar[u] & \rmH^1(B,\calJ) \ar[u] \ar[r] & \rmH^1(B,\widetilde{\calJ}) \ar[r] & 0. \\
          0 \ar[r] & (\Sigma^\perp)^{1,1}  \ar[u,"\simeq"]\ar[r] & \NS(S) \ar[r]  \ar[u,"\simeq"]& \rmH^2(S,\bbZ)/\Sigma^{\perp} \ar[r] \ar[u] & \Br_{\an}(\Sigma^\perp) \ar[r] \ar[u] & \Br_{\rm an}(S) \ar[r] \ar[u,"\simeq"] & 0 \\
          & & & 0 \ar[u] & 0\ar[u] & 
    \end{tikzcd}$$
\end{citedprop}

\begin{proof}
For the definition of $\BrAn(\Sigma^\perp)$, see Appendix \ref{app:Brauer}. All the identifications come from computations with cohomology and Leray spectral sequences similar to \S\ref{sec:CohandPushforwards}.
\end{proof}

Proposition \ref{prop:5termseqBM} is a complete analog of Theorem \ref{thm:Sha-seq-nonDG}. In particular, denoting $\rmH^1(B,\calJ)^0\coloneqq \Ker (\rmH^1(B,\calJ) \to \rmH^2(B,\Lambda))$ we obtain the exact sequence
\begin{equation}\label{eq:Sha-sequenceBM}
    0 \to \rmH^2(S,\bbZ)/(\Sigma^\perp+\NS(S)) \to \rmH^1(B,\calJ)^0 \to \Br_{\rm an}(S) \to 0
\end{equation}
which is \cite[(7.7)]{MarkmanLagrangian}. It is also (\ref{eq:BrAn-sequence}) for the inclusion of Hodge lattice $\Sigma^\perp\subset \rmH^2(S,\bbZ)$.
\vspace{1\baselineskip}

 When smooth, the moduli space $M\coloneqq M_{S}(0,c_1(\calL),d+1-g)$
 of Gieseker stable  sheaves (with respect to a general polarization) on $S$ is a hyperk\"ahler manifold of K$3^{[g]}$-type. 
 It is equipped with a Lagrangian fibration $f\colon M\to B$, the fibers of which are compactified (degree $d$) Jacobians of the corresponding curves. 
 It is known as the Beauville--Mukai system \cite{MukaiBMsystem, BeaBMsystem}.

 Recall from \eqref{eq:exseq-expAM} the construction of the relative Albanese sheaf $\sA_{M/B}$. Analogously to Corollary \ref{cor:JisAM}, we have the following.

 \begin{theorem}\label{thm:PicisPicPicBM}
 Let $f\colon M\to B$ be a Beauville--Mukai system associated to a base point free primitive linear system $B=|\calL|$ of genus at least $2$ on a K3 surface. Assume that $f$ has no strictly multiple fibers
 in codimension $1$.
     Then we have
     $$\calJ = \JJ_{\CC/B} \simeq \calJ_{M/B}   \simeq \calA_{M/B}.$$
 \end{theorem}

 \begin{proof}
     Denote $\Lambda_M=\rmR^1f_*\bbZ_M$. 
     Note that the two vHs underlying the local systems $\Lambda_U=\rmR^1p_{U*}\bbZ_{\calC_U}$ and $\Lambda_{M_U}=\rmR^1f_{U*}\bbZ_{M_U}$ agree. 
    Since both $\Lambda$ and $\Lambda_M$ are tcf everywhere on $B$ by Corollary \ref{thm:R1f-tcf-cor}, they are isomorphic.
     The isomorphism $\calJ_{\calC/B}\simeq \calJ_{M/B}$ follows from Theorem \ref{thm:HMreljac}.

     It is known that all Beauville--Mukai systems admit a principal special K\"ahler class, see for instance \cite[Theorem 6.1]{WieneckPolType}. 
     Therefore, the isomorphism $\calJ_{M/B}\simeq \calA_{M/B}$ follows from Theorem \ref{thm:JMBisAMppHK}.
 \end{proof}

 \begin{remark}
    The assumption on the fibers of $f$ holds
for example when 
the locus of nonreduced 
curves in the linear system $|\calL|$
has codimension at least two \cite[Theorem 2.3]{MRVFineJacobians}.
We expect that the assumption on codimension $1$ fibers of $f$ follows from the assumptions on codimension $1$ curves in $|\calL|$ (see e.g.\ \cite{CKLRjacADE}).

\end{remark}
 
\begin{corollary}
    Under the assumptions of Theorem \ref{thm:PicisPicPicBM}, assume in addition that all curves parametrized by $B$ are integral. Then there is an exact sequence
     \begin{equation}\label{eq:Sha-seqBM-irred}
         0 \to \bbZ/d\bbZ \to \Sha(M/B) \to \Br_{\rm an}(S) \to 0
     \end{equation}
     where $d$ is the divisibility of $c_1(\calL)$ in $\NS(S)$.
\end{corollary}

\begin{proof}
Since the fibers are integral, $\Lambda^{\bullet} \simeq \Lambda$. 
Then the proof follows by an argument similar to \cref{cor:ShaseqSmoothCubic}.
\end{proof}

Note that (\ref{eq:Sha-seqBM-irred}) is an analogue of Corollary \ref{cor:Sha-sequence}. The same sequence was described by other means in \cite[Theorem 1.1]{HuyMat}.
For an interpretation of degree twists as moduli spaces, see \cite[Example 7.8]{MarkmanLagrangian}.

\section{Brauer groups of K3 type Hodge structures, after Huybrechts}
\label{app:Brauer}

 Let $H$ be an integral 
 Hodge structure of even weight $2k$.
 We say that $H$ is of K3 type, if it is a 
 finitely generated free abelian group with all 
 graded pieces of the Hodge filtration on $H_{\bbC}$ trivial, apart from
 \[
 H^{k-1,k+1}, H^{k,k}, H^{k+1,k-1} 
 \]
 and such that the pieces
 $H^{k-1,k+1}$ and $H^{k+1,k-1}$ are one-dimensional.
 This is a slight generalization of the usual definition when $k = 1$ is assumed.
By analogy with the case when $H$ is the middle cohomology of a complex K3 surface, we define 
\[
\NS(H) \coloneqq H \cap H^{k,k} = \Ker(H \to H \otimes \C \to H^{k-1,k+1}).
\]
We obtain an exact sequence
\begin{equation}\label{eq:NS-Br-seq}
0 \to \NS(H) \to H \to H^{k-1,k+1} \to \BrAn(H) \to 0.
\end{equation}
Note that all these groups do not change under Tate twisting of $H$.

When the Hodge structure $H=\rmH^{2k}(W,\bbZ)$ of a smooth projective variety $W$ is of K3 type, then we write $\Br_{\an}^{2k}(W)\coloneqq \BrAn(H)$.
In particular, if $k = 1$
(this applies if $W$ is a K3 surface),
 then 
 $\Br_{\an}(H)$ coincides with the analytic Brauer group $\Br^2_{\an}(W) \coloneqq \rmH^2_{\rm an}(W, \OO^*_W)$
 and \eqref{eq:NS-Br-seq} comes from the exponential exact sequence.

\begin{example}\label{ex:Br2-X}
If $X$ is a smooth complex cubic fourfold, then
$\rmH^4(X, \Z)$ is a Hodge structure of K3 type of weight $4$ 
and we can consider the analytic Brauer group
$\Br^4_{\an}(X) = \BrAn(\rmH^4(X,\bbZ))$.

Let $F(X)$ be the Fano variety of lines
of $X$. Beauville and Donagi proved that $\rmH^2(F(X), \Z)$ and $\rmH^4(X, \Z)(1)$
are isomorphic as Hodge structures (even though 
they are not isomorphic as Hodge lattices), see \cite[Proposition 4]{BeaDon}. We 
obtain
$\Br^2_{\an}(F(X)) \simeq \Br^4_{\an}(X)$.

\end{example}

By analogy with the geometric case
we also define
\begin{equation}\label{eq:def-BrtorsH}
\Br(H) \coloneqq
 (H / \NS(H)) \otimes \Q/\Z
\simeq
\BrAn(H)_\mathrm{tors}
\end{equation}
so
that for every $m \ge 1$ 
we have
a homomorphism
\begin{equation}\label{eq:projection-H-BrH}
\rho_m\colon H \otimes \Z/m \to \Br(H), \quad
\ol{a} \mapsto a/m
\end{equation}
which induces an isomorphism
$
\Br(H)[m] \simeq \Coker(\NS(H) \otimes \Z/m \to H \otimes \Z/m).
$

The following lemma says that the Brauer group typically becomes larger when passing to a Hodge substructure.

\begin{lemma}\label{lem:Bprim}
If $H \subset G$ is an embedding of Hodge structures of K3 type, then we have a canonical exact sequence
\begin{equation}\label{eq:BrAn-sequence}
0 \to G / (H + \NS(G)) \to \BrAn(H) \to \BrAn(G) \to 0.
\end{equation}
In particular, if $v\colon G\to \Z$
is a surjective group homomorphism and $H = \Ker(v)$, then the first term in \eqref{eq:BrAn-sequence} is the cyclic group $\Z / m$ with $m = v(\NS(G))$.
With $\rho_m$ as in \eqref{eq:projection-H-BrH}, the image of this subgroup in $\Br_{\an}(H)$
is generated by the class
\[
\rho_m(\ol{mg-s}) = (mg-s)/m  \in \Br(H) 
\subset \Br_{\an}(H)
\]
for any $g \in G$ with $v(g) = 1$ 
and $s\in \NS(G)$ with $v(s) = m$.
\end{lemma}
\begin{proof}
The first statement follows from the Snake Lemma applied to the definition of the Brauer group.
The second statement follows since in this case
$G/H \simeq v(G) = \Z$.

For the last claim, using \eqref{eq:def-BrtorsH} note that torsion in the sequence (\ref{eq:BrAn-sequence}) is obtained by applying $-\otimes \bbQ/\bbZ$ to
\begin{equation}\label{eq:preseqBrauer}
    0 \to H/\NS(H) \xrightarrow{j} G/\NS(G) \to G / (H + \NS(G)) \to 0.
\end{equation}
The last group in (\ref{eq:preseqBrauer}) is generated by the image of $\overline{g}$, and it is easy to check that the kernel of the surjection
$$\left(H/\NS(H)\right)\otimes \bbQ/\bbZ \twoheadrightarrow \left( G/\NS(G)\right)\otimes\bbQ/\bbZ$$
is generated by any element of the form $\rho_m(\overline{t})$, where $t\in H$ satisfies $j(\overline{t})=m\overline{g}$. One can choose $t=mg-s$, for $mg\equiv mg-s \pmod{\NS(G)}$.
\end{proof}

\begin{example}
Let $X$ be a defect general 
smooth cubic fourfold.
There are two natural hyperk\"ahler varieties
associated to $X$: the Fano variety $F(X)$
and the OG10 type Lagrangian fibration $M$
and their Brauer groups can be related
using Lemma \ref{lem:Bprim} as follows.

Set $G = \rmH^4(X, \Z)$ and take $v \colon G \to \Z$ to be cup product with $h^2$ so that $H = \rmH^4(X, \Z)_{\pr}$ and 
\[
v(G)/v(\NS(G)) \simeq \Z/\div(h^2).
\]
By Example \ref{ex:Br2-X} we have
$\Br_{\an}(G) \simeq \Br_{\an}^2(F(X))$
and by Corollary \ref{cor:Br-M}
$\Br_{\an}(H) \simeq \Br_{\an}^2(M)$.
Thus, Lemma \ref{lem:Bprim} gives a short exact sequence
\[0\to \bbZ/\div(h^2) \to \Br_{\an}^2(M)\to \Br_{\an}^2\left(F(X)\right)\to 0.\]
Combining Corollary \ref{cor:Sha-sequence} and \cref{cor:Br-M} we have a
different interpretation of this sequence.
\end{example}
    The sequence (\ref{eq:BrAn-sequence}) appears naturally in many geometric contexts. 
    See (\ref{eq:Sha-sequenceBM}) for $G=\rmH^2(S,\bbZ)$, where $S$ is a K3 surface.
For other examples of this phenomenon, see
\cite[Theorem 1.1]{HuyMat} and
\cite[Proposition 4.2]{MatteiMeinsma}.


\begin{thebibliography}{CKLR24}
\expandafter\ifx\csname url\endcsname\relax
  \def\url#1{\texttt{#1}}\fi
\expandafter\ifx\csname doi\endcsname\relax
  \def\doi#1{\burlalt{doi:#1}{http://dx.doi.org/#1}}\fi
\expandafter\ifx\csname urlprefix\endcsname\relax\def\urlprefix{URL }\fi
\expandafter\ifx\csname href\endcsname\relax
  \def\href#1#2{#2}\fi
\expandafter\ifx\csname burlalt\endcsname\relax
  \def\burlalt#1#2{\href{#2}{#1}}\fi

\bibitem[Aba24]{AbashevaSTII}
Anna Abasheva.
\newblock Shafarevich--{Tate} groups of holomorphic {Lagrangian} fibrations
  {II}, 2024, arXiv:~\burlalt{2407.09178v3}{http://arxiv.org/abs/2407.09178v3}.

\bibitem[ACLS25]{ACLSRelativeLefschetz}
Giuseppe Ancona, Mattia Cavicchi, Robert Laterveer, and Giulia Sacc\`a.
\newblock Relative and absolute {L}efschetz standard conjectures for some
  {L}agrangian fibrations.
\newblock {\em J. Lond. Math. Soc. (2)} {\bfseries 111}(4), 2025.
\newblock \doi{10.1112/jlms.70133}.

\bibitem[AF16]{AriFed}
Dima Arinkin and Roman Fedorov.
\newblock Partial {Fourier}-{Mukai} transform for integrable systems with
  applications to {Hitchin} fibration.
\newblock {\em Duke Math. J.} {\bfseries 165}(15), pp.\ 2991--3042, 2016.
\newblock \doi{10.1215/00127094-3645223}.

\bibitem[AF24]{AnconaFratila}
Giuseppe Ancona and Dragos Fratila.
\newblock Ng\^o support theorem and polarizability of quasi-projective
  commutative group schemes.
\newblock {\em \'Epijournal Geom. Alg\'ebrique} Volume 8, July 2024.
\newblock \doi{10.46298/epiga.2024.12345}.

\bibitem[AR21]{AbashevaRogov}
Anna Abasheva and Vasily Rogov.
\newblock Shafarevich--{Tate} groups of holomorphic {Lagrangian} fibrations,
  2021, arXiv:~\burlalt{2112.10921}{http://arxiv.org/abs/2112.10921}.

\bibitem[Ari13]{Arinkin-autoduality}
Dima Arinkin.
\newblock Autoduality of compactified {J}acobians for curves with plane
  singularities.
\newblock {\em J. Algebraic Geom.} {\bfseries 22}(2), pp.\ 363--388, 2013.
\newblock \doi{10.1090/S1056-3911-2012-00596-7}.

\bibitem[BBD82]{BBDG}
Alexander~A. Beilinson, Joseph Bernstein, and Pierre Deligne.
\newblock Faisceaux pervers.
\newblock In {\em Analysis and topology on singular spaces, {I} ({L}uminy,
  1981)}, volume 100 of {\em Ast\'erisque}, pp. 5--171. Soc. Math. France,
  Paris, 1982.

\bibitem[BD85]{BeaDon}
Arnaud Beauville and Ron Donagi.
\newblock La vari\'et\'e{} des droites d'une hypersurface cubique de dimension
  {$4$}.
\newblock {\em C. R. Acad. Sci. Paris S\'er. I Math.} {\bfseries 301}(14), pp.\
  703--706, 1985.

\bibitem[Bea91]{BeaBMsystem}
Arnaud Beauville.
\newblock Syst\`emes hamiltoniens compl\`etement int\'egrables associ\'es aux
  surfaces {$K3$}.
\newblock In {\em Problems in the theory of surfaces and their classification
  ({C}ortona, 1988)}, volume XXXII of {\em Sympos. Math.}, pp. 25--31. Academic
  Press, London, 1991.

\bibitem[Bea11]{Beauville-HK-survey}
Arnaud Beauville.
\newblock Holomorphic symplectic geometry: a problem list.
\newblock In {\em Complex and differential geometry}, volume~8 of {\em Springer
  Proc. Math.}, pp. 49--63. Springer, Heidelberg, 2011.
\newblock \doi{10.1007/978-3-642-20300-8\_2}.

\bibitem[Bei16]{Bei15}
Alexander Beilinson.
\newblock Constructible sheaves are holonomic.
\newblock {\em Sel. Math., New Ser.} {\bfseries 22}(4), pp.\ 1797--1819, 2016.
\newblock \doi{10.1007/s00029-016-0260-z}.

\bibitem[BFNP09]{BFNP}
Patrick Brosnan, Hao Fang, Zhaohu Nie, and Gregory Pearlstein.
\newblock Singularities of admissible normal functions (with an appendix by
  {Najmuddin} {Fakhruddin}).
\newblock {\em Invent. Math.} {\bfseries 177}(3), pp.\ 599--629, 2009.
\newblock \doi{10.1007/s00222-009-0191-9}.

\bibitem[Bot25]{bottini}
Alessio Bottini.
\newblock {O}'{G}rady's tenfolds from stable bundles on hyper-k\"ahler
  fourfolds, 2025,
  arXiv:~\burlalt{2411.18528}{http://arxiv.org/abs/2411.18528}.

\bibitem[BPS08]{BrosnanPearlsteinSaito}
Patrick Brosnan, Gregory Pearlstein, and Morihiko Saito.
\newblock A generalization of the {N}eron models of {G}reen, {G}riffiths and
  {K}err, 2008, arXiv:~\burlalt{0809.5185}{http://arxiv.org/abs/0809.5185}.

\bibitem[BW79]{BruceWall}
James~W. Bruce and Charles T.~C. Wall.
\newblock On the classification of cubic surfaces.
\newblock {\em J. London Math. Soc. (2)} {\bfseries 19}(2), pp.\ 245--256,
  1979.
\newblock \doi{10.1112/jlms/s2-19.2.245}.

\bibitem[Cam05]{CampanaMultipleFibers}
Fr{\'e}d{\'e}ric Campana.
\newblock Multiple fibres on surfaces: geometry, hyperbolic and arithmetic
  aspects.
\newblock {\em Manuscr. Math.} {\bfseries 117}(4), pp.\ 429--461, 2005.
\newblock \doi{10.1007/s00229-005-0570-5}.

\bibitem[Cam06]{Camisotrivial}
Fr\'ed\'eric Campana.
\newblock Isotrivialit\'e{} de certaines familles k\"ahl\'eriennes de
  vari\'et\'es non projectives.
\newblock {\em Math. Z.} {\bfseries 252}(1), pp.\ 147--156, 2006.
\newblock \doi{10.1007/s00209-005-0851-4}.

\bibitem[CG72]{ClemensGriffiths}
C.~Herbert Clemens and Phillip~A. Griffiths.
\newblock The intermediate {J}acobian of the cubic threefold.
\newblock {\em Ann. of Math. (2)} 95, pp.\ 281--356, 1972.
\newblock \doi{10.2307/1970801}.

\bibitem[CKLR24]{CKLRjacADE}
Adam Czapli\'nski, Andreas Krug, Manfred Lehn, and S\"onke Rollenske.
\newblock Compactified {J}acobians of extended {ADE} curves and {L}agrangian
  fibrations.
\newblock {\em Commun. Contemp. Math.} {\bfseries 26}(10), pp.\ Paper No.
  2450004, 46, 2024.
\newblock \doi{10.1142/S0219199724500044}.

\bibitem[Cle83]{Clem83}
Herbert Clemens.
\newblock The {N}{\'e}ron model for families of intermediate {Jacobians}
  acquiring ''algebraic'' singularities.
\newblock {\em Publ. Math., Inst. Hautes {\'E}tud. Sci.} 58, pp.\ 217--230,
  1983.

\bibitem[CML09]{CML}
Sebastian Casalaina-Martin and Radu Laza.
\newblock The moduli space of cubic threefolds via degenerations of the
  intermediate {J}acobian.
\newblock {\em J. Reine Angew. Math.} 633, pp.\ 29--65, 2009.
\newblock \doi{10.1515/CRELLE.2009.059}.

\bibitem[dCM05]{dCM-HT}
Mark Andrea~A. de~Cataldo and Luca Migliorini.
\newblock The {H}odge theory of algebraic maps.
\newblock {\em Ann. Sci. \'{E}cole Norm. Sup. (4)} {\bfseries 38}(5), pp.\
  693--750, 2005.
\newblock \doi{10.1016/j.ansens.2005.07.001}.

\bibitem[dCRS21]{dCRS_HodgeNumberOG10}
Mark Andrea~A. de~Cataldo, Antonio Rapagnetta, and Giulia Sacc\`a.
\newblock The {H}odge numbers of {O}'{G}rady 10 via {N}g\^o{} strings.
\newblock {\em J. Math. Pures Appl. (9)} 156, pp.\ 125--178, 2021.
\newblock \doi{10.1016/j.matpur.2021.10.004}.

\bibitem[Dim92]{Dimca-book}
Alexandru Dimca.
\newblock {\em Singularities and topology of hypersurfaces}.
\newblock Universitext. Springer-Verlag, New York, 1992.
\newblock \doi{10.1007/978-1-4612-4404-2}.

\bibitem[Dim04]{Dimca-sheaves}
Alexandru Dimca.
\newblock {\em Sheaves in topology}.
\newblock Universitext. Springer-Verlag, Berlin, 2004.
\newblock \doi{10.1007/978-3-642-18868-8}.

\bibitem[DM96]{DonMar}
Ron Donagi and Eyal Markman.
\newblock {\em Spectral covers, algebraically completely integrable,
  hamiltonian systems, and moduli of bundles}, pp. 1--119.
\newblock Springer, Berlin, Heidelberg, 1996.
\newblock \doi{10.1007/BFb0094792}.

\bibitem[DM25]{DutMarq}
Yajnaseni Dutta and Lisa Marquand.
\newblock Relative compactified {P}rym and {P}icard fibrations associated to
  very good cubic fourfolds, 2025,
  arXiv:~\burlalt{2502.21301}{http://arxiv.org/abs/2502.21301}.

\bibitem[EV88]{EsnaultViehweg}
H\'{e}l\`ene Esnault and Eckart Viehweg.
\newblock Deligne-{B}e\u{\i}linson cohomology.
\newblock In {\em Be\u{\i}linson's conjectures on special values of
  {$L$}-functions}, volume~4 of {\em Perspect. Math.}, pp. 43--91. Academic
  Press, Boston, MA, 1988.

\bibitem[GGK10]{GreenGriffithsKerr}
Mark Green, Phillip Griffiths, and Matt Kerr.
\newblock N\'{e}ron models and limits of {A}bel-{J}acobi mappings.
\newblock {\em Compos. Math.} {\bfseries 146}(2), pp.\ 288--366, 2010.
\newblock \doi{10.1112/S0010437X09004400}.

\bibitem[GGO24]{Gx2Onorati}
Franco Giovenzana, Luca Giovenzana, and Claudio Onorati.
\newblock On the period of {L}i, {P}ertusi, and {Z}hao's symplectic variety.
\newblock {\em Canad. J. Math.} {\bfseries 76}(4), pp.\ 1432--1453, 2024.
\newblock \doi{10.4153/S0008414X23000470}.

\bibitem[GM83]{GoreskyMacphersonII}
Mark Goresky and Robert MacPherson.
\newblock Intersection homology. {II}.
\newblock {\em Invent. Math.} {\bfseries 72}(1), pp.\ 77--129, 1983.
\newblock \doi{10.1007/BF01389130}.

\bibitem[HM23]{HuyMat}
Daniel Huybrechts and Dominique Mattei.
\newblock The special {B}rauer group and twisted {P}icard varieties, 2023,
  arXiv:~\burlalt{2310.04032}{http://arxiv.org/abs/2310.04032}.

\bibitem[HO09]{HOCharFoliation}
Jun-Muk Hwang and Keiji Oguiso.
\newblock Characteristic foliation on the discriminant hypersurface of a
  holomorphic {L}agrangian fibration.
\newblock {\em Amer. J. Math.} {\bfseries 131}(4), pp.\ 981--1007, 2009.
\newblock \doi{10.1353/ajm.0.0062}.

\bibitem[HTT08]{HTT08}
Ryoshi Hotta, Kiyoshi Takeuchi, and Toshiyuki Tanisaki.
\newblock {\em {$D$}-modules, perverse sheaves, and representation theory},
  volume 236 of {\em Progress in Mathematics}.
\newblock Birkh\"{a}user Boston, Inc., Boston, MA, 2008.
\newblock \doi{10.1007/978-0-8176-4523-6}.
\newblock Translated from the 1995 Japanese edition by Takeuchi.

\bibitem[{Huy}99]{HuycptHKbasic}
Daniel {Huybrechts}.
\newblock {Compact hyperk\"ahler manifolds: Basic results}.
\newblock {\em {Invent. Math.}} {\bfseries 135}(1), pp.\ 63--113, 1999.

\bibitem[Huy16]{HuybrechtsK3}
Daniel Huybrechts.
\newblock {\em Lectures on {K}3 surfaces}, volume 158 of {\em Cambridge Studies
  in Advanced Mathematics}.
\newblock Cambridge University Press, Cambridge, 2016.
\newblock \doi{10.1017/CBO9781316594193}.

\bibitem[Huy23a]{Huybrechts-Br}
Daniel Huybrechts.
\newblock Brilliant families of {K}3 surfaces: twistor spaces, {B}rauer groups,
  and {N}oether-{L}efschetz loci.
\newblock {\em Ann. Fac. Sci. Toulouse Math. (6)} {\bfseries 32}(2), pp.\
  397--421, 2023.

\bibitem[Huy23b]{Huycubic}
Daniel Huybrechts.
\newblock {\em The geometry of cubic hypersurfaces}, volume 206 of {\em Camb.
  Stud. Adv. Math.}
\newblock Cambridge: Cambridge University Press, 2023.
\newblock \doi{10.1017/9781009280020}.

\bibitem[Hwa08]{HwangBase}
Jun-Muk Hwang.
\newblock Base manifolds for fibrations of projective irreducible symplectic
  manifolds.
\newblock {\em Invent. Math.} {\bfseries 174}(3), pp.\ 625--644, 2008.
\newblock \doi{10.1007/s00222-008-0143-9}.

\bibitem[Kim24]{Yoonjoo}
Yoon-Joo Kim.
\newblock The {N}\'eron model of a higher-dimensional {L}agrangian fibration,
  2024, arXiv:~\burlalt{2410.21193}{http://arxiv.org/abs/2410.21193}.

\bibitem[KLSV18]{KLSV}
J\'{a}nos Koll\'{a}r, Radu Laza, Giulia Sacc\`a, and Claire Voisin.
\newblock Remarks on degenerations of hyper-{K}\"{a}hler manifolds.
\newblock {\em Ann. Inst. Fourier (Grenoble)} {\bfseries 68}(7), pp.\
  2837--2882, 2018.
\newblock \doi{10.5802/aif.3228}.

\bibitem[Kol86]{KoldualI}
J\'{a}nos Koll\'{a}r.
\newblock Higher direct images of dualizing sheaves. {I}.
\newblock {\em Ann. of Math. (2)} {\bfseries 123}(1), pp.\ 11--42, 1986.
\newblock \doi{10.2307/1971351}.

\bibitem[Kol13]{KK13}
J\'{a}nos Koll\'{a}r.
\newblock {\em Singularities of the minimal model program}, volume 200 of {\em
  Cambridge Tracts in Mathematics}.
\newblock Cambridge University Press, Cambridge, 2013.
\newblock \doi{10.1017/CBO9781139547895}.
\newblock With a collaboration of S\'{a}ndor Kov\'{a}cs.

\bibitem[LPZ22]{LPZ20}
Chunyi Li, Laura Pertusi, and Xiaolei Zhao.
\newblock Elliptic quintics on cubic fourfolds, {O}'{G}rady 10, and
  {L}agrangian fibrations.
\newblock {\em Advances in Mathematics} 408, pp.\ 108584, 2022.
\newblock \doi{https://doi.org/10.1016/j.aim.2022.108584}.

\bibitem[LSV17]{LSV}
Radu Laza, Giulia Sacc{\`a}, and Claire Voisin.
\newblock A hyper{K{\"a}hler} compactification of the intermediate {Jacobian}
  fibration associated with a cubic 4-fold.
\newblock {\em Acta Math.} {\bfseries 218}(1), pp.\ 55--135, 2017.
\newblock \doi{10.4310/ACTA.2017.v218.n1.a2}.

\bibitem[Mar96]{Markushevich}
Dimitri Markushevich.
\newblock Lagrangian families of {J}acobians of genus {$2$} curves.
\newblock {\em J. Math. Sci.} 82, pp.\ 3268--3284, 1996.
\newblock \doi{10.1007/BF02362472}.
\newblock Algebraic geometry, 5.

\bibitem[Mar14]{MarkmanLagrangian}
Eyal Markman.
\newblock Lagrangian fibrations of holomorphic-symplectic varieties of
  {$K3^{[n]}$}-type.
\newblock In {\em Algebraic and complex geometry}, volume~71 of {\em Springer
  Proc. Math. Stat.}, pp. 241--283. Springer, Cham, 2014.
\newblock \doi{10.1007/978-3-319-05404-9\_10}.

\bibitem[Mat99]{Mats1}
Daisuke Matsushita.
\newblock On fibre space structures of a projective irreducible symplectic
  manifold.
\newblock {\em Topology} {\bfseries 38}(1), pp.\ 79--83, 1999.
\newblock \doi{10.1016/S0040-9383(98)00003-2}.

\bibitem[Mat01]{Mats2}
Daisuke Matsushita.
\newblock Addendum: ``{O}n fibre space structures of a projective irreducible
  symplectic manifold'' [{T}opology {\bf 38} (1999), no. 1, 79--83; {MR}1644091
  (99f:14054)].
\newblock {\em Topology} {\bfseries 40}(2), pp.\ 431--432, 2001.
\newblock \doi{10.1016/S0040-9383(99)00048-8}.

\bibitem[Mat05]{Matdirect}
Daisuke Matsushita.
\newblock Higher direct images of dualizing sheaves of {Lagrangian} fibrations.
\newblock {\em Am. J. Math.} {\bfseries 127}(2), pp.\ 243--259, 2005.
\newblock \doi{10.1353/ajm.2005.0009}.

\bibitem[Mat16]{Matsushita-deformations}
Daisuke Matsushita.
\newblock On deformations of {L}agrangian fibrations.
\newblock In {\em K3 surfaces and their moduli}, volume 315 of {\em Progr.
  Math.}, pp. 237--243. Birkh\"{a}user/Springer, [Cham], 2016.
\newblock \doi{10.1007/978-3-319-29959-4\_9}.

\bibitem[MM24]{MatteiMeinsma}
Dominique Mattei and Reinder Meinsma.
\newblock Obstruction classes for moduli spaces of sheaves and {Lagrangian}
  fibrations, 2024,
  arXiv:~\burlalt{2404.16652}{http://arxiv.org/abs/2404.16652}.

\bibitem[MO22]{MongardiOnorati}
Giovanni Mongardi and Claudio Onorati.
\newblock Birational geometry of irreducible holomorphic symplectic tenfolds of
  {O}'{G}rady type.
\newblock {\em Math. Z.} {\bfseries 300}(4), pp.\ 3497--3526, 2022.
\newblock \doi{10.1007/s00209-021-02966-6}.

\bibitem[MO24]{MongardiOnorati-erratum}
Giovanni Mongardi and Claudio Onorati.
\newblock Correction: {B}irational geometry of irreducible holomorphic
  symplectic tenfolds of {O}'{G}rady type.
\newblock {\em Math. Z.} {\bfseries 308}(3), pp.\ Paper No. 54, 6, 2024.
\newblock \doi{10.1007/s00209-024-03588-4}.

\bibitem[MRV17]{MRVFineJacobians}
Margarida Melo, Antonio Rapagnetta, and Filippo Viviani.
\newblock Fine compactified {Jacobians} of reduced curves.
\newblock {\em Trans. Am. Math. Soc.} {\bfseries 369}(8), pp.\ 5341--5402,
  2017.
\newblock \doi{10.1090/tran/6823}.

\bibitem[MSY23]{MSY23}
Davesh Maulik, Junliang Shen, and Qizheng Yin.
\newblock Fourier-{Mukai} transforms and the decomposition theorem for
  integrable systems, 2023,
  arXiv:~\burlalt{2301.05825}{http://arxiv.org/abs/2301.05825}.

\bibitem[Muk87]{MukaiBMsystem}
S.~Mukai.
\newblock On the moduli space of bundles on {$K3$} surfaces. {I}.
\newblock In {\em Vector bundles on algebraic varieties ({B}ombay, 1984)},
  volume~11 of {\em Tata Inst. Fund. Res. Stud. Math.}, pp. 341--413. Tata
  Inst. Fund. Res., Bombay, 1987.

\bibitem[Mur74]{Murre}
Jaap~P. Murre.
\newblock Some results on cubic threefolds.
\newblock In {\em Classification of algebraic varieties and compact complex
  manifolds}, volume Vol. 412 of {\em Lecture Notes in Math.}, pp. 140--160.
  Springer, Berlin-New York, 1974.

\bibitem[MV24]{MarqVikt}
Lisa Marquand and Sasha Viktorova.
\newblock The defect of a cubic threefold, 2024,
  arXiv:~\burlalt{2312.05118}{http://arxiv.org/abs/2312.05118}.

\bibitem[Ng{\^o}10]{Ngo}
B{\'a}o~Ch{\^a}u Ng{\^o}.
\newblock The fundamental lemma for {Lie} algebras.
\newblock {\em Publ. Math., Inst. Hautes {\'E}tud. Sci.} 111, pp.\ 1--169,
  2010.
\newblock \doi{10.1007/s10240-010-0026-7}.

\bibitem[NS95]{NamikawaSteenbrink}
Yoshinori Namikawa and J.~H.~M. Steenbrink.
\newblock Global smoothing of {C}alabi-{Y}au threefolds.
\newblock {\em Invent. Math.} {\bfseries 122}(2), pp.\ 403--419, 1995.
\newblock \doi{10.1007/BF01231450}.

\bibitem[Pop18]{PopaHM}
Mihnea Popa.
\newblock Positivity for {H}odge modules and geometric applications.
\newblock In {\em Algebraic geometry: {S}alt {L}ake {C}ity 2015}, volume 97.1
  of {\em Proc. Sympos. Pure Math.}, pp. 555--584. Amer. Math. Soc.,
  Providence, RI, 2018.
\newblock \doi{10.1090/pspum/097.1/01685}.

\bibitem[Rap08]{Rapagnetta-OG10}
Antonio Rapagnetta.
\newblock On the {B}eauville form of the known irreducible symplectic
  varieties.
\newblock {\em Math. Ann.} {\bfseries 340}(1), pp.\ 77--95, 2008.
\newblock \doi{10.1007/s00208-007-0139-6}.

\bibitem[Rei83]{Reid-models}
Miles Reid.
\newblock Minimal models of canonical {$3$}-folds.
\newblock In {\em Algebraic varieties and analytic varieties ({T}okyo, 1981)},
  volume~1 of {\em Adv. Stud. Pure Math.}, pp. 131--180. Amsterdam, 1983.
\newblock \doi{10.2969/aspm/00110131}.

\bibitem[Rei87]{Reid-young}
Miles Reid.
\newblock Young person's guide to canonical singularities.
\newblock In {\em Algebraic geometry, {B}owdoin, 1985 ({B}runswick, {M}aine,
  1985)}, volume~46 of {\em Proc. Sympos. Pure Math.}, pp. 345--414. Amer.
  Math. Soc., Providence, RI, 1987.
\newblock \doi{10.1090/pspum/046.1/927963}.

\bibitem[Sac23]{Sacbirational}
Giulia Sacc{\`a}.
\newblock Birational geometry of the intermediate {Jacobian} fibration of a
  cubic fourfold (appendix by {Claire} {Voisin}).
\newblock {\em Geom. Topol.} {\bfseries 27}(4), pp.\ 1479--1538, 2023.
\newblock \doi{10.2140/gt.2023.27.1479}.

\bibitem[Sac24]{Sac25}
Giulia Saccà.
\newblock Compactifying {L}agrangian fibrations, 2024,
  arXiv:~\burlalt{2411.06505}{http://arxiv.org/abs/2411.06505}.

\bibitem[Sai88]{Sai88}
Morihiko Saito.
\newblock Modules de {Hodge} polarisables. ({Polarisable} {Hodge} modules).
\newblock {\em Publ. Res. Inst. Math. Sci.} {\bfseries 24}(6), pp.\ 849--995,
  1988.
\newblock \doi{10.2977/prims/1195173930}.

\bibitem[Sai91]{SaiKolConj}
Morihiko Saito.
\newblock On {K}oll\'ar's conjecture.
\newblock In {\em Several complex variables and complex geometry, {P}art 2
  ({S}anta {C}ruz, {CA}, 1989)}, volume 52, Part 2 of {\em Proc. Sympos. Pure
  Math.}, pp. 509--517. Amer. Math. Soc., Providence, RI, 1991.
\newblock \doi{10.1090/pspum/052.2/1128566}.

\bibitem[Sai96]{Saiadmissible}
Morihiko Saito.
\newblock Admissible normal functions.
\newblock {\em J. Algebr. Geom.} {\bfseries 5}(2), pp.\ 235--276, 1996.

\bibitem[Sch12a]{SchNeron}
Christian Schnell.
\newblock Complex analytic {N{\'e}ron} models for arbitrary families of
  intermediate {Jacobians}.
\newblock {\em Invent. Math.} {\bfseries 188}(1), pp.\ 1--81, 2012.
\newblock \doi{10.1007/s00222-011-0341-8}.

\bibitem[Sch12b]{Schresidues}
Christian Schnell.
\newblock Residues and filtered {{\(D\)}}-modules.
\newblock {\em Math. Ann.} {\bfseries 354}(2), pp.\ 727--763, 2012.
\newblock \doi{10.1007/s00208-011-0746-0}.

\bibitem[Sch14]{Schoverview}
Christian Schnell.
\newblock An overview of {M}orihiko {S}aito's theory of mixed {H}odge modules,
  2014, arXiv:~\burlalt{1405.3096}{http://arxiv.org/abs/1405.3096}.

\bibitem[Sch16]{Schvanish}
Christian Schnell.
\newblock On {Saito}'s vanishing theorem.
\newblock {\em Math. Res. Lett.} {\bfseries 23}(2), pp.\ 499--527, 2016.
\newblock \doi{10.4310/MRL.2016.v23.n2.a10}.

\bibitem[Sch23]{schHK}
Christian Schnell.
\newblock Hodge theory and {L}agrangian fibrations on holomorphic symplectic
  manifolds, 2023,
  arXiv:~\burlalt{2303.05364}{http://arxiv.org/abs/2303.05364}.

\bibitem[Shi90]{Shioda}
Tetsuji Shioda.
\newblock On the {M}ordell-{W}eil lattices.
\newblock {\em Comment. Math. Univ. St. Paul.} {\bfseries 39}(2), pp.\
  211--240, 1990.

\bibitem[SS18]{SabSch}
Claude Sabbah and Christian Schnell.
\newblock The mhm project, 2018.
\newblock
  \urlprefix\url{https://perso.pages.math.cnrs.fr/users/claude.sabbah/MHMProject/mhm.html}.

\bibitem[SV24]{VerbitskySoldatenkovKaehlertwists}
Andrey Soldatenkov and Misha Verbitsky.
\newblock Hermitian-symplectic and {K}\"aler structures on degenerate twistor
  deformations, 2024,
  arXiv:~\burlalt{2407.07867}{http://arxiv.org/abs/2407.07867}.

\bibitem[SY22]{ShenYin-PW}
Junliang Shen and Qizheng Yin.
\newblock Topology of {L}agrangian fibrations and {H}odge theory of
  hyper-{K}\"{a}hler manifolds.
\newblock {\em Duke Math. J.} {\bfseries 171}(1), pp.\ 209--241, 2022.
\newblock \doi{10.1215/00127094-2021-0010}.
\newblock With Appendix B by Claire Voisin.

\bibitem[SY23]{ShenYin}
Junliang Shen and Qizheng Yin.
\newblock Perverse-{H}odge complexes for {L}agrangian fibrations.
\newblock {\em \'{E}pijournal G\'{e}om. Alg\'{e}brique.} Special volume in
  honour of Claire Voisin, pp.\ Art. 6, 20, 2023.

\bibitem[Ver76]{Verdier-cycleclass}
Jean-Louis Verdier.
\newblock Classe d'homologie associ\'{e}e \`a un cycle.
\newblock In {\em S\'{e}minaire de g\'{e}om\'{e}trie analytique (\'{E}cole
  {N}orm. {S}up., {P}aris, 1974--75)}, Ast\'{e}risque, No. 36-37, pp. 101--151.
  Soc. Math. France, Paris, 1976.

\bibitem[Ver15]{VerbitskyDegenerateTwistor}
Misha Verbitsky.
\newblock Degenerate twistor spaces for hyperk{\"a}hler manifolds.
\newblock {\em J. Geom. Phys.} 91, pp.\ 2--11, 2015.
\newblock \doi{10.1016/j.geomphys.2014.08.010}.

\bibitem[Voi86]{VoisinTorelliCubic}
Claire Voisin.
\newblock Th\'eor\`eme de {T}orelli pour les cubiques de {${\bf P}^5$}.
\newblock {\em Invent. Math.} {\bfseries 86}(3), pp.\ 577--601, 1986.
\newblock \doi{10.1007/BF01389270}.

\bibitem[Voi02]{VoiHTone}
Claire Voisin.
\newblock {\em Hodge theory and complex algebraic geometry. {I}. {Translated}
  from the {French} by {Leila} {Schneps}}, volume~76 of {\em Camb. Stud. Adv.
  Math.}
\newblock Cambridge: Cambridge University Press, 2002.

\bibitem[Voi07]{VoiHTtwo}
Claire Voisin.
\newblock {\em Hodge theory and complex algebraic geometry. {II}. {Transl}.
  from the {French} by {Leila} {Schneps}}, volume~77 of {\em Camb. Stud. Adv.
  Math.}
\newblock Cambridge: Cambridge University Press, 2007.

\bibitem[Voi08]{VoisinTorelliCubicErratum}
Claire Voisin.
\newblock Erratum: ``{A} {T}orelli theorem for cubics in {$\mathbb{P}^5$}''
  ({F}rench) [{I}nvent. {M}ath. {\bf 86} (1986), no. 3, 577--601; mr0860684].
\newblock {\em Invent. Math.} {\bfseries 172}(2), pp.\ 455--458, 2008.
\newblock \doi{10.1007/s00222-008-0116-z}.

\bibitem[Voi18]{Voitwist}
Claire Voisin.
\newblock Hyper-{K{\"a}hler} compactification of the intermediate {Jacobian}
  fibration of a cubic fourfold: the twisted case.
\newblock In {\em Local and global methods in algebraic geometry. Conference in
  honor of Lawrence Ein's 60th birthday, UIC, IL, USA. Proceedings}, pp.
  341--355. Amer. Math. Soc., Providence, RI, 2018.
\newblock \doi{10.1090/conm/712/14354}.

\bibitem[Voi24]{Voisin-ppav}
Claire Voisin.
\newblock Cycle classes on abelian varieties and the geometry of the
  {A}bel-{J}acobi map.
\newblock {\em Pure Appl. Math. Q.} {\bfseries 20}(5), pp.\ 2469--2496, 2024.
\newblock \doi{10.4310/pamq.241105235723}.

\bibitem[Wie16]{WieneckPolType}
Benjamin Wieneck.
\newblock On polarization types of {L}agrangian fibrations.
\newblock {\em Manuscripta Math.} {\bfseries 151}(3-4), pp.\ 305--327, 2016.
\newblock \doi{10.1007/s00229-016-0845-z}.

\bibitem[Zuc76]{Zuc76}
Steven Zucker.
\newblock Generalized intermediate {Jacobians} and the theorem on normal
  functions.
\newblock {\em Invent. Math.} 33, pp.\ 185--222, 1976.
\newblock \doi{10.1007/BF01404203}.

\end{thebibliography}
\end{document}